\pdfoutput=1

\documentclass{amsart}
\usepackage{etex}

\usepackage[mathscr]{eucal}
\usepackage{amssymb}
\usepackage[usenames,dvipsnames]{xcolor} 
\usepackage[normalem]{ulem}
\usepackage[section]{placeins}
\usepackage{amsthm}
\usepackage{bbold}
\usepackage{stackrel}
\usepackage{enumerate}
\usepackage{array}
\usepackage{amsmath}
\usepackage{mathtools}

\usepackage{hyperref}

\usepackage{rotating} 

\hypersetup{colorlinks=true,citecolor=BrickRed,urlcolor=BrickRed,linkcolor=BrickRed}

\usepackage{caption}
\usepackage{subcaption}

\numberwithin{equation}{section}
\usepackage[all]{xy}

\newdir{ >}{{}*!/-10pt/\dir{>}}

\usepackage{tikz}
\usetikzlibrary{calc}
\usetikzlibrary{snakes}
\usetikzlibrary{patterns}
\usetikzlibrary{positioning}
\usepackage{tikz-cd}

\tikzset{cong/.style={draw=none,edge node={node [sloped, allow upside down, auto=false]{$\cong$}}},
         Isom/.style={draw=none,every to/.append style={edge node={node [sloped, allow upside down, auto=false]{$\cong$}}}}}

\swapnumbers 

\newtheorem{Thm}[equation]{Theorem}
\newtheorem*{Thm*}{Theorem}
\newtheorem{Prop}[equation]{Proposition}
\newtheorem{Lem}[equation]{Lemma}
\newtheorem{Cor}[equation]{Corollary}
\theoremstyle{remark}
\newtheorem{Def}[equation]{Definition}
\newtheorem{Ter}[equation]{Terminology}
\newtheorem{Not}[equation]{Notation}
\newtheorem{Exa}[equation]{Example}
\newtheorem{Exas}[equation]{Examples}

\newtheorem{Hyp}[equation]{Hypothesis}

\newtheorem{Rem}[equation]{Remark}

\newcommand{\nc}{\newcommand}
\nc{\dmo}{\DeclareMathOperator}

\nc{\Adhspace}{\hspace{-1.5pt}}

\nc{\HZ}{{\rmH \hspace{-0.2em}\bbZ}}
\nc{\HFp}{{\rmH \hspace{-0.15em}\bbF_{\hspace{-0.1em}p}}}
\dmo{\Iso}{Iso}
\dmo{\Ab}{Ab}
\dmo{\Abelem}{Abelem}
\dmo{\Add}{Add}
\dmo{\Aut}{Aut}
\dmo{\Bi}{bi}
\dmo{\Bisets}{Bisets}
\dmo{\CAT}{CAT}
\dmo{\coev}{coev}
\dmo{\Coloc}{Coloc}
\dmo{\ev}{ev}
\dmo{\Fib}{Fib}
\dmo{\Free}{Free}
\dmo{\Id}{Id}
\dmo{\Loc}{Loc}
\dmo{\rmI}{I}
\dmo{\rmL}{L}
\dmo{\rmR}{R}
\dmo{\Spc}{Spc}
\dmo{\thick}{thick}
\dmo{\Thick}{Thick}
\dmo{\chara}{char}%
\dmo{\coh}{coh} 
\dmo{\Coind}{CoInd}
\dmo{\coker}{coker}
\dmo{\pd}{pd}
\dmo{\fd}{fd}

\dmo{\pt}{\text{pt}}
\dmo{\cone}{cone}
\dmo{\Der}{D}
\dmo{\Derqc}{D_{\Qcoh}}
\nc{\Rder}{\mathrm{R}} 
\nc{\Lder}{\mathrm{L}} 
\dmo{\Khocat}{K}
\dmo{\End}{End}
\dmo{\Ext}{Ext}
\dmo{\rmH}{H}
\dmo{\Ho}{Ho}
\dmo{\Hom}{Hom}
\dmo{\id}{id}
\dmo{\Img}{Im}
\dmo{\incl}{incl}
\dmo{\Ker}{Ker}
\dmo{\Res}{res}
\dmo{\Infl}{infl}
\dmo{\triv}{triv}
\dmo{\Ind}{ind}
\dmo{\CoInd}{coind}
\dmo{\Oname}{O}
\nc{\Op}{\Oname^p}
\nc{\Oq}{\Oname^q}
\nc{\smashh}{\wedge}
\dmo{\Ad}{Ad}
\dmo{\Les}{Les}
\dmo{\Map}{Map}%
\dmo{\Mod}{Mod}
\dmo{\GrMod}{GrMod}
\dmo{\lax}{lax}
\dmo{\modname}{mod}%
\dmo{\grmod}{grmod}
\dmo{\Mor}{Mor}%
\dmo{\Obj}{Obj}
\dmo{\opname}{op}
\dmo{\Or}{Or}
\dmo{\pr}{pr}
\dmo{\canin}{in} 
\dmo{\Proj}{Proj} 
\dmo{\proj}{proj}
\dmo{\Qcoh}{Qcoh}
\dmo{\rank}{rank}
\dmo{\Rname}{R}
\dmo{\DM}{DM}
\dmo{\SH}{SH}
\nc{\SHc}{{\SH^c}}
\nc{\SHp}{{\SH_{(p)}}}
\nc{\SHcp}{{\SH^c_{(p)}}}
\nc{\SHG}{\SH(G)}
\nc{\SHGp}{\SH(G)_{(p)}}
\nc{\SHGc}{\SHG^c}
\nc{\SHGcp}{\SHG^c_{(p)}}
\dmo{\smallb}{b}
\dmo{\smallperf}{perf}
\dmo{\Span}{Span}
\dmo{\Spec}{Spec}
\dmo{\Stab}{Stab}
\dmo{\stab}{stab}
\dmo{\supp}{supp}
\dmo{\switch}{switch}
\dmo{\TTR}{TTR}

\nc{\locus}{compactness locus}
\dmo{\lenormal}{\unlhd}
\dmo{\lnormal}{\lhd}
\dmo{\Z}{Z}
\dmo{\A}{{\cat A}}
\nc{\ZGN}{\Z_{N,G}}
\nc{\tildeE}{\widetilde{E}}
\nc{\AGN}{\A_{N,G}}
\nc{\E}{{\cat E}}
\nc{\F}{{\cat F}}
\nc{\G}{{\cat G}}
\nc{\perfgen}{{\cat S}}
\nc{\U}{{\cat U}}
\nc{\V}{{\cat V}}

\nc{\Beren}[1]{{\color{MidnightBlue}#1}}
\nc{\Ivo}[1]{{\color{OliveGreen}#1}}
\nc{\Paul}[1]{{\color{Violet}#1}}
\nc{\Pout}[1]{\Paul{\sout{#1}}}
\nc{\Bout}[1]{\Beren{\sout{#1}}}
\nc{\Iout}[1]{\Ivo{\sout{#1}}}

\nc{\notsupseteq}{\not\supseteq\hspace{-0.7ex}}
\nc{\Fp}{{\mathbb{F}_{\hspace{-2pt}p}}}
\nc{\Fpstar}{{\mathbb{F}^*_{\hspace{-2pt}p}}}
\nc{\IFF}{$\Leftrightarrow$}
\nc{\DO}{\omega_f}
\nc{\Crelcpt}{\cat{C}^{c/f}}
\nc{\Crich}{\underline{\cat{C}}}
\nc{\DbG}{\Db(\kk G\mmod)}
\nc{\uA}{\underline{A}}
\nc{\doublequot}[3]{#1\backslash #2/#3}
\nc{\HGK}{\doublequot HGK}
\nc{\quadtext}[1]{\quad\textrm{#1}\quad}
\nc{\qquadtext}[1]{\qquad\textrm{#1}\qquad}
\nc{\PZG}{\cat C_{\bbZ}(\bbZ G)}
\nc{\TTRK}{\TTR(\cat K)}
\nc{\psets}{\mathsf{-sets}_\sbull}
\nc{\Gsets}{G\mathsf{-sets}}
\nc{\Hsets}{H\mathsf{-sets}}
\nc{\AddK}{\Add^{\Sigma}(\cat K)}
\nc{\adj}{\dashv}
\nc{\adjto}{\rightleftarrows}
\nc{\AK}{A\MModcat{K}}
\nc{\BK}{B\MModcat{K}}
\nc{\bbA}{\mathbb{A}}
\nc{\bbB}{\mathbb{B}}
\nc{\bbC}{\mathbb{C}}
\nc{\bbI}{\mathbb{I}}
\nc{\bbN}{\mathbb{N}}
\nc{\bbP}{\mathbb{P}}
\nc{\bbQ}{\mathbb{Q}}
\nc{\bbR}{\mathbb{R}}
\nc{\bbZ}{\mathbb{Z}}
\nc{\bbF}{\mathbb{F}}
\nc{\bbZp}{\mathbb{Z}_{(p)}}
\nc{\Sphere}{\mathbb{S}} 
\nc{\cat}[1]{\mathscr{#1}}
\nc{\Displ}{\displaystyle}
\nc{\ie}{{\sl i.e.}\ }
\nc{\into}{\mathop{\rightarrowtail}}
\nc{\inv}{^{-1}}
\nc{\isoto}{\buildrel \sim\over\to}
\nc{\isotoo}{\mathop{\buildrel \sim\over\too}}
\nc{\kk}{\Bbbk}
\nc{\onto}{\mathop{\twoheadrightarrow}}
\nc{\too}{\mathop{\longrightarrow}\limits}
\nc{\xytriangle}[7]{\xymatrix@C=#7em{#1\ar[r]^-{\Displ #4} & #2 \ar[r]^-{\Displ #5}&#3\ar[r]^-{\Displ #6}&T #1}}
\nc{\ababs}{{\sl ab absurdo}}
\nc{\adh}[1]{\overline{#1}}
\nc{\adhpt}[1]{\adh{\{#1\}}}
\nc{\aka}{{a.\,k.\,a.}\ }
\nc{\ala}{{\sl \`a la}\ }
\nc{\Autcat}[1]{\Aut_{\cat #1}}
\nc{\cO}{\mathcal{O}}
\nc{\calO}{\mathcal{O}}
\nc{\dimm}{{\operatorname{dim}}}
\nc{\tordim}{{\operatorname{tor-dim}}}
\nc{\cV}{\mathcal{V}}
\nc{\Db}{\Der^{\smallb}}
\nc{\Dqc}{\Der_{\Qcoh}}
\nc{\Dperf}{\Der^{\smallperf}}
\nc{\eg}{{\sl e.\,g.}}
\nc{\eps}{\epsilon}
\nc{\FFree}{\,\text{--}\Free}%
\nc{\FFreecat}[1]{\FFree_{\cat #1}}
\nc{\FK}{\mathcal{F}(\cat K)}
\nc{\gm}{\mathfrak{m}}
\nc{\Homcat}[1]{\Hom_{\cat #1}}
\nc{\Morcat}[1]{\Mor_{\cat #1}}
\nc{\hook}{\hookrightarrow}
\nc{\Idcat}[1]{\Id_{\cat{#1}}}
\nc{\ideal}[1]{\langle #1\rangle}
\nc{\ihom}{{\mathsf{hom}}} 
\nc{\ihomcat}[1]{\ihom_{\cat #1}}
\nc{\Kcat}[1]{#1\MModcat{K}}
\nc{\KP}{\cat{K}_{\cat P}}
\nc{\loccit}{{\sl loc.\ cit.}}
\nc{\lind}{\rmL\!}
\nc{\RR}{\rmR\!}
\nc{\RRb}{\mathbb{R}}
\nc{\Lotimes}{\otimes^{\rmL}}
\nc{\Mid}{\,\big|\,}
\nc{\MMod}{\,\text{-}\Mod}%
\nc{\MModcat}[1]{\MMod_{\cat #1}}%
\nc{\mmod}{\,\text{--}\modname}%
\nc{\mmodb}{\mmod^\sbull}%
\nc{\op}{^{\opname}}
\nc{\oto}[1]{\overset{#1}\to}
\nc{\otoo}[1]{\overset{#1}{\,\too\,}}
\nc{\ourfrac}[2]{\genfrac{}{}{0pt}{}{\Displ #1}{\scriptstyle #2}}
\nc{\ouriff}{\Leftrightarrow}
\nc{\oursetminus}{\!\smallsetminus\!}
\nc{\potimes}[1]{^{\otimes #1}}
\nc{\pproj}{\,\text{-}\proj}
\nc{\ptimes}[1]{^{\times #1}}
\nc{\dd}[1]{_{{\scriptscriptstyle(#1)}}}
\nc{\uu}[1]{^{{\scriptscriptstyle(#1)}}}
\nc{\pushout}{\textrm{\rm p.o.}}
\nc{\qp}{q_{_{\scriptstyle \cat P}}\!}%
\nc{\Rcat}[1]{\Rname_{\cat #1}^\sbull}
\nc{\rdto}{}
\nc{\restr}[1]{_{|_{\scriptstyle #1}}}
\nc{\RK}{\Rcat{K}}
\nc{\sbull}{{\scriptscriptstyle\bullet}}
\nc{\SET}[2]{\big\{\,#1\Mid#2\,\big\}}
\nc{\SHA}{\SH{}^{\bbA^{1}}}
\nc{\SHfin}{\SH^{\text{\rm fin}}}
\nc{\smallmatrice}[1]{\left(\begin{smallmatrix} #1 \end{smallmatrix}\right)}
\nc{\SpcAK}{\Spc(A\MModcat{K})}
\nc{\SpcK}{\Spc(\cat K)}
\nc{\suppcat}[1]{\supp(\cat #1)}
\nc{\then}{\Rightarrow}
\nc{\tideal}[1]{\ideal{#1}}
\nc{\unit}{\mathbb{1}}
\nc{\unitcat}[1]{\unit_{\cat #1}}
\nc{\onept}{\mathrm{B}} 

\nc{\HG}{\!{}^{^H}\overline{G}}
\nc{\uY}{\widetilde{Y}}

\nc{\Dk}{\dual_{\kappa}}
\nc{\Dkk}{\dual_{\kappa'}}
\nc{\bs}{\backslash}
\nc{\biCpt}{\mathrm{biCpt}}
\nc{\biLCpt}{\mathrm{biLCpt}} 
\nc{\Grps}{\mathsf{Grps}}
\nc{\Sets}{\mathsf{Sets}}
\nc{\Top}{\mathsf{Top}}
\nc{\Comp}{\mathsf{Top}^{\mathsf{comp}}}

\nc{\lG}{{}_{{\color{Gray}\scriptscriptstyle G}}}
\nc{\lH}{{}_{{\color{Gray}\scriptscriptstyle H}}}
\nc{\rG}{_{{\color{Gray}\!\scriptscriptstyle G}}}
\nc{\rH}{_{{\color{Gray}\!\scriptscriptstyle H}}}
\nc{\rK}{_{{\color{Gray}\!\scriptscriptstyle K}}}

\nc{\dual}{\Delta}

\nc{\ra}{\rightarrow}
\nc{\xra}{\xrightarrow}

\nc{\C}{\mathbb{C}} 
\nc{\Cont}{\mathrm{C}} 
\nc{\KK}{\mathsf{KK}}
\nc{\Modules}{\mathsf{Mod}}
\nc{\Alg}{\mathsf{Alg}}
\nc{\Sep}{\mathsf{Sep}}
\begin{document}



\title[Compactness locus and Adams isomorphism]{The compactness locus of a geometric functor and the formal construction of the Adams isomorphism}

\author{Beren Sanders}
\date{October 20, 2018}

\address{Beren Sanders, Mathematics Department, UC Santa Cruz, 95064 CA, USA}
\email{beren@ucsc.edu}
\urladdr{http://people.ucsc.edu/$\sim$beren}

\begin{abstract}
	We introduce the compactness locus of a geometric functor between
	rigidly-compactly generated tensor-triangulated categories, and describe it
	for several examples arising in equivariant homotopy theory and algebraic
	geometry.  It is a subset of the tensor-triangular spectrum of the target
	category which, crudely speaking, measures the failure of the functor to
	satisfy Grothendieck-Neeman duality (or equivalently, to admit a left
	adjoint).  We prove that any geometric functor --- even one which does not
	admit a left adjoint --- gives rise to a Wirthm\"{u}ller isomorphism once
	one passes to a colocalization of the target category determined by the
	compactness locus.  When applied to the inflation functor in equivariant
	stable homotopy theory, this produces the Adams isomorphism.
\end{abstract}


\thanks{Author supported by the Danish National Research Foundation through the Centre for Symmetry and Deformation (DNRF92)}

\maketitle


\vskip-\baselineskip\vskip-\baselineskip
\tableofcontents
\vskip-\baselineskip\vskip-\baselineskip\vskip-\baselineskip


\section{Introduction}
\medskip

Motivated by a desire to clarify the relationship between Grothendieck duality in algebraic geometry and the Wirthm\"{u}ller isomorphism in equivariant stable homotopy theory, P.~Balmer, I. Dell'Ambrogio, and the present author recently made a general study \cite{BalmerDellAmbrogioSanders16} of the existence and properties of adjoints to a geometric functor $f^*:\cat D \to \cat C$ between rigidly-compactly generated tensor-triangulated categories.  (These definitions will be recalled in Section~\ref{sec:main-thm}.) One highlight of that work was the realization that the Wirthm\"{u}ller isomorphism is a general formal phenomenon, emerging (in our rigidly-compactly generated setting) as an isomorphism between the left and right adjoint of any tensor-triangulated functor --- up to a twist by the relative dualizing object --- whenever those adjoints exist.  Moreover, we proved that the existence of those adjoints is equivalent to the original functor satisfying a form of Grothendieck duality.

More precisely, as explained in \cite{BalmerDellAmbrogioSanders16}, a geometric functor $f^*:\cat D \to \cat C$ between rigidly-compactly generated categories always admits a right adjoint $f_*$ which itself admits a right adjoint $f^!$\,:
\[\xymatrix{
		\cat D \ar[d]_{f^*} \ar@<22pt>@{<-}[d]^-{f_*} \ar@<45pt>[d]^-{f^!} \\
		\cat C
	}\]
	Then, defining the \emph{relative dualizing object} to be $\DO:=f^!(\unit_{\cat D})$, we proved:
\begin{Thm}[{{\cite[Thm.~B]{BalmerDellAmbrogioSanders16}}}]
	\label{thm:original-thm}
	The following are equivalent:
	\begin{enumerate}
		\item The functor $f^{!}$ admits a right adjoint.
		\item There is a natural isomorphism $f^! \simeq \DO \otimes f^*(-)$.
		\item The functor $f_*$ preserves compact objects.
		\item There is a natural isomorphism $f^* \simeq \ihom(\DO,f^!(-))$.
		\item The functor $f^*$ admits a left adjoint $f_!$.
	\end{enumerate}
	Moreover, in this case, there is a canonical natural isomorphism
	\begin{equation}\label{eq:original-wirthmuller}
		f_! \simeq f_*(- \otimes \DO).
	\end{equation}
\end{Thm}

We say that $f^*$ satisfies \emph{Grothendieck-Neeman duality} (or \emph{GN-duality}, for short) when it satisfies the equivalent conditions (a)--(e) of Theorem~\ref{thm:original-thm}.  This terminology is motivated by the natural isomorphism in (b), exhibiting the double right adjoint $f^!$ as a twisted version of the original functor $f^*$, as in classical algebro-geometric Grothendieck duality.  On the other hand, when applied to the restriction functor $f^*:=\Res_H^G:\SH(G) \to \SH(H)$ between equivariant stable homotopy categories,~\eqref{eq:original-wirthmuller} specializes to the Wirthm\"{u}ller isomorphism between induction and coinduction.  In this way, Theorem~\ref{thm:original-thm} provides a purely formal, canonical construction of the classical Wirthm\"{u}ller isomorphism in equivariant stable homotopy theory.

On the other hand, there is another important isomorphism in equivariant stable homotopy theory: the Adams isomorphism.  Loosely speaking, it is an isomorphism (up to a twist) between the $N$-orbits and the $N$-fixed points of an $N$-free $G$-spectrum, and it lies at the heart of genuine equivariant stable homotopy theory, appearing for example in the tom Dieck splitting of equivariant homotopy groups.  More precisely, given a closed normal subgroup $N \lenormal G$ of a compact Lie group $G$, it is a natural isomorphism of $G/N$-spectra
\begin{equation}\label{eq:adams-iso-intro}
	(i^*X \smashh E\cat F(N)_+)/N \cong \lambda^N(X \smashh S^{{{-}\Adhspace\Ad}(N;G)})
\end{equation}
defined for any $N$-free $G$-spectrum $X$.
(For the uninitiated, this notation will be explained in~Section~\ref{sec:adams-wirth}.)
At naive first glance,~\eqref{eq:adams-iso-intro} looks like it might be some kind of twisted isomorphism between a left and right adjoint, and it is then very natural to attempt to give it a formal treatment, just like Theorem~\ref{thm:original-thm} gave a formal treatment of the classical Wirthm\"uller isomorphism.  Indeed, J.P.~May raises the problem of giving a formal analysis of the Adams isomorphism in~\cite{May03}.

Now, the functor $\lambda^N : \SH(G) \to \SH(G/N)$ appearing on the right-hand side of~\eqref{eq:adams-iso-intro} is the categorical fixed point functor, which is right adjoint to the inflation functor $\Infl_{G/N}^G : \SH(G/N) \to \SH(G)$.  However, the tom Dieck splitting theorem implies that $\lambda^N$ (with or without the twist by $S^{-\Adhspace\Ad(N;G)}$) does not preserve compact objects, except in the trivial case when $N=1$ (see Proposition~\ref{prop:infl-doesnt-satisfy-gn-duality} below).  It then follows from condition (c) of Theorem~\ref{thm:original-thm} that the inflation functor $\Infl_{G/N}^G$ does not have a left adjoint, and moreover, that the right-hand side of the Adams isomorphism~\eqref{eq:adams-iso-intro} cannot be left adjoint to \emph{any} functor between compactly generated categories which admits a right adjoint.
Thus, naive attempts to understand~\eqref{eq:adams-iso-intro} as a twisted isomorphism between a left and right adjoint\label{page:adams} do not succeed.

Nevertheless, we will show that the Adams isomorphism can be given a purely formal, conceptual construction, and that it can in fact be realized as a Wirthm\"{u}ller isomorphism, properly understood.
To this end, we continue the study initiated in~\cite{BalmerDellAmbrogioSanders16}, now with a focus on functors $f^*$ --- such as $\Infl_{G/N}^G$ --- which do not satisfy GN-duality (\ie~ do not have a left adjoint).
We can prove that every geometric functor unconditionally gives rise to a Wirthm\"{u}ller isomorphism (and satisfies GN-duality) once one passes to a canonically determined finite colocalization of the target category:
\begin{Thm}\label{thm:main-thm-intro}
	Let $f^*:\cat D\to \cat C$ be a geometric functor between rigidly-compactly generated tensor-triangulated categories.
	Consider the thick $\otimes$-ideal 
	\[\A_f := \SET{ x \in \cat C^c }{ f_*(x \otimes y) \text{ is compact for all } y \in \cat C^c} \subseteq \cat C^c \]
	and let $\Gamma^* : \cat C \to \cat B := \Loc\langle \A_f\rangle$ 
	denote the associated finite colocalization, \ie the
	right adjoint of the inclusion functor $\Gamma_! : \cat B \hookrightarrow \cat C$.
Then:
\begin{enumerate}
	\item The composite $\Gamma^* \circ f^! : \cat D \to \cat B$ has a right adjoint.
	\item There is a canonical natural isomorphism $\Gamma^*\circ f^! \cong \Gamma^*(\DO \otimes f^*(-))$.
	\item There is a canonical natural isomorphism $\Gamma^*\circ f^* \cong \Gamma^*(\ihom(\DO,f^!(-)))$.
	\item The composite $\Gamma^* \circ f^* : \cat D \to \cat B$ has a left adjoint, which we will denote
	\[ (f \circ \Gamma)_! : \cat B \to \cat D \]
	and there is a canonical natural isomorphism
	\begin{equation}\label{eq:generalized-wirthmuller-intro}
	 (f \circ \Gamma)_!(x) \cong f_*( \Gamma_!(x) \otimes \DO ) 
	\end{equation}
	for all $x \in \cat B$.
\end{enumerate}
\end{Thm}

This theorem will be proved in Section~\ref{sec:main-thm}.  In the case that $f^*$ does satisfy GN-duality, then $\cat B=\cat C$ and~\eqref{eq:generalized-wirthmuller-intro} recovers the original Wirthm\"{u}ller isomorphism~\eqref{eq:original-wirthmuller}.  We shall see in Section~\ref{sec:adams-wirth} that when applied to the inflation functor $f^*:=\Infl_{G/N}^G$, the associated ``Wirthm\"{u}ller'' isomorphism \eqref{eq:generalized-wirthmuller-intro} is nothing but the Adams isomorphism.  In this way, Theorem~\ref{thm:main-thm-intro} provides a purely formal, conceptual, and canonical construction of the Adams isomorphism.  Moreover, we obtain a unification of the Adams and Wirthm\"{u}ller isomorphisms in equivariant stable homotopy theory; they arise by applying the same formal construction to inflation and restriction, respectively.

\smallskip

Now let's take a step back.  According to Theorem~\ref{thm:main-thm-intro}, every geometric functor $f^*:\cat D \to \cat C$ between rigidly-compactly generated tensor-triangulated categories has a canonically associated colocalization and concomitant Wirthm\"{u}ller isomorphism.  From the perspective of tensor-triangular geometry \cite{BalmerICM}, the thick tensor-ideal $\A_f$ appearing in the theorem corresponds to a certain Thomason subset 
$$\Z_f \subset \Spc(\cat C^c)$$
of the tensor-triangular spectrum of the subcategory of compact objects $\cat C^c$.  This subset $\Z_f \subset \Spc(\cat C^c)$ is a new invariant of the functor $f^*$ which, crudely speaking, measures the failure of $f^*$ to satisfy GN-duality (\ie~have a left adjoint).
We call it the \emph{\locus{}} of $f^*$ (Def.~\ref{def:locus}) and initiate its general study in Section~\ref{sec:locus}.  The rest of the paper is devoted to understanding the \locus{} geometrically and computing it in specific examples.

We begin our geometric study in Section~\ref{sec:finite}, where we give a complete topological description of the \locus{} of a finite localization (Proposition~\ref{prop:locus-of-finite-localization}) and discuss several examples.

In Section~\ref{sec:adams-locus}, we complete our discussion of the Adams isomorphism by completely describing the \locus{} of the inflation functor $\Infl_{G/N}^G$ for any finite group $G$ and normal subgroup $N \lenormal G$ (Theorem~\ref{thm:locus-of-inflation}).  We will see that the \locus{} singles out a larger subcategory than the subcategory of \mbox{$N$-free} $G$-spectra --- which can be interpreted as saying that the ``natural domain'' of the Adams isomorphism is actually larger than the category of $N$-free $G$-spectra (cf.~Rem.~\ref{rem:adams-larger}).  Two explicit examples for $G=D_{10}$ the dihedral group of order 10 are depicted in Figures~\ref{fig:D10}--\ref{fig:D10C5} on pages \pageref{fig:D10}--\pageref{fig:D10C5}.  We restrict ourselves to finite groups in this section because it is only for finite groups that we have a description of the spectrum $\Spc(\SH(G)^c)$ of the \mbox{$G$-equivariant} stable homotopy category (see \cite{BalmerSanders17}).

In Section~\ref{sec:alg-geom}, we study examples arising in algebraic geometry.  Indeed, given any morphism $f:X \to Y$ of (quasi-compact and quasi-separated) schemes, we can consider the \locus{} of the associated pull-back functor $f^*:\Derqc(Y)\to \Derqc(X)$. It is a certain Thomason subset
	\[ \Z_f \subset X \cong \Spc(\Dqc(X)^c) \]
of the domain of the morphism $f$ which we would like to understand scheme-theoretically. 
One of our main theorems (Theorem~\ref{thm:main-algebro-thm})
states that for a separated morphism of finite type $f:X \to Y$ between noetherian schemes,
the categorically-defined \locus{} $\Z_f$ is 
the union of all closed subsets of $X$ over which the morphism $f$ is proper and which are contained in the scheme-theoretic perfect locus of~$f$.
In particular, the \locus{} of a proper morphism 
coincides with
the largest specialization closed subset 
contained in the perfect locus.

Finally, in Section~\ref{sec:further-directions}, we give some miscellaneous additional examples and discuss directions for future research.
Given its generality, it is perhaps not surprising that there exist geometric functors $f^*$ whose \locus{} is empty, \ie~functors~$f^*$ for which the colocalization of Theorem~\ref{thm:main-thm-intro} is the zero category.  This can be seen in simple algebro-geometric examples, but Example~\ref{exa:eilenberg-maclane} provides a non-trivial example of this phenomenon in stable homotopy theory.  Nevertheless, it is quite satisfying that the property of having a left adjoint can be refined by a topological invariant associated to the functor, and it remains an interesting zoological challenge to analyze the \locus{} of any geometric functor we now find in nature.

\subsection*{Acknowledgements\,:}
We readily thank 
Paul Balmer, 
Ivo Dell'Ambrogio,
John Greenlees, 
Jesper Grodal, 
Irakli Patchkoria, and
Greg Stevenson 
for their encouragement and helpful discussions.
We would also like to thank an anonymous referee for useful comments and especially for noticing a problem with the original proof of
Proposition~\ref{prop:infl-doesnt-satisfy-gn-duality}.
\bigbreak
\section{The Wirthm\"{u}ller isomorphism of a geometric functor}
\label{sec:main-thm}

Let's begin with some standard terminology. 

\begin{Ter}
	By a \emph{tensor-triangulated category}, we mean 
	a triangulated category with a compatible closed symmetric monoidal structure
	as in \cite[App.~A]{HoveyPalmieriStrickland97}.
	Such a category is \emph{rigidly-compactly generated}
	if it is compactly generated as a triangulated category and 
	if
	the compact objects coincide with the rigid objects (\aka the strongly dualizable objects).
	In particular, the unit object $\unit$ is rigid-compact.
	In the language of \cite{HoveyPalmieriStrickland97}, this is precisely the same thing as a ``unital algebraic stable homotopy category''.
	We will sometimes drop the ``tensor-triangulated'' and just speak of rigidly-compactly generated categories.
\end{Ter}

\begin{Exas}\label{exas:examples}
	Several examples of rigidly-compactly generated categories are discussed in \cite[Examples~2.9--2.13]{BalmerDellAmbrogioSanders16}.
	Briefly, examples include:
	\begin{enumerate}
		\item The derived category $\Dqc(X)$ of complexes of $\cat O_X$-modules having quasi-coherent homology, for $X$ a quasi-compact and quasi-separated scheme.
		\item The stable homotopy category $\SH$.
		\item The genuine $G$-equivariant stable homotopy category $\SHG$ for $G$ a compact Lie group.
		\item The stable module category $\Stab(\kk G)$ for $\kk$ a field and $G$ a finite group (or, more generally, $G$ a finite group scheme over $\kk$).
		\item The stable $\bbA^1$-homotopy category $\SHA\hspace{-0.4ex}(\kk)$ 
			for a field $\kk$ of characteristic zero.
		\item The derived category of motives $\DM(\kk;R)$ 
			for a field $\kk$ whose exponential characteristic is invertible in the coefficient ring $R$.
			(See \cite{Totaro18} and the references therein.)
		\item The derived category $\Der(A)$ of a highly-structured commutative ring spectrum $A$ (\eg~a commutative dg-algebra).
		\item Any smashing localization of a rigidly-compactly generated category (see \eg~\cite[Sec.~3.3]{HoveyPalmieriStrickland97}).
	\end{enumerate}
\end{Exas}

\begin{Ter}
	By a \emph{tensor-triangulated functor}, we mean a triangulated functor which is a \emph{strong} monoidal functor. 
	For brevity, a coproduct-preserving tensor-triangulated functor between rigidly-compactly generated categories will
	be called a \emph{geometric functor}.
	(This terminology is motivated by \cite[Def.~3.4.1]{HoveyPalmieriStrickland97}.)
\end{Ter}

\begin{Hyp}
Throughout this section $f^*:\cat D \to \cat C$ will denote a geometric functor between rigidly-compactly generated categories.
As recalled in the Introduction, any such functor admits two adjoints on the right, $f^* \dashv f_* \dashv f^!$, and we define the \emph{relative dualizing object} of $f^*$ to be $\DO:=f^!(\unit_{\cat D})$. (See \cite{BalmerDellAmbrogioSanders16} for more details.)
\end{Hyp}

\begin{Rem}
	Under these hypotheses, the functors $f^* \dashv f_*$ automatically satisfy the 
	following projection formula:
	$f_*(f^*x \otimes y) \cong x \otimes f_* y$ (see \cite[(2.16)]{BalmerDellAmbrogioSanders16}).
\end{Rem}

\begin{Rem}[The spectrum]
	Associated to $\cat C$ is the topological space $\Spc(\cat C^c)$, consisting of the prime tensor-ideals of the subcategory of compact objects~$\cat C^c$ (see \cite{Balmer05a}).  Every compact object $x \in \cat C^c$ has an associated closed subset $\supp(x) := \SET{ \cat P \subset \cat C^c }{x \not\in \cat P} \subset \Spc(\cat C^c)$, and these sets form a basis of closed sets for the topology on $\Spc(\cat C^c)$.  By the abstract Thick Subcategory Classification Theorem (see \cite[Thm~4.10]{Balmer05a} and \cite[Rem.~4.3]{Balmer05a}), the thick tensor-ideals of $\cat C^c$ are in one-to-one correspondence with the \emph{Thomason subsets} of $\Spc(\cat C^c)$ --- \ie the unions of closed sets, each of which has quasi-compact complement.  The bijection sends a thick tensor-ideal $\cat I \subset \cat C^c$ to the Thomason subset $\bigcup_{x \in \cat I} \supp(x)$, while a Thomason subset $Y \subset \Spc(\cat C^c)$ is sent to the thick tensor-ideal $\cat C^c_Y := \SET{x \in \cat C^c}{\supp(x) \subset Y}$.  It will also be useful to recall that the Thomason \emph{closed} sets are precisely those of the form $\supp(x)$ for some compact $x \in \cat C^c$ (cf.~\cite[Prop~2.14]{Balmer05a}).
\end{Rem}
\begin{Rem}[Finite localization]\label{rem:finite-localization}
	Let $Y \subset \Spc(\cat C^c)$ be a Thomason subset of the spectrum, with corresponding thick tensor-ideal $\cat C^c_Y := \SET{x \in\cat C^c}{\supp(x) \subset Y}$, and let $V := \Spc(\cat C^c)\setminus Y$ denote the complement.  By the general theory of finite localizations, there is an associated idempotent triangle \[ e_Y \to \unit \to f_Y \to \Sigma e_Y \] in $\cat C$ (cf.~\cite[Sec.~2--3]{BalmerFavi11}), and we have identifications $e_Y \otimes \cat C = \Loc(e_Y) = \Loc_\otimes(\cat C^c_Y) = \Loc(\cat C^c_Y)$ and $f_Y \otimes \cat C = \Loc(\cat C^c_Y)^\perp$.  We denote the category of colocal objects (\aka acyclic objects) by $\cat C_Y := e_Y \otimes \cat C$ and the category of local objects by $\cat C(V) := f_Y \otimes \cat C$.  The inclusion $\cat C_Y \hookrightarrow \cat C$ of the colocal objects has a right adjoint $e_Y \otimes - : \cat C \to \cat C_Y$, called colocalization, and the inclusion $\cat C(V) \hookrightarrow \cat C$ of the local objects has a left adjoint $f_Y \otimes - :\cat C \to \cat C(V)$, called localization.  (See \cite[Rem.~5.3]{BalmerSanders17} for a diagram of all the relevant functors.)

	The category of local objects $\cat C(V)$ inherits the structure of a tensor-triangulated category such that the localization functor $\cat C \to \cat C(V)$ is a tensor-triangulated functor.  It maps a set of compact-rigid generators of $\cat C$ to a set of compact-rigid generators of $\cat C(V)$. In particular, the unit $f_Y$ is compact in $\cat C(V)$ and $\cat C(V)$ is rigidly-compactly generated just as $\cat C$ is.  The localization functor preserves compact objects and hence induces a map on spectra $\Spc(\cat C(V)^c) \to \Spc(\cat C^c)$ which identifies $\Spc(\cat C(V)^c) \cong V \subset \Spc(\cat C^c)$.  This follows from the Thomason-Neeman localization theorem \cite[Thm.~2.1]{Neeman92b} which identifies $\cat C(V)^c \cong (\cat C/\cat C_Y)^c \cong (\cat C^c/\cat C^c_Y)^\natural$.  

	On the other hand, the category of colocal objects $\cat C_Y$ also inherits the structure of a tensor-triangulated category such that the colocalization functor $e_Y \otimes -:{\cat C \to \cat C_Y}$ is a tensor-triangulated functor.
	By construction, $\cat C_Y$ is compactly generated, but it is usually not rigidly-compactly generated
	since the unit object $e_Y$ is usually not compact in $\cat C_Y$.
	Otherwise, since the inclusion $\cat C_Y \hookrightarrow \cat C$ preserves compact objects, it
	would follow that $e_Y$ is compact in $\cat C$, 
	and this is only possible in the very special situation that $Y$ is an open
	and closed subset of $\Spc(\cat C^c)$ (see the proof of Prop.~\ref{prop:smashing-localization} below).
	In the language of \cite{HoveyPalmieriStrickland97}, the category of colocal objects~$\cat C_Y$ is
	an algebraic stable homotopy category, but it is not a \emph{unital} algebraic stable homotopy category (in general).
\end{Rem}
\begin{Rem}\label{rem:two-view-compact}
	The notation $\cat C^c_Y$ can now be interpreted
	in two ways: as $(\cat C_Y)^c$ or as $(\cat C^c)_Y$.
	However, these coincide. More precisely, setting $\cat K := \cat C^c$,
	we have that $(\cat C^c)_Y =: \cat K_Y = (\Loc(\cat K_Y))^c = (\cat C_Y)^c$
	where the middle equality is \cite[Lem.~2.2]{Neeman92b}.
\end{Rem}

Our goal is to establish the following result:

\begin{Thm}\label{thm:main-thm}
	Let $f^*:\cat D \to \cat C$ be a geometric functor between rigidly-compactly generated tensor-triangulated categories.
	Let $\cat I \subseteq \cat C^c$ be a thick tensor-ideal of the category of compact objects $\cat C^c$ such that
	$f_*(\cat I) \subseteq \cat D^c$.
	Let $\cat B:=\Loc(\cat I) = \Loc_\otimes(\cat I)$ 
	denote the localizing subcategory of $\cat C$ generated by $\cat I$
	and let $\Gamma^*: \cat C \to \cat B$ denote the associated finite colocalization (Rem.~\ref{rem:finite-localization}), \ie the right adjoint of the inclusion $\Gamma_! : \cat B \hookrightarrow \cat C$.
(Keep the diagram
\begin{equation}\label{eq:diagram-of-adjoints}
	\begin{gathered}
	\xymatrix @R=3em{
		\cat D \ar@<-18pt>@{<.}[d]|-{\times} \ar@<-2.5pt>[d]_(.46){f^*} \ar@<2.5pt>@{<-}[d]^-{f_*} \ar@<18pt>[d]^(.45){f^!} \\
		\cat C \ar@<-18pt>@{<-_{)}}[d]_(.49){\Gamma_{\hspace{-1pt}!}} \ar@<-2.5pt>[d]_(.47){\Gamma^*} \ar@<2.5pt>@{<-}[d]^(.49){\Gamma_{\hspace{-1.5pt}*}} \ar@<18pt>@{.>}[d]|-{\times} \\
		\cat B 
	}
	\end{gathered}
\end{equation}
in mind.)
	Then:
	\begin{enumerate}
		\item The composite $\Gamma^* \circ f^!: \cat D \to \cat B$ has a right adjoint.
		\item There is a canonical natural isomorphism $\Gamma^* \circ f^! \cong \Gamma^*(\DO \otimes f^*(-))$.
		\item There is a canonical natural isomorphism $\Gamma^* \circ f^* \cong \Gamma^*(\ihom(\DO,f^!(-)))$.
		\item The composite $\Gamma^* \circ f^* : \cat D \to \cat B$ has a left adjoint, which we will denote by
		\[ (f \circ \Gamma)_! : \cat B \to \cat D \]
	and there is a canonical natural isomorphism
	\begin{equation}\label{eq:generalized-wirthmuller}
		 (f \circ \Gamma)_! \cong f_*(\Gamma_!(-) \otimes \DO) 
	\end{equation}
	which we call the \emph{Wirthm\"{u}ller isomorphism} associated to $f^*$ and $\cat I$.
\end{enumerate}
\end{Thm}
\begin{Rem}
	If $f^*$ satisfies GN-duality (\ie~if it already admits a left adjoint)
	then we can just take $\cat I := \cat C^c$.
	In this case, $\cat B = \cat C$ and the colocalization is just the identity.
One can readily check 
(cf.~Rem.~\ref{rem:defn-wirth} and \cite[(3.19)]{BalmerDellAmbrogioSanders16})
that the Wirthm\"{u}ller isomorphism~\eqref{eq:generalized-wirthmuller} reduces in this case to the original Wirthm\"{u}ller isomorphism~\eqref{eq:original-wirthmuller} of Theorem~\ref{thm:original-thm}.
	The point of the new theorem is that it holds for functors which do not satisfy GN-duality.
\end{Rem}
\begin{Rem}\label{rem:not-rigidly-compactly-generated}
	The composite $\Gamma^*\circ f^*$ also has a right adjoint $f_* \circ \Gamma_*$.
	However, the category $\cat B$ is usually not rigidly-compactly generated
	and so we cannot apply the original Theorem~\ref{thm:original-thm} to the
	functor $\Gamma^*\circ f^*$. (The unit $\unit_{\cat B}$ is the left idempotent for the colocalization
	and is almost never compact, neither in $\cat B$ nor in $\cat C$.)
	Indeed, since $\Gamma^*\circ f^*$ does not preserve compact objects in general,
	its right adjoint $f_* \circ \Gamma_*$ does not itself have a right adjoint.
	This demonstrates, in particular, that 
	the Trichotomy Theorem of \cite[Cor.~1.13]{BalmerDellAmbrogioSanders16} can fail when the
	target category is not rigidly-compactly generated.
\end{Rem}
\begin{Rem}\label{rem:not-right-adjoint}
	The functor appearing on the right-hand side of \eqref{eq:generalized-wirthmuller} uses the \emph{left} adjoint $\Gamma_!$ of $\Gamma^*$ rather than its right adjoint $\Gamma_*$.  Indeed, as mentioned in Remark~\ref{rem:not-rigidly-compactly-generated}, the right adjoint $f_* \circ \Gamma_*$ (or a twisted version like $f_*(\Gamma_*(-)\otimes \DO)$) is very unlikely to preserve coproducts and so have any chance of being a left adjoint.  This is the precise issue that we mentioned in the Introduction (p.~\pageref{page:adams}) about how the Adams isomorphism cannot naively be realized as an isomorphism of left and right adjoints.
\end{Rem}

\begin{Rem}\label{rem:defn-of-psi}
	Let $\psi : f^*x\otimes \DO \to f^!(x)$ denote the map adjoint to
	\[ f_*(f^*x\otimes \DO) \stackrel[\text{proj}]{}{\simeq} x\otimes f_*\DO \xra{1 \otimes \eps} x \otimes \unit \simeq x\]
	and let $\psi^{\text{ad}} : f^*x \to \ihom(\DO,f^!x)$ denote its adjoint.
	The precise statement of Theorem~\ref{thm:main-thm} (b)--(c) is that
	$\Gamma^*(\psi)$ and $\Gamma^*(\psi^{\text{ad}})$ are isomorphisms.
	Similarly, the Wirthm\"{u}ller isomorphism~\eqref{eq:generalized-wirthmuller}
	is precisely the isomorphism induced by $\Gamma^*(\psi^{\text{ad}})$ by taking left adjoints.
	Thus, all three isomorphisms of the theorem are derived from the same source.
	Recognizing that the Wirthm\"{u}ller isomorphism arises from $\Gamma^*(\psi^{\text{ad}})$ 
clarifies the subtle point mentioned in Remark~\ref{rem:not-right-adjoint} that the right-hand side of 
	\eqref{eq:generalized-wirthmuller} involves a mixture of left and right adjoints.
	While this might make formula~\eqref{eq:generalized-wirthmuller} look strange when compared with the original~\eqref{eq:original-wirthmuller},
the two isomorphisms from which they spring, 
$\Gamma^*(\psi^{\text{ad}})$ and $\psi^{\text{ad}}$,
	are perfectly harmonious.
\end{Rem}

\begin{Rem}\label{rem:not-formal}
	Obtaining an ``explicit'' description of the relative dualizing object~$\DO$ is of course not formal.
	Indeed, the notion of ``explicit'' is subjective and naturally depends on the particular
	subject at hand --- be it equivariant stable homotopy theory, algebraic geometry, modular representation theory, and so on.
	Nevertheless, there is a useful formal procedure we can use to get a grasp on $\DO$ by applying previously established work.
	Namely, suppose $\Theta$ is an ``explicit'' object (\ie~an object we have explicitly defined using the methods of our
	particular subject) for which we have established that $f_*(\Gamma_!(-) \otimes \Theta)$ is left adjoint to $\Gamma^*\circ f^*$.
	Then by taking right adjoints we have
	$\Gamma^*\ihom(\Theta,\DO) \cong \unit_{\cat B}$
	so that $\Gamma^* \DO \cong \Gamma^* \Theta$ if, moreover, the object $\Theta$ is invertible.
	If $f^*$ satisfies GN-duality then of course $\Gamma^*$ may be removed and we can conclude that
	$\DO \cong \Theta$, thereby giving an ``explicit'' description of the categorically defined $\DO$.
	As an example,
	this method shows that the relative dualizing object $\DO$ of the
	restriction functor $f^*:=\Res_H^G:{\SHG\to \SH(H)}$ in equivariant stable
	homotopy theory is isomorphic to 
	the representation sphere $S^{L(H,G)}$ associated to the
	tangent representation of $G/H$ at the identity coset
	(cf.~the proof of \cite[Prop.~4.4]{BalmerDellAmbrogioSanders16}).
\end{Rem}

We devote the remainder of this section to proving Theorem~\ref{thm:main-thm}. 

\begin{Not}
We write $\dual x := \ihom(x,\unit)$ for the dual of a (not necessarily rigid) object $x$, 
and abbreviate $\cat C(x,y) := \Hom_{\cat C}(x,y)$.
In commutative diagrams, we will sometimes use the symbol $\pi$ to indicate the use of 
the $f^* \dashv f_*$ projection formula:
$f_*(f^*x\otimes y)\cong x\otimes f_*y$ (see \cite[(2.16)]{BalmerDellAmbrogioSanders16}).
\end{Not}
\begin{Lem}\label{lem:preserves-coproducts}
	The composite $\Gamma^* \circ f^! : \cat D \to \cat B$ preserves coproducts, and hence has 
	a right adjoint.
\end{Lem}
\begin{proof}
	For a set $\{d_i\}_{i \in I}$ of objects in $\cat D$, let $\varphi : \amalg_i \Gamma^*f^!d_i \to \Gamma^*f^!\amalg_i d_i$
	denote the canonical map.
	By Yoneda, it suffices to prove that the natural map
	\begin{equation}\label{eq:yoneda-prod}
		\cat B(b,\amalg_i \Gamma^* f^!d_i) \xra{\varphi\circ -} \cat B(b,\Gamma^* f^!\amalg_i d_i)
	\end{equation}
	is an isomorphism for all $b \in \cat B$.
	Since $\cat B$ is compactly generated, it suffices to check this for $b \in \cat B^c$ compact.
	Now, for $b \in \cat B^c$ compact we have a chain of isomorphisms
	\begin{align*}
		\cat B(b,\amalg_i \Gamma^* f^!d_i) &\;\;\cong\;\; \oplus_i \cat B(b,\Gamma^*f^!d_i)
					&&b \textrm{ is compact} \\
		&\;\;\cong\;\; \oplus_i \cat C(\Gamma_! b,f^!d_i)
					&&\textrm{adjunction } \Gamma_! \dashv \Gamma^*\\
		&\;\;\cong\;\; \oplus_i \cat D(f_*\Gamma_! b,d_i)
					&&\textrm{adjunction } f_*\dashv f^!\\
		&\;\;\cong\;\; \cat D(f_*\Gamma_! b,\amalg_i d_i)
					&&f_*\Gamma_! b \textrm{ is compact}\\
		&\;\;\cong\;\; \cat C(\Gamma_!b,f^!\amalg_i d_i) 
					&&\textrm{adjunction }f_*\dashv f^!\\
		&\;\;\cong\;\; \cat B(b,\Gamma^*f^!\amalg_i d_i)
					&&\textrm{adjunction }\Gamma_!\dashv \Gamma^*.
	\end{align*}
	Here we have used that $\Gamma_!b \in \Loc(\cat I)^c = \cat I$ (Rem.~\ref{rem:two-view-compact}) and hence $f_*(\Gamma_!b)$ is compact in $\cat D$.
	Then we just need to check that the composite isomorphism coincides with \eqref{eq:yoneda-prod}.
	This is a routine diagram chase.
	The existence of the right adjoint follows from Neeman's Brown Representability Theorem for compactly generated categories (see \eg~\cite[Cor.~2.3]{BalmerDellAmbrogioSanders16}).
\end{proof}

\begin{Prop}\label{prop:b-iso}
The natural transformation
\[ \Gamma^*(\psi) : \Gamma^*(f^*x\otimes \DO) \to \Gamma^*f^!x \]
(cf.~Rem.~\ref{rem:defn-of-psi}) is a natural isomorphism for all $x \in \cat D$.
\end{Prop}
\begin{proof}
	By Lemma~\ref{lem:preserves-coproducts}, both sides are coproduct-preserving exact functors $\cat D \to \cat B$.
	As $\cat D$ is compactly generated, it suffices to prove the statement for $x \in \cat D^c$ compact.
	Then by Yoneda it suffices to show that
	\begin{equation}\label{eq:gamma-psi-yoneda}
		\cat B(b,\Gamma^*(f^*x \otimes \DO)) \xrightarrow{\Gamma^*(\psi)\circ -} \cat B(b,\Gamma^*f^!x)
	\end{equation}
	is an isomorphism for all $b \in \cat B$.
	Now observe that for $b \in \cat B$ and compact $x \in \cat D^c$, we have a chain of isomorphisms
\begin{equation}
\label{eq:chaina}
\begin{aligned}
		\cat B(b,\Gamma^*(f^*x \otimes \DO)) &\;\;\cong\;\; \cat C(\Gamma_!b,f^*x\otimes \DO) 
		&&
		\\
		&\;\;\cong\;\; \cat C(\Gamma_!b\otimes\dual f^*x,\DO) 
		&&f^*x \in \cat C^c \textrm{ is rigid}\\
		&\;\;\cong\;\; \cat C(\Gamma_!b\otimes f^*\dual x,\DO)
		&&x\in\cat D^c \textrm{ is rigid}\\
		&\;\;\cong\;\; \cat D(f_*(\Gamma_!b \otimes f^*\dual x),\unit_{\cat D})
		&&
		\\
		&\;\;\cong\;\; \cat D(f_*\Gamma_!b \otimes \dual x,\unit_{\cat D})
		&&\textrm{projection formula}\\
		&\;\;\cong\;\; \cat D(f_*\Gamma_!b, x) 
		&&x \in \cat D^c \textrm{ is rigid}\\
		&\;\;\cong\;\; \cat C(\Gamma_!b,f^!x) 
		&&
		\\
		&\;\;\cong\;\; \cat B(b,\Gamma^*f^!x).
		&&
	\end{aligned}
\end{equation}
	We claim that the composite of this chain of isomorphisms coincides (for $x \in \cat D^c$) with \eqref{eq:gamma-psi-yoneda}.
	According to the definition of $\psi$ (Rem.~\ref{rem:defn-of-psi}), $\Gamma^*(\psi)$ is
	\begin{equation*}
		\Gamma^*(f^*x \otimes \DO) \xra{\Gamma^* \eta} \Gamma^*f^!f_*(f^*x\otimes \DO)\xra{\Gamma^*f^!\pi} \Gamma^*f^!(x\otimes f_*\DO) \xra{\Gamma^*f^!(1\otimes\eps)} \Gamma^*f^!x.
	\end{equation*}
	On the other hand, the chain of isomorphisms~\eqref{eq:chaina} sends $u \in \cat B(b,\Gamma^*(f^*x\otimes \DO))$ to the composite
	\begin{equation*}
		\begin{tikzpicture}
			\def\arrowlen{10.4pt};
			\def\varrowlen{30pt};
			\node (b1) {$b$};
			\node [right=\arrowlen of b1] (b2) {$\Gamma^*\Gamma_! b$};
			\node [right=\arrowlen of b2] (b3) {$\Gamma^*f^!f_*\Gamma_! b$};
			\node [right=\arrowlen of b3] (b4) {$\Gamma^*f^!(f_*\Gamma_! b\otimes \dual x\otimes x) \simeq \Gamma^*f^!(f_*(\Gamma_! b \otimes f^*\dual x) \otimes x)$\,\,};
			\node [below=\varrowlen of b4.east,anchor=east] (c4) {$\Gamma^*f^!(f_*(\Gamma_! b \otimes \dual f^* x) \otimes x)$\,\,};
			\node [below=\varrowlen of c4.east,anchor=east] (d4) {$\Gamma^*f^!(f_*(\Gamma_!\Gamma^*(f^*x\otimes \DO) \otimes \dual f^*x)\otimes x)$\,\,};
			\node [below=\varrowlen of d4.east,anchor=east] (e4) {$\Gamma^*f^!(f_*(f^*x\otimes \DO\otimes \dual f^* x)\otimes x)$\,\,};
			\node [below=\varrowlen of e4.east,anchor=east] (f4) {$\Gamma^*f^!(f_*(\DO \otimes f^*x\otimes \dual f^* x)\otimes x)$\,.};
			\node [anchor=west] (f3) at (b3.west|-f4) {$\Gamma^* f^!(f_*\DO \otimes x)$};
			\node (f2) at (b1|-f4) {$\Gamma^* f^!x$};
			\draw [->] (b1) -- node[above]{$\scriptstyle \eta $} (b2);
			\draw [->] (b2) -- node[above]{$\scriptstyle \eta $} (b3);
			\draw [->] (b3) -- node[above]{$\scriptstyle \text{coev}$} (b4);
			\coordinate (o) at (b4.south -| c4);
			\draw [->] (o) -- node[right]{$\scriptstyle \simeq$} (c4);
			\draw [->] (c4.south -| o) -- node[right]{$\scriptstyle u$} (d4.north -| o);
			\draw [->] (d4.south -| o) -- node[right]{$\scriptstyle \eps$} (e4.north -| o);
			\draw [->] (e4.south -| o) -- node[right]{$\scriptstyle \simeq$} (f4.north -| o);
			\draw [->] (f4) -- node[above]{$\scriptstyle \text{ev}$} (f3);
			\draw [->] (f3) -- node[above]{$\scriptstyle \eps $} (f2);
			\draw [dashed, ->] (b1) -- (f2);
		\end{tikzpicture}
	\end{equation*}
Our claim 
then follows from a lengthy but routine diagram chase.
In performing this verification, 
the commutativity of
\[\begin{tikzcd}
		f_*(\DO\otimes f^*x)\otimes \dual x \arrow{r}[swap]{\simeq}{\pi}
		\arrow{d}{\simeq}[swap]{\pi^{-1}\otimes 1} & f_*(\DO \otimes f^*x \otimes f^*\dual x) \arrow{d}{\simeq} \\
		f_*\DO \otimes x \otimes \dual x \arrow{r}{\pi}[swap]{\simeq} & f_*(\DO\otimes f^*(x\otimes \dual x)).
	\end{tikzcd}\]
is useful, which can be readily checked
using the definition of the projection formula (\cite[(2.16)]{BalmerDellAmbrogioSanders16}).
\end{proof}
\begin{Lem}\label{lem:adj}
	For any $b \in \cat B^c$, we have natural isomorphisms
	\begin{equation}\label{eq:adj}
		\cat D( \dual f_* \dual \Gamma_!(b),-) \cong \cat C(\Gamma_!(b),f^*(-)) \cong \cat B(b,\Gamma^* f^*(-)) 
	\end{equation}
	of functors $\cat D \to \Ab$.
\end{Lem}
\begin{proof}
	If $b \in \cat B^c$ then $\Gamma_!(b) \in \cat I$ (Rem.~\ref{rem:two-view-compact})
	and hence $\dual\Gamma_!(b) \in \cat I$ 
	(since $\cat I$ is a thick tensor-ideal and $\dual x$ is a
	direct summand 
	of $\dual x \otimes x \otimes \dual x$ for any rigid $x$).
	Hence, $f_*(\dual \Gamma_!(b))$ is compact-rigid in $\cat D$.
	So for any $d \in \cat D$, we have a chain
	\begin{align*}
		\cat D(\dual f_* \dual \Gamma_! b, d) &\;\;\cong\;\; \cat D(\unit_{\cat D},f_* \dual \Gamma_! b \otimes d) &&f_*\dual \Gamma_! b \in \cat D^c \textrm{ is rigid}\\
		&\;\;\cong\;\; \cat D(\unit_{\cat D},f_*(\dual \Gamma_! b \otimes f^* d)) &&\text{projection formula}\\
		&\;\;\cong\;\; \cat C(f^*\unit_{\cat D}, \dual \Gamma_! b \otimes f^* d) &&\\
		&\;\;\cong\;\; \cat C(\unit_{\cat C},\dual \Gamma_! b \otimes f^* d) &&\\
		&\;\;\cong\;\; \cat C(\Gamma_! b, f^* d) &&\Gamma_! b\in \cat C^c \text{ is rigid}\\
		&\;\;\cong\;\; \cat B(b,\Gamma^* f^* d)&&
	\end{align*}
	of isomorphisms natural in $d$.
\end{proof}
\begin{Prop}\label{prop:preserves-products}
	The composite $\Gamma^*\circ f^* : \cat D \to \cat B$ preserves products, and hence has 
	a left adjoint.
\end{Prop}
\begin{proof}
	For any set of objects $\{d_i \}_{i\in I}$ in $\cat D$, let
	$\varphi: \Gamma^* f^*(\Pi_i d_i) \to \Pi_i \Gamma^* f^*(d_i)$
	denote the canonical map.
	By Yoneda, it suffices to prove that the natural map
\begin{equation}\label{eq:natb}
	\cat B(b,\Gamma^* f^*(\Pi_i d_i)) \xra{\varphi \circ -} \cat B(b,\Pi_i \Gamma^* f^* d_i)
\end{equation}
	is an isomorphism for all $b \in \cat B$.
	Since $\cat B$ is compactly generated, it suffices to check this for $b \in \cat B^c$.
	Now, for $b \in \cat B^c$ we have a chain of isomorphisms
	\begin{align*}
		\cat B(b,\Gamma^* f^* \Pi_i d_i) &\;\;\cong\;\; \cat D(\dual f_* \dual \Gamma_! b, \Pi_i d_i) &&\textrm{Lemma~}\ref{lem:adj}\\
		&\;\;\cong\;\; \Pi_i\cat D(\dual f_* \dual \Gamma_! b,d_i) && \\
		&\;\;\cong\;\; \Pi_i \cat B(b,\Gamma^* f^* d_i)&&\textrm{Lemma~}\ref{lem:adj} \\
		&\;\;\cong\;\; \cat B(b,\Pi_i \Gamma^* f^* d_i)&&
	\end{align*}
	which we claim coincides 
	with~\eqref{eq:natb}.
	This is a routine diagram chase using the naturality in $d$ of the 
	isomorphisms of Lemma~\ref{lem:adj}.
	The existence of the left adjoint then follows from Neeman's Brown Representability Theorem for compactly generated categories (see \eg~\cite[Cor.~2.3]{BalmerDellAmbrogioSanders16}).
\end{proof}
\begin{Not}
	We will denote the left adjoint of $\Gamma^* \circ f^* :\cat D \to \cat B$ by
	\[ (f \circ \Gamma)_! : \cat B \to \cat D.\]
\end{Not}
\begin{Rem}\label{rem:canon-iso}
	Lemma~\ref{lem:adj} implies that
	there is a canonical isomorphism
	${(f \circ \Gamma)_!b} \cong \dual f_* \dual \Gamma_! b$
	for all compact $b \in \cat B^c$.
	More explicitly, chasing the unit for the $(f \circ \Gamma)_! \dashv {\Gamma^*\circ f^*}$ adjunction
	through \eqref{eq:adj}
	we obtain 
	an
	isomorphism $\dual f_*\dual \Gamma_! b \to (f\circ \Gamma)_!b$
	whose inverse 
	$(f\circ \Gamma)_!b \to \dual f_*\dual \Gamma_!b$
	is the map adjoint to the map $b \to \Gamma^* f^* \dual f_* \dual \Gamma_! b$
	corresponding under 
	\eqref{eq:adj}
	to the
	identity of $\dual f_*\dual \Gamma_! b$.
\end{Rem}
\begin{Rem}\label{rem:defn-wirth}
Consider the natural transformation
	\[ \psi^{\text{ad}} : f^* \to \ihom(\DO,f^!(-)) \]
defined in Remark~\ref{rem:defn-of-psi}.
	Applying $\Gamma^*$ we get a natural transformation
	\[ \Gamma^*(\psi^{\text{ad}}) : \Gamma^*\circ f^*  \to \Gamma^* \ihom(\DO,f^!(-))\]
	of functors $\cat D \to \cat B$.
	Taking left adjoints it corresponds to a natural transformation
	\begin{equation}
		\varpi:		f_*(\Gamma_!(-) \otimes \DO) \to (f\circ \Gamma)_!
	\end{equation}
	of functors $\cat B \to \cat D$.
	A long but straight-forward unravelling of the definitions
	shows that it is given explicitly for any $x \in \cat B$ as the following composite:
	\begin{equation*}\label{eq:explicit-wirthmuller}
		\begin{gathered}
		\xymatrix @C=3em{
			f_*(\Gamma_!x \otimes \DO) \ar@{.>}[d]_{\varpi} \ar[r]^-{f_*(\Gamma_! \eta \otimes 1)} &
		 f_*(\Gamma_! \Gamma^* f^*(f \circ \Gamma)_!x \otimes \DO)  
		 \ar[r]^-{f_*(\eps \otimes 1)}& f_*(f^*(f\circ \Gamma)_!x \otimes \DO)
		 \ar[d]_-{\simeq}^-{\pi}\\
		 (f \circ \Gamma)_!x
		 &	 (f\circ \Gamma)_!x\otimes \unit \ar[l]_-{\simeq}& (f \circ \Gamma)_!x \otimes f_*\DO.
		 \ar[l]_-{1 \otimes \eps}
	 }
 \end{gathered}
 \end{equation*}
\end{Rem}
\begin{Prop}\label{prop:wirth-is-iso}
	The natural transformation
	$\varpi:f_*(\Gamma_!x \otimes \DO) \to (f \circ \Gamma)_!x$
	(Rem.~\ref{rem:defn-wirth})
	is a natural isomorphism for all $x \in \cat B$.
\end{Prop}
\begin{proof}
	Note that both functors preserve coproducts, being (composites of) left adjoints.
	Hence, since $\cat B$ is compactly generated, it suffices to check that 
	$\varpi$
	is a natural isomorphism
	for $x \in \cat B^c$.
	Then by Yoneda, it suffices to prove that
	\begin{equation}\label{eq:natc}
		\cat D(d,f_*(\Gamma_! x \otimes f^!\unit)) \xra{\varpi \circ -} \cat D(d,(f\circ \Gamma)_!x)
	\end{equation}
	is an isomorphism for all $d \in \cat D$ and all $x \in \cat B^c$.
	Now, for $x \in \cat B^c$, we have a chain of isomorphisms
	\begin{equation}\label{eq:chained-iso}
		\begin{aligned}
		\cat D(d,f_*(\Gamma_! x \otimes f^!\unit)) &\;\;\cong\;\; \cat C(f^*d,\Gamma_! x \otimes f^!\unit) && \\
		&\;\;\cong\;\; \cat C(f^*d\otimes \dual \Gamma_! x, f^!\unit) &&\Gamma_!x \in \cat C^c \text{ is rigid}\\
		&\;\;\cong\;\; \cat D(f_*(f^*d \otimes \dual \Gamma_! x),\unit) &&\\
		&\;\;\cong\;\; \cat D(d \otimes f_* \dual \Gamma_! x,\unit) && \text{projection formula} \\
		&\;\;\cong\;\; \cat D(d,\dual f_* \dual \Gamma_! x) && f_* \dual \Gamma_! x \in \cat D^c \text{ is rigid}\\
		&\;\;\cong\;\; \cat D(d,(f \circ \Gamma)_!x) &&\text{Remark~\ref{rem:canon-iso}}
	\end{aligned}
	\end{equation}
	which we claim coincides with \eqref{eq:natc}.
	Verifying this claim is quite involved, so for
	notational brevity, let us write 
	$\omega:=\DO=f^!\unit$ for the dualizing object,
	$g_! := (f \circ \Gamma)_!$ for the left adjoint,
	and set $c:=\Gamma_! x$.
	Then let $u \in \cat D(d,f_*(c \otimes \omega))$ and 
	consider the following diagram
	\[\xymatrix @C=0.3em @R=1.95em{
			d \ar[d]_-{u} \ar[r]^-{\text{coev}\otimes 1} & \dual f_* \dual c \otimes f_* \dual c \otimes d \ar[r]_-{\simeq}^-{1\otimes \pi} \ar[d]^{1\otimes 1 \otimes u} & \dual f_* \dual c \otimes f_*(\dual c \otimes f^* d) \ar[dd]^-{1\otimes f_*(1\otimes f^* u)} \\
			f_*(c\otimes \omega) \ar[d]_-{f_*(\Gamma_!\eta_x \otimes 1)} \ar[r]^-{\text{coev}\otimes 1} & \dual f_* \dual c \otimes f_* \dual c \otimes f_*(c \otimes \omega) \ar[dr]_-{\simeq}^-{1\otimes \pi} \ar@/_1pc/[ddr]_-{1\otimes\text{lax}} &  \\
			f_*(\Gamma_! \Gamma^* f^* g_! x \otimes \omega) \ar@{}[rrd]|-{(\dagger)} \ar[d]_-{f_*(\eps\otimes 1)} &&  \dual f_* \dual c \otimes f_*(\dual c \otimes f^* f_*(c \otimes \omega)) \ar[d]^-{1\otimes f_*(1\otimes \eps)}\\
			f_*(f^* g_! x \otimes \omega) \ar[d]_{\pi}^{\simeq} &&\dual f_* \dual c \otimes f_*(\dual c \otimes c \otimes \omega) \ar[d]^-{1\otimes f_*(\text{ev} \otimes 1)} \\
			g_!x \otimes f_* \omega \ar[d]_{1 \otimes \eps}\ar[rr]_{\simeq}^{\Psi \otimes 1} && \dual f_* \dual c \otimes f_* \omega \ar[d]^-{1 \otimes \eps} \\
			g_!x\ar[rr]_{\Psi}^{\simeq}  && \dual f_* \dual c 
		}\]
	where $\Psi$ denotes the isomorphism of Remark~\ref{rem:canon-iso}.
	Going along the top, right-hand side, and bottom we get the image of $u$ under the chained isomorphism~\eqref{eq:chained-iso} while going along the left-hand side
	we get the image of $u$ under \eqref{eq:natc} (cf.~Rem.~\ref{rem:defn-wirth}).
	The triangular region can be checked directly from the definition of the projection formula.
	What remains is to check the interior region~$(\dagger)$.
	This is a nightmarish diagram chase, so we will guide the reader through it.

	We begin by obtaining a description of the isomorphism $\Psi$.
	To this end, 
	let $\delta : x \to \Gamma^* f^*\dual f_*\dual c$ denote the morphism corresponding to the identity of $\dual f_* \dual c$ under the isomorphism~\eqref{eq:adj}. Chasing the identity through the definition of \eqref{eq:adj} we find that $\delta$ equals the following morphism
	\begin{equation}\label{eq:bleh}
		x \xra{\eta} \Gamma^* c \simeq \Gamma^*(\unit \otimes c) \xra{\Gamma^*(\zeta \otimes 1)} 
		\Gamma^*(f^*\dual f_*\dual c \otimes \dual c \otimes c) \xra{1\otimes \text{ev}} \Gamma^*f^*\dual f_*\dual c
	\end{equation}
	where $\zeta : \unit \to f^*\dual f_*\dual c \otimes \dual c$ is defined by
	\[ \zeta:\unit \simeq f^*\unit \xra{\text{coev}} f^*(\dual f_*\dual c\otimes f_*\dual c)\simeq f^*f_*(f^*\dual f_*\dual c\otimes \dual c)\xra{\eps}f^*\dual f_*\dual c\otimes\dual c.\]
	Let $\alpha_a:\Gamma^*a \otimes x \to \Gamma^*(a\otimes c)$ denote the map
	\[\alpha_a : \Gamma^* a \otimes x \xra{1\otimes \eta_x} \Gamma^*a \otimes \Gamma^*\Gamma_! x = \Gamma^*a\otimes \Gamma^*c \simeq \Gamma^*(a\otimes c)\]
	which is natural in $a$.  Then by applying naturality of $\alpha$ with respect to $\zeta$ one readily sees that \eqref{eq:bleh} coincides with the following morphism
\[x\simeq \Gamma^*\unit \otimes x \stackrel{\Gamma^*\zeta \otimes 1}{\longrightarrow} \Gamma^*(f^*\dual f_*\dual c \otimes \dual c)\otimes x \xra{\alpha} \Gamma^*(f^*\dual f_*\dual c \otimes \dual c \otimes c) \stackrel{1\otimes\text{ev}}{\longrightarrow} \Gamma^* f^*\dual f_*\dual c\]
where we have abbreviated $x\simeq \unit\otimes x \simeq \Gamma^*\unit\otimes x$.
Now, by definition, the isomorphism $\Psi:g_! x \xra{\sim} \dual f_* \dual c$ of Remark~\ref{rem:canon-iso} is given by 
\[g_!x \xra{g_!(\delta)} g_!(\Gamma^*f^*\dual f_*\dual c) \xra{\eps} \dual f_*\dual c.\]
Using our second description of $\delta$ (and consequently of $\Psi$) we 
check that
the bottom-left edge of $(\dagger)$ coincides with 
	\[\xymatrix {
			f_*(\Gamma_! x \otimes \omega)
			\simeq f_*(\Gamma_!(\Gamma^* \unit \otimes x)\otimes \omega)
			\ar[r]^-{\zeta} &
			f_*(\Gamma_!(\Gamma^* (f^* \dual f_* \dual c \otimes \dual c)\otimes x) \otimes \omega ) \ar[d]^{\alpha} \\
			& f_*(\Gamma_! \Gamma^*(f^* \dual f_*\dual c \otimes \dual c \otimes c) \otimes \omega) \ar[d]^{1\otimes\text{ev}} \\
			& f_*(\Gamma_! \Gamma^*(f^* \dual f_* \dual c) \otimes \omega) \ar[d]^{\eps}\\
			& f_*(f^* \dual f_* \dual c \otimes \omega) \ar[d]_{\simeq}^{\pi} \\
			& \dual f_* \dual c \otimes f_* \omega.
		}\]
	We are now in a position to check the commutativity of $(\dagger)$.
	Define
	\[\beta : \Gamma_!(\Gamma^* a \otimes b) \xra{\text{colax}} \Gamma_! \Gamma^* a \otimes \Gamma_! b \xra{\eps\otimes 1} a \otimes \Gamma_! b.\]
	The commutativity of $(\dagger)$ can now be outlined as follows
	\begin{equation}\label{eq:outline}
		\begin{gathered}
		\xymatrix{
			&&&&\bullet\ar[r]^-{\text{ev}}& \bullet\\
			\bullet \ar@/^1pc/[rrrru]^{1\otimes\text{lax}} \ar[r]^{\pi} & \bullet \ar[r]^{\pi} & \bullet \ar[r]^{\eps} & \bullet \ar@{=}[r] & \bullet \ar[r]^{\text{ev}} \ar[u]_{\pi} & \bullet \ar[u]_{\pi} \\
			\bullet \ar[r]_{\text{coev}} \ar@/_2pc/[rrr]_-{\zeta} \ar[u]^{\text{coev}} & \bullet \ar[u]_{\beta} \ar[r]_{\pi} & \bullet \ar[u]_{\beta} \ar[r]_{\eps} &\bullet \ar[u]_{\beta} \ar[r]_{\alpha} &\bullet \ar[u]_{\eps} \ar[r]_{\text{ev}}\ar[u]_{\eps} &\bullet \ar[u]_{\eps} 
		}
	\end{gathered}
	\end{equation}
	where for example the first square is
	\[\xymatrix @C=2em{
			\dual f_*\dual c \otimes f_*\dual c \otimes f_*(c\otimes \omega) \ar[r]^{\pi} & f_*(f^*(\dual f_*\dual c \otimes f_*\dual c)\otimes c\otimes \omega)\\
			f_*(\Gamma_!x\otimes \omega)\simeq f_*(\Gamma_!(\Gamma^*f^*\unit \otimes x)\otimes \omega) \ar@<40pt>[u]^{\text{coev}\otimes 1} \ar[r]^-{\text{coev}} & f_*(\Gamma_!(\Gamma^*f^*(\dual f_* \dual c \otimes f_* \dual c)\otimes x)\otimes \omega) \ar@<-40pt>[u]_{\beta}
		}\]
	and the second square is
	\[\xymatrix @C=.5em{
			f_*(f^*(\dual f_*\dual c \otimes f_*\dual c)\otimes c\otimes \omega) \ar[r]^{\pi} & f_*(f^*f_*(f^*\dual f_*\dual c \otimes \dual c)\otimes c\otimes\omega)\\
			f_*(\Gamma_!(\Gamma^*f^*(\dual f_* \dual c \otimes f_* \dual c)\otimes x)\otimes \omega) \ar[u]^-{\beta} \ar@{}[r]|-{\simeq}^-{\pi} & f_*(\Gamma_!(\Gamma^* f^*f_*(f^*\dual f_*\dual c \otimes \dual c) \otimes x)\otimes \omega)\ar[u]_{\beta}.
		}\]
%
The remaining squares should then be clear by following the guide~\eqref{eq:outline}.
The commutativity of the second, third, fifth, and sixth squares follows immediately from naturality, while the commutativity of the fourth square follows from the definitions of $\alpha$ and $\beta$.
The first square requires a diagram chase but is fairly routine.
Finally, the curved portion can be checked using the definition of the projection formula.
This completes the proof, modulo the details we have only sketched.
\end{proof}

\begin{proof}[Proof of Theorem~\ref{thm:main-thm}]
	Part (b) has been proved in Proposition~\ref{prop:b-iso}.
	It follows immediately (or directly from Lemma~\ref{lem:preserves-coproducts}) that (a) holds.
	Part (d) has been proved in Proposition~\ref{prop:preserves-products} and Proposition~\ref{prop:wirth-is-iso}.
	Part (c) then follows 
 (Rem.~\ref{rem:defn-wirth})
	by taking right adjoints of the isomorphism~\eqref{eq:generalized-wirthmuller} in
	Part (d).
\end{proof}

\bigbreak
\section{The Adams isomorphism as a Wirthm\"{u}ller isomorphism}
\label{sec:adams-wirth}

The purpose of this section is to show that the Adams isomorphism in equivariant stable homotopy theory can be realized as a special case of the formal Wirthm\"{u}ller isomorphism of Theorem~\ref{thm:main-thm}.  To this end, let $\SHG$ denote the genuine $G$-equivariant stable homotopy category for a compact Lie group~$G$.  It is a rigidly-compactly generated tensor-triangulated category which is generated by the orbits $\Sigma^\infty G/H_+$ as $H\le G$ ranges over all closed subgroups.  
(All subgroups are closed unless otherwise noted.)
Any homomorphism of compact Lie groups $f:G\to G'$ induces a
tensor-triangulated functor $f^*:\SH(G')\to\SH(G)$.  For example, the inclusion
$f:H \hookrightarrow G$ of a 
subgroup induces the restriction functor
$\Res_H^G:=f^*:\SH(G)\to\SH(H)$.  This geometric functor satisfies GN-duality,
and the resulting Wirthm\"{u}ller isomorphism is the eponymous classical
Wirthm\"{u}ller isomorphism (cf.~\cite{May03} and
\cite[Ex.~4.5]{BalmerDellAmbrogioSanders16}).  On the other hand, the quotient
$f:G\to G/N$ by a 
normal subgroup $N \lenormal G$ induces the inflation
functor $\Infl_{G/N}^G:=f^*:\SH(G/N)\to\SH(G)$.  This is again a geometric
functor whose right adjoint $\lambda^N := f_*:\SH(G)\to\SH(G/N)$ is the
so-called categorical \mbox{$N$-fixed} point functor.  By the usual formal story, it
itself admits a right adjoint $f^!:\SH(G/N)\to\SH(G)$ which, to the author's
knowledge, has not yet been studied
(cf.~Rem.~\ref{rem:right-adjoint-to-fixed-points}).  On the other hand, as
Prop.~\ref{prop:infl-doesnt-satisfy-gn-duality} below shows, $\Infl_{G/N}^G$
does not satisfy GN-duality (\ie does not have a left adjoint) except in the
trivial case when $N=1$.  To prove this, we will need the following standard
facts about group cohomology:

\begin{Prop}\label{prop:group-coho}
	Let $G$ be a compact Lie group.
	Then $H^*(BG;k)$ is not concentrated in finitely many degrees in either of the following cases:
	\begin{enumerate}
		\item $G$ is finite and $k=\Fp$ with $p$ dividing the order of $G$;
		\item $G$ is not finite and $k=\mathbb Q$.
		\end{enumerate}
\end{Prop}
\begin{proof}
These facts follow from the Evens-Venkov Theorem, graded Noether normalization (\eg~\cite[Thm.~4.2.3]{HunekeSwanson06}), and Quillen's theorem that the Krull dimension of the cohomology ring $H^*(BG;k)$ is equal to the $p$-rank of $G$, where $p=\chara k$.
(By definition, the $0$-rank of $G$ is the rank of a maximal torus in~$G$.)
\end{proof}

\begin{Prop}\label{prop:infl-doesnt-satisfy-gn-duality}
	Let $1 \neq N \lenormal G$ be a nontrivial 
	normal subgroup of a compact Lie group $G$.
	Then the categorical fixed point functor $\lambda^N:\SH(G) \to \SH(G/N)$
	does not preserve compact objects.
\end{Prop}
\begin{proof}
	According to the tom Dieck splitting theorem of \cite[Cor.~3.4 (a)]{Lewis00},
	the $N$-fixed points of the unit $\lambda^N(\unit)$ splits as a wedge of $G/N$-spectra
 indexed over the $G$-conjugacy classes of subgroups $H \le N$.
 The term for $H=1$ is
	\begin{equation}\label{eq:tomdieck-summand}
		\Sigma^\infty_{G/N}\left( E\cat F(N;G)_+ \smashh_N S^{\Ad(N;G)}\right)
	\end{equation}
 where $E\cat F(N;G)$ denotes the universal $N$-free $G$-space
	and $\Ad(N;G)$ denotes the adjoint representation that conjugation by $G$ 
	induces on the Lie algebra of $N$.
	Similarly,
	for $H = N_0$, the connected component of the identity of $N$,
	the corresponding term is
	\begin{equation}\label{eq:tomdieck-summandb}
	\Sigma^\infty_{G/N}\left( E\cat F(N/N_0;G/N_0)_+ \smashh_{N/N_0} S^{\Ad(N/N_0,G/N_0)}\right).
\end{equation}
Thus, if $\lambda^N(\unit)$ is compact then so are the two $G/N$-spectra
\eqref{eq:tomdieck-summand} and \eqref{eq:tomdieck-summandb}.
Since restriction $\SH(G/N)\to\SH$ preserves compactness,
	it would follow that
	\[ \Sigma^\infty( EN_+ \smashh_N S^{\Ad(N)} )
		\qquad\text{and}\qquad
		\Sigma^\infty B(N/N_0)_+\]
	are compact in $\SH$.
	If $N$ is not connected then $N/N_0$ is a nontrivial finite group and hence $\Sigma^\infty B(N/N_0)_+$ is not compact in $\SH$.
	Otherwise the cohomology groups
	\[ H^*(B(N/N_0);\Fp)\cong \tilde{H}^*(B(N/N_0)_+;\Fp) \cong H\Fpstar(\Sigma^\infty B(N/N_0)_+) \]
	would be concentrated in only finitely many degrees, contradicting Proposition~\ref{prop:group-coho} for any prime $p$ dividing the order of $N/N_0$.
	On the other hand, if $N$ is connected 
	then
$\Sigma^\infty( EN_+ \smashh_N S^{\Ad(N)} )$ 
is not
compact in $\SH$. 
Otherwise
the reduced Borel cohomology
\[ \tilde{H}_{N}^*(S^{\Ad(N)};\bbQ) \cong \tilde{H}^*(EN_+ \smashh_N S^{\Ad(N)};\mathbb Q) 
	\cong H\mathbb{Q}^*(\Sigma^\infty EN_+ \smashh_N S^{\Ad(N)})\]
	would be concentrated in finitely many degrees.
	To see that this cannot be the case, note
	that $EN_+ \smashh_N S^{\Ad(N)}$ is the Thom space
	for the vector bundle over $BN$ associated to the adjoint representation of $N$.
	This vector bundle is orientable since~$BN$ is simply connected.
Hence the Thom isomorphism gives
\[\tilde{H}^*(EN_+ \smashh_N S^{\Ad(N)};\bbQ)\cong H^{*-\dimm(N)}(BN;\bbQ)\]
and we invoke Proposition~\ref{prop:group-coho} again.
\end{proof}

\begin{Rem}
	Since $\Infl_{G/N}^G$ does not satisfy GN-duality, it is an excellent candidate for Theorem~\ref{thm:main-thm}.
	Let's first set the stage by recalling some additional notions from equivariant homotopy theory.
\end{Rem}

\begin{Rem}\label{rem:thick-is-tensor-ideal}
	Let $\cat F$ be a family of subgroups of $G$, closed under conjugation and taking subgroups.
	Then we can consider the thick subcategory 
	\begin{equation}\label{eq:family-thick}
		\thick\langle \Sigma^\infty G/K_+ \mid K \in \cat F\rangle \subset \SHGc
	\end{equation}
	generated by those orbits having isotropy in $\cat F$.
	The fact that $\cat F$ is closed under subconjugation implies that \eqref{eq:family-thick}
	is actually a thick tensor-ideal of $\SHGc$.
	Indeed, it suffices to check that
$\Sigma^\infty G/H_+ \smashh \Sigma^\infty G/K_+ \cong \Sigma^\infty (G/H \times G/K)_+$
	is contained in \eqref{eq:family-thick} for $K \in \cat F$ and $H \le G$ arbitrary.
	For $G$ finite this follows immediately from the double-coset decomposition of the product $G/H \times G/K$.
	For $G$ arbitrary, we invoke a theorem of Illman \cite{Illman83} that smooth compact $G$-manifolds
	are finite $G$-CW-complexes.
	We conclude that $G/H\times G/K$ is built out of finitely many cells $G/L \times D^n$ with $L \in \cat F$
	since each such cell contributes a point with isotropy $L$ (and the isotropy of a point in $G/H \times G/K$ is the intersection of a conjugate of $H$ with a conjugate of $K$).
\end{Rem}
\begin{Rem}
	The idempotent triangle for the finite localization (Rem.~\ref{rem:finite-localization})
	associated to the thick tensor-ideal \eqref{eq:family-thick} is nothing but 
	the exact triangle obtained by applying $\Sigma^\infty$ to the isotropy cofiber sequence
	$E\cat F_+ \to S^0 \to \tildeE\cat{F}$. 
	A $G$-spectrum $X$ is colocal, that is $E\cat F_+ \smashh X \simeq X$, if and only if $\Phi^H X = 0$ for all $H \not\in \cat F$,
	while a $G$-spectrum $X$ is local, that is $\tildeE\cat F\smashh X \simeq X$, if and only if $\Phi^H X = 0$ for all $H \in \cat F$.
	(Here we have dropped the suspensions, as usual, for readability.)
\end{Rem}

\begin{Exa}\label{exa:N-free}
	For $N \lenormal G$ a normal subgroup we can take the family $\cat F(N) := \SET{ {K \le G} }{K \cap N = 1}$.
The colocal objects are the so-called \emph{$N$-free $G$-spectra}, \ie those $G$-spectra $X$
such that $\Phi^H(X) = 0$ if $H \cap N\neq 1$.
\end{Exa}
\begin{Exa}\label{exa:N-concentrated}
	For $N \lenormal G$ a normal subgroup we can take the family
	${\cat F[\notsupseteq N]} := \SET{K \le G}{K \not\supseteq
		N}$.  The local objects are the so-called
	\emph{$N$-concentrated $G$-spectra}, \ie those $G$-spectra $X$
such that $\Phi^H(X) = 0$ unless $H \supseteq N$.
\end{Exa}

\begin{Lem}\label{lem:free-are-compact}
	Let $N \lenormal G$ be a 
	normal subgroup of a compact Lie group $G$.
	For any compact $N$-free $G$-spectrum $X \in \thick\langle \Sigma^\infty G/K_+ \mid K \cap N = 1 \rangle \subset \SHG^c$,
	the $N$-fixed point spectrum $\lambda^N(X)$ is compact in $\SH(G/N)$.
\end{Lem}
\begin{proof}
	The collection of $X \in \SH(G)^c$ such that $\lambda^N(X) \in \SH(G/N)^c$
	is a thick subcategory, so it suffices to check that $\lambda^N(\Sigma^\infty G/K_+) \in \SH(G/N)^c$
	whenever $K \cap N = 1$.
	For any based $G$-space $Y$, the tom Dieck splitting theorem (cf.~\cite[Corollary~3.4 (a)]{Lewis00}) provides an isomorphism
	\begin{eqnarray*}
		 \lambda^N(\Sigma_G^\infty Y) \cong \bigvee_{(H) \subseteq N} \Sigma^\infty_{G/N}\; G/N \ltimes_{\cat WH}(E\cat F(W_NH;W_GH)_+ \smashh_{W_NH} \Sigma^{\Ad{(W_NH)}} Y^H)
	 \end{eqnarray*}
	in $\SH(G/N)$.
	Here the wedge is indexed over the $G$-conjugacy classes of subgroups $H \leq N$.
	Note that $W_N H \lenormal W_G H$ and $\cat W H := W_GH / W_N H \hookrightarrow G/N$ embeds as a subgroup.
	Moreover, $E\cat F(W_NH;W_GH)$ denotes, as usual, the universal $W_NH$-free $W_GH$-space
	and $\Ad(W_NH)$ denotes the adjoint representation that $W_GH$ induces on the Lie algebra of its normal subgroup $W_N H$.
	However, $(G/K_+)^H = *$ if $H \not\le_G K$, so the only nonzero summands in the tom Dieck splitting
	for an orbit $Y=G/K_+$ are for those conjugacy classes of subgroups $H \leq K \cap N$.
	Thus, if $K \cap N=1$, then it collapses to
	\begin{equation}\label{eq:tom-dieck-G/K}
		\lambda^N(\Sigma^\infty_G \; G/K_+) \cong \Sigma^\infty_{G/N}\big( E\cat F(N;G)_+ \smashh_N \Sigma^{\Ad(N;G)} G/K_+ \big).
	\end{equation}
	Moreover, the space $G/K_+$ is $N$-free under our assumption $K \cap N=1$, so~\eqref{eq:tom-dieck-G/K} is just
	\[ \lambda^N(\Sigma^\infty_G G/K_+) \cong \Sigma^\infty_{G/N} \big( (G/K_+ \smashh S^{\Ad(N;G)})/N \big)\]
	which is certainly compact in $\SH(G/N)$.
	Indeed, (a) the $G$-space $G/K_+ \smashh S^{\Ad(N;G)}$ has the homotopy type of a finite $G$-CW-complex, so its suspension $G$-spectrum (indexed on any $G$-universe) is compact, (b) the space-level and spectrum-level orbit functors intertwine the suspension functors to $G/N$-spectra and to $G$-spectra indexed on an $N$-trivial $G$-universe, and (c) the spectrum-level orbit functor (defined for $G$-spectra indexed on an $N$-trivial $G$-universe) preserves compact objects (as it is an exact functor between compactly generated categories with a double right adjoint).
\end{proof}
\begin{Exa}\label{exa:adams}
	Let $N \lenormal G$ be a 
	normal subgroup of a compact Lie group $G$
	and let $f^*:=\Infl_{G/N}^G : \SH(G/N) \to \SH(G)$.
	By Lemma~\ref{lem:free-are-compact} and Remark~\ref{rem:thick-is-tensor-ideal}, 
	the subcategory
	\[\cat I := \thick\langle G/H_+ \mid H\cap N=1\rangle \subset \SHGc\]
	is a thick tensor-ideal of compact $G$-spectra,
	such that $\lambda^N(\cat I) \subset \SH(G/N)^c$,
	and we can apply Theorem~\ref{thm:main-thm}.
	Then $\Gamma^*\cong E\cat F(N)_+ \smashh - : \SH(G) \to N\text{-free-}\SH(G) := \Loc\langle G/H_+ \mid H \cap N =1 \rangle$
	is the colocalization onto the subcategory of genuine \mbox{$N$-free} $G$-spectra.
	As we now prove, the associated Wirthm\"{u}ller isomorphism~\eqref{eq:generalized-wirthmuller} can be identified as a natural isomorphism
	\[ (i^*X \smashh E\cat F(N)_+)/N \cong i^*(X \smashh \DO)^N \]
	defined for all $N$-free $G$-spectra $X$, where $i:\U^N\to \U$ denotes the inclusion of the $N$-fixed points of a complete $G$-universe $\U$. This requires some explanation.

	Recall from \cite{LewisMaySteinbergerMcClure86}
	that associated to $i:\U^N \to \U$ is the change of universe adjunction
	\[i_* : \SH_{\U^N}(G) \adjto \SH_\U(G) : i^*\]
	which unfortunately clashes with our own notational conventions, since here $i_*$ is left adjoint to $i^*$.
	We continue to write $\SH(G)=\SH_\U(G)$ for the stable homotopy category of genuine $G$-spectra,
	but also write $\SH_{N\text{-triv}}(G):=\SH_{\U^N}(G)$ for the stable homotopy category of $G$-spectra indexed on the $N$-trivial $G$-universe $\U^N$.
Regarding $\U^N$ as a complete $G/N$-universe, we have the change of group functor 
	\[\eps^*:\SH(G/N)\cong\SH_{\U^N}(G/N) \to \SH_{\U^N}(G)=\SH_{N\text{-triv}}(G)\]
which admits adjoints on both sides
	\[\xymatrix{
			\SH(G/N) \ar[d]^{\eps^*} \\
			\SH_{N\text{-triv}}(G) \ar@<15pt>[u]^-{(-)/N} \ar@<-15pt>[u]_-{(-)^N}
		}\]
	the $N$-orbits and the $N$-fixed points.
	By definition, the inflation functor is the composite $\Infl_{G/N}^G := i_*\circ \eps^*:\SH(G/N)\to\SH(G)$,
	while the categorical $N$-fixed point functor is the composite $\lambda^N:=(i^*(-))^N:\SH(G)\to\SH(G/N)$.

	Since $i_*\circ\Sigma^\infty_{\U^N} \cong \Sigma^\infty_\U$ and $i_*$ is tensor-functor,
	we see that if $X \in \SH_{N\text{-triv}}(G)$ is $N$-free
	then $i_*X$ is also $N$-free. (In contrast, $i^*$ does not preserve $N$-free spectra in general.)
	Thus $i_*:\SH_{N\text{-triv}}(G)\to\SH(G)$ restricts to a functor on the two categories of $N$-free spectra:
	$i_*:N\text{-free-}\SH_{N\text{-triv}}(G) \to N\text{-free-}\SH(G)$.
	The crucial fact about $N$-free $G$-spectra is that this functor is an equivalence (\cite[Theorem~2.8]{LewisMaySteinbergerMcClure86}).
	A quasi-inverse is given by $E\cat F(N)_+ \smashh i^*(-)$.
	Then contemplate the following diagram:
		\begin{equation}\label{eq:big-figure}
			\begin{gathered}
			\xymatrix @C=1em @R=1em{
				& \SH(G/N) \ar@<-16pt>@{<-}[dd]_-{-/N} \ar@{}@<-8pt>[dd]|-{\dashv} \ar[dd]^-{\eps^*} \ar@/^1pc/[ddddrrr]^-{\Infl_{G/N}^G} &&& \\
		&&&&\\
		&\SH_{N\text{-triv}}(G)\ar@{}@<-8pt>[dd]|-{\dashv} \ar[dd]^(.6){E\cat F(N)_+\smashh-} \ar[ddrrr]^-{i_*}&&&\\
		&&&&\\
		&N\text{-free-}SH_{N\text{-triv}}(G) \ar@<16pt>@{^{(}->}[uu] \ar[ddd]^-{i_*}&&& \SH(G) \ar@/^1pc/[dddlll]^-{E\cat F(N)_+ \smashh-}\\
		*+[r]{\SH_{N\text{-triv}}(G)}  \ar[ur]+<-30pt,-5pt>^-(.25){E\cat F(N)_+\smashh -} &&&&\\
		*+[r]{\SH(G)}\ar@<-10pt>[u]^-{i^*} \ar@{<-^{)}}[dr]+<-35pt,7pt>&&&&\\
		& N\text{-free-}SH(G) \ar@<32pt>@{}[uuu]|-{(\dagger)}  &&
	}
\end{gathered}
\end{equation}
	The two regions on the right commute, while the parallel arrows in the middle 
	are adjoints.
	Here $(\dagger)$ is the adjoint equivalence we just mentioned.
	In summary, we find that $(E\cat F(N)_+ \smashh i^*(-))/N: N\text{-free-}\SH(G) \to \SH(G/N)$ is left
	adjoint to $E\cat F (N)_+ \smashh \Infl_{G/N}^G = \Gamma^* \circ f^*$.
	Thus, the Wirthm\"{u}ller isomorphism~\eqref{eq:generalized-wirthmuller} takes the form 
	\begin{equation}\label{eq:adams-final}
		(i^*X \smashh E\cat F(N)_+)/N \cong i^*(X \smashh \DO)^N.
	\end{equation}
	That is, it is an isomorphism between the $N$-fixed points and the $N$-orbits of an $N$-free $G$-spectrum, up to a twist by the
	relative dualizing object $\DO$.
\end{Exa}

\begin{Rem}\label{rem:almost-dualizing}
	This isomorphism~\eqref{eq:adams-final} is the Adams isomorphism.
	As explained in Remark~\ref{rem:not-formal}, 
	it is by no means formal to give an ``explicit'' description of the relative dualizing object $\DO$.	
In our present case, however, we can apply what is classically known about the Adams isomorphism (\eg~\cite[Thm.~II.7.1]{LewisMaySteinbergerMcClure86}) 
via the method of Remark~\ref{rem:not-formal} to conclude that 
	\begin{equation}\label{eq:dualizing-up-to-e}
	E\cat F(N)_+ \smashh \DO \cong E\cat F(N)_+ \smashh S^{-\Adhspace\Ad(N;G)}
	\end{equation}
	where $S^{-\Adhspace\Ad(N;G)}$ is the inverse in $\SH(G)$ of the representation sphere for the adjoint \mbox{$G$-representation $\Ad(N;G)$}.
	It follows that
	$X \smashh \DO \cong X \smashh S^{-\Adhspace\Ad(N;G)}$ for any $N$-free \mbox{$G$-spectrum} $X$,
	so the isomorphism~\eqref{eq:adams-final} provided by Theorem~\ref{thm:main-thm} is indeed the Adams isomorphism
	\[ (i^*X \smashh E\cat F(N)_+)/N \cong i^*(X \smashh S^{-\Adhspace\Ad(N;G)})^N \]
	between $N$-fixed points and $N$-orbits up to a twist by the adjoint representation.
	However, it does not follow from \eqref{eq:dualizing-up-to-e} that $\DO \cong S^{-\Adhspace\Ad(N;G)}$ (\ie~before applying 
	the colocalization $E\cat F(N)_+\smashh -$).
	It remains an interesting challenge to obtain an explicit description of the
	relative dualizing object of inflation and, more generally, of the right adjoint
	$f^!$ to categorical fixed points. 
	(See Remark~\ref{rem:right-adjoint-to-fixed-points} for more on this issue.)
\end{Rem}
\bigbreak
\section{The compactness locus of a geometric functor}
\label{sec:locus}
The statement of Theorem~\ref{thm:main-thm} involved the choice of a certain thick tensor-ideal~$\cat I \subset \cat C^c$.
Let us clarify the dependence on this choice.
\begin{Lem}\label{lem:nested}
	Let $\cat I_1 \subset \cat I_2$ be two thick tensor-ideals of $\cat C^c$
	such that $f_*(\cat I_2) \subset \cat D^c$. Then 
the inclusion
	$\cat B_1 := \Loc(\cat I_1) \hookrightarrow \Loc(\cat I_2) =: \cat B_2$
	has a right adjoint, so that the colocalization $\Gamma_1^*:\cat C \to \cat B_1$
	factors through the colocalization $\Gamma_2^*:\cat C \to \cat B_2$. 
	It follows that we have a natural isomorphism
	\begin{equation}\label{eq:nest}
 (f\circ \Gamma_1)_!(x) \cong (f\circ \Gamma_2)_!(x)
	\end{equation}
for $x \in \cat B_1$. 
Moreover, under this isomorphism,
the Wirthm\"{u}ller isomorphism 
associated to $\cat I_1$ coincides with
the Wirthm\"{u}ller isomorphism associated to $\cat I_2$ restricted to $\cat B_1$.
\end{Lem}
\begin{proof}
	That the two colocalizations nest is standard (consider the two associated idempotent triangles)
	and we thus get the isomorphism~\eqref{eq:nest}. The real thing that needs to be checked is that
	the two Wirthm\"{u}ller isomorphisms coincide --- {\sl i.e.}, in pedantic notation, that
\[\xymatrix{
		f_*( (\Gamma_1)_!x \otimes \DO) \ar[r]^{\cong} \ar@{=}[d] &(f \circ \Gamma_1)_!(x) \ar[d]^{\cong} \\
		f_*( (\Gamma_2)_!(\gamma_! x) \otimes \DO) \ar[r]^{\cong} & (f \circ \Gamma_2)_!(\gamma_! x)	
	}\]
commutes, where $\gamma_!$ denotes the inclusion $\cat B_1 \hookrightarrow \cat B_2$.
	This is a straight-forward but lengthy diagram-chase
	utilizing the 
	description of the Wirtm\"{u}ller isomorphism 
	in Remark~\ref{rem:defn-wirth} and 
	how the various isomorphisms between adjoints are defined in terms of the units and counits.
\end{proof}
\begin{Rem}
Since the Wirthm\"{u}ller isomorphism obtained by applying
Theorem~\ref{thm:main-thm} to the subcategory $\cat I_1$ is completely contained in what we obtain by
applying the theorem to the larger subcategory $\cat I_2$,
we really are only interested in maximal such~$\cat I$.
Well, there is a unique maximal such subcategory: 
\end{Rem}
\begin{Def}\label{def:rel-comp}
Let $f^* : \cat D \to \cat C$ be a geometric functor between rigidly-compactly generated tensor-triangulated categories.
The \emph{subcategory of $f$-relatively compact objects} is the thick tensor-ideal
of $\cat C^c$ 
defined as follows:
\[ \A_f := \SET{ x \in \cat C^c }{ f_*(x \otimes y) \in \cat D^c \text{ for all } y \in \cat C^c}. \]
It is the largest thick tensor-ideal of $\cat C^c$ sent under $f_*$ to $\cat D^c$.
\end{Def}
\begin{Rem}\label{rem:only-one}
	According to Lemma~\ref{lem:nested},
	we should really only apply Theorem~\ref{thm:main-thm}
	to this specific 
	thick tensor-ideal
	$\cat I:=\A_f$ which is canonically determined by $f^*$.
	We have stated (and proved) the theorem in terms of a chosen subcategory $\cat I$ which could be smaller than $\A_f$ for added flexibility in situations where $\A_f$
	is not yet explicitly understood.
\end{Rem}
\begin{Rem}
	The thick tensor-ideal $\A_f$ corresponds to a certain Thomason subset $\Z_f \subset \Spc(\cat C^c)$ which we now single out for special attention.
\end{Rem}
\begin{Def}\label{def:locus}
	Let $f^*:\cat D \to \cat C$ be a geometric functor between rigidly-compactly generated tensor-triangulated categories.
	The \emph{\locus{}} of the functor $f^*$ is the Thomason subset
		\[\Z_f := \bigcup_{x \in \A_f} \supp(x) \subset \Spc(\cat C^c)\]
		corresponding to the thick tensor-ideal $\A_f \subset \cat C^c$
		(Def.~\ref{def:rel-comp}).
\end{Def}

\begin{Rem}
	The \locus{} $\Z_f \subset \Spc(\cat C^c)$ is a new topological invariant of the functor $f^*$
	and the remainder of this paper is devoted to its study.
	It is the whole space $\Spc(\cat C^c)$ precisely when $f^*$ satisfies GN-duality (cf.~Thm.~\ref{thm:original-thm}),
	but we would like to obtain a more refined and geometric understanding of $\Z_f$ in general.
	Before studying specific examples, we end this section with an alternative characterization 
	of the $f$-relatively compact objects:
\end{Rem}
\begin{Lem}\label{lem:alt-rel-compact-charac}
	Let $f^*:\cat D \to \cat C$ be a geometric functor between rigidly-compactly generated tensor-triangulated categories.
	A compact object $x \in \cat C^c$ is contained in
	the subcategory $\A_f\subset\cat C^c$
	(Def.~\ref{def:rel-comp}) if and only if 
	the functor $\ihom(x,f^!(-))\cong \dual x\otimes f^!(-)$ preserves coproducts.
\end{Lem}
\begin{proof}
	Indeed, we have
	\begin{align*}
		x \in \A_f &\;\;\Leftrightarrow\;\; f_*(x\otimes y) \text{ is compact in $\cat D$ for all } y \in \cat C^c\\
		&\;\;\Leftrightarrow\;\; \cat D(f_*(x \otimes y),-) \text{ preserves coproducts for all } y \in \cat C^c\\
		&\;\;\Leftrightarrow\;\; \cat C(x\otimes y,f^!(-)) \text{ preserves coproducts for all } y \in \cat C^c\\
		&\;\;\Leftrightarrow\;\; \cat C(y,\ihom(x,f^!(-))) \text{ preserves coproducts for all } y \in \cat C^c\\
		&\;\;\Leftrightarrow\;\; \ihom(x,f^!(-)) \text{ preserves coproducts}
	\end{align*}
	where the only claim which is not immediate is the last $\Rightarrow$.
To prove this, by Yoneda, we need only check that
	post-composition
	\[\cat C(y,\amalg_i \ihom(x,f^!(z_i))) \xra{\varphi\circ -} \cat C(y,\ihom(x,f^!(\amalg_i z_i)))\]
	by the canonical map $\varphi:\amalg_i \ihom(x,f^!(z_i)) \to \ihom(x,f^!(\amalg_i z_i))$ is an isomorphism for all $y\in \cat C$.
	Since $\cat C$ is compactly generated, we need only check this for $y\in \cat C^c$ compact.
	But one can directly check that we have a factorization
	\[\xymatrix{
			{\cat C(y,\amalg_i \ihom(x,f^!z_i))} \ar[r]^-{\varphi \circ -} & {\cat C(y,\ihom(x,f^! \amalg_i z_i))} \\
			{\amalg_i \cat C(y,\ihom(x,f^!z_i))} \ar[ur]\ar[u] &
		}\]
	where the diagonal and vertical maps are isomorphisms (for ${y \in \cat C^c}$).
\end{proof}
\bigbreak
\section{The compactness locus of a finite localization}
\label{sec:finite}

We begin our study of the \locus{} by 
considering
it for
smashing and finite localizations.
Recall (\eg~from \cite[Sec.~3.3]{HoveyPalmieriStrickland97})
that a smashing localization of a rigidly-compactly generated category $\cat C$ is rigidly-compactly generated and the
localization functor $\cat C \to \cat C_L$ is a geometric functor.
\begin{Prop}\label{prop:smashing-localization}
	A smashing localization $\cat C \to \cat C_L$ satisfies GN-duality if and only if
	it is the finite localization associated to an open and closed subset $Y \subseteq \Spc(\cat C^c)$.
\end{Prop}
\begin{proof}
	$(\Rightarrow)$	Let $e \to \unit \to f \to \Sigma e$ be the idempotent triangle for the smashing localization.
	The right adjoint of localization is just the inclusion $\cat C_L \hookrightarrow \cat C$ of the local objects $\cat C_L = f \otimes \cat C$.
	So GN-duality implies that the right idempotent~$f$ is compact as an object of $\cat C$.
	But then so is $e$ (since $\unit$ is compact), so that
	$e \to \unit \to f \to \Sigma e$ is an idempotent triangle of compact objects.
	It follows that we have a decomposition $\Spc(\cat C^c) = \supp(\unit)=\supp(e) \sqcup \supp(f)$
	into disjoint closed sets. In particular, $Y:=\supp(e)$ is an open and closed subset of $\Spc(\cat C^c)$.
	Moreover, since $e$ is compact, 
	$\Loc_\otimes(e) = \Loc(\cat C^c_Y)$,
	so indeed $\cat C \to \cat C_L$ is nothing but the finite localization associated to the open and closed set $Y=\supp(e)$.
$(\Leftarrow)$
	Let $\cat C \to \cat C(V)$ 
	denote
	the finite localization associated to an open and closed
	subset $Y \subseteq \Spc(\cat C^c)$ (so $V:=\Spc(\cat C^c)\setminus Y$ as usual).
	Applying the generalized Carlson connectedness theorem \cite[Thm.~2.11]{Balmer07} to
	the decomposition $\supp(\unit)=\Spc(\cat C^c) = Y \sqcup V$, we conclude
	that $\unit \simeq a \oplus b$ for some $a,b \in \cat C^c$ 
	such that
	$\supp(a)=Y$ and $\supp(b)=V$.
	Now let $e \to \unit \to f \to \Sigma e$ denote the idempotent triangle of the localization.
	We claim that it splits.
	First note that $\supp(a)=Y$ implies $a\in \cat C^c_Y$ is acyclic so that $f\otimes a = 0$
	and hence $f \cong f\otimes b$.
	On the other hand, $\supp(b) \cap Y = \emptyset$ implies that $b \otimes z =0$
	for all $z \in \cat C^c_Y$.
	Hence $b \otimes e=0$ since $e \in \Loc(\cat C^c_Y)$.
	Thus, $e \cong e\otimes a$.
	Now, $\Hom_{\cat C}(f,\Sigma e) \cong \Hom_{\cat C}(f \otimes b,\Sigma e \otimes a)
	\cong \Hom(f,\Sigma e \otimes a \otimes \dual b{) =0}$
	since $\supp(a\otimes \dual b)=\supp(a)\cap\supp(\dual b)=\supp(a)\cap\supp(b)=\emptyset$.
	So the sequence splits, and hence  $f$ is compact (being a direct summand of $\unit$).
	Now let's prove that the localization satisfies GN-duality.
	Well, if $x \in \cat C(V)^c$ then
	the Thomason-Neeman localization theorem (cf.~\cite[Cor.~4.5.14, Rem.~4.5.15]{Neeman01}) implies
	$x \oplus \Sigma x = f \otimes d$ for some $d \in \cat C^c$,
	which is compact in $\cat C$. Thus any $x \in \cat C(V)^c$ is compact in $\cat C$,
	meaning that the inclusion $\cat C(V) \hookrightarrow \cat C$ preserves compact objects.
\end{proof}
\begin{Rem}
	For finite localizations we completely understand the spectrum of the compact part of the localized category, and we can give a precise topological description of the \locus{}, as follows:
\end{Rem}
\begin{Prop}\label{prop:locus-of-finite-localization}
	Let $Y \subseteq \Spc(\cat C^c)$ be a Thomason subset,
	with complement $V:=\Spc(\cat C^c)\setminus Y$,
	and let $f^*:\cat C \to \cat C(V)$ 
	denote
	the associated finite localization.
	Under the identification
	\[ \Spc(\cat C(V)^c) \cong V \subseteq \Spc(\cat C^c)\]
	the \locus{} of $f^*$, $\Z_f \subseteq V \subseteq \Spc(\cat C^c)$, coincides with the largest Thomason subset of $\Spc(\cat C^c)$
	which is contained in $V$.
	In other words,
	\[ \Z_f \;\;= \bigcup_{\substack{x\in \cat C^c: \\\supp(x) \subseteq V}} \supp(x) 
		\;\;= \bigcup_{\substack{Z \subset \Spc(\cat C^c)\\ {\textrm{Thomason closed}}\,:\\ Z \subseteq V}} Z.\]
\end{Prop}
\begin{proof}
	Let $e \to \unit \to f \to \Sigma e$ denote the idempotent triangle for the finite localization.
	We hope the double use of the letter $f$ will cause no confusion.
	Indeed, $f^*(-) = f\otimes -$ and $\Ker(f\otimes -) = \Loc_\otimes(e) = \Loc(\cat C_Y^c)$.
	Then observe that if $x \in \cat C^c$ then $\supp(x) \cap Y = \emptyset$
	iff $\supp(x) \cap \supp(y) = \emptyset$ for all $y\in \cat C_Y^c$
	iff $x \otimes \cat C_Y^c = 0$
	iff $x \otimes \Loc(\cat C_Y^c) = 0$
	iff $x \otimes \Loc_\otimes(e) = 0$
	iff $x \otimes e = 0$.
	Thus, $f \otimes x \simeq x$ iff $\supp(x) \subseteq V$.
	Now consider $a \in \cat C(V)^c$.
	By the 
	Thomason-Neeman localization theorem (cf.~\cite[Cor.~4.5.14, Rem.~4.5.15]{Neeman01}),
	$a \oplus \Sigma a \simeq f \otimes b$ for some $b \in \cat C^c$.
	If $a \in \A_f$ then so too $a \oplus \Sigma a \in \A_f$
	and hence $f \otimes b$ is compact in $\cat C$.
	Thus
	\begin{equation}\label{eq:finite-local-ineq}
		\Z_f \subseteq \bigcup_{\substack{b \in \cat C^c:\\f\otimes b \in \cat C^c}} \supp(f \otimes b)
		\subseteq \bigcup_{\substack{x \in \cat C^c:\\f \otimes x\simeq x}} \supp(x)
	\end{equation}
	where the second inequality follows just by taking $x := f\otimes b$.
	On the other hand, consider $x \in \cat C^c$ such that $f \otimes x \simeq x$.
	For any $y \in \cat C(V)^c$ we have $y \oplus \Sigma y \simeq f\otimes c$
	for some $c \in \cat C^c$, and then
	$x \otimes (y \oplus \Sigma y) \simeq x \otimes f \otimes c \simeq x \otimes c$
	is compact in $\cat C$.
	Hence its $\oplus$-summand $x \otimes y$ is also compact in $\cat C$.
	This shows that if $x \in \cat C^c$ satisfies $f \otimes x \simeq x$
	then $x \in \A_f$, showing that the inequalities in \eqref{eq:finite-local-ineq} are equalities.
\end{proof}
\begin{Exa}\label{exa:SH}
	The $p$-localization $\SH \to \SHp$ of the stable homotopy category is an example
	of a finite localization. 
	Indeed, if $\cat C_{p,n} \in \Spc(\SHc)$ denotes 
	the kernel in $\SHc$ of the $(n-1)$th Morava K-theory (after localization at $p$),
	then $p$-localization is the finite localization associated to the Thomason subset
	$Y = \bigcup_{q\neq p} \overline{\{\cat C_{q,2}\}}$.
	(We refer the reader to \cite{BalmerSanders17} or \cite[Sec.~9]{Balmer10b} for a 
	description of the space $\Spc(\SHc)$.)
	Then $\Spc(\SHcp) \cong V \subseteq \Spc(\SHc)$
	identifies the 
	spectrum of the $p$-local category
	with $V = \SET{\cat C_{p,n}}{1 \le n \le \infty}$.
	One immediately sees from Proposition~\ref{prop:locus-of-finite-localization}
	that the compactness locus of $\SH \to \SHp$ is
	\[ \Z_f=\SET{\cat C_{n}}{2 \le n \le \infty} \subset \Spc(\SHcp).\]
	That is, it is the whole of the $p$-local spectrum except for a single point --- the generic point.
	(See diagram (a) below.)
In fact, we can easily use Proposition~\ref{prop:locus-of-finite-localization},
	to describe the \locus{} of any finite localization of $\SH$.
	For illustrative purposes, we have drawn some examples below.
	In each example, we have drawn the three subsets
	$\Z_f \subset V \subset X$; the light blue region is~$\Z_f$ and the dark blue region is~$V$.
	(Note that in example (d), $\Z_f$ is empty.)
\end{Exa}
\begin{figure}[ht]
	\begin{subfigure}[b]{0.49\textwidth}
		\resizebox{\linewidth}{!}{
\def\diayoffbig{1.0}
\def\diayoffset{0.6}
\def\diaxoffset{0.3}
\def\bignubradius{0.6*\diaxoffset}
\def\smallnubradius{0.4*\diaxoffset}
\def\diabulletsize{0.03cm}
\def\diatinybulletsize{0.005cm}
\def\dianumbranches{9}
\def\diabranchlen{4}
\def\diabranchlenptwo{6}
\def\diasmallratio{0.25}
\def\diabigratio{0.5}
\def\diasmallopac{nearly opaque}
\def\diabigopac{nearly transparent}
\def\bigfillop{50}
\def\smallfillop{25}
\def\smallcol{blue}
\def\bigcol{blue}
\begin{tikzpicture}[trim left, scale=.5, every node/.style={scale=0.5}]

	\coordinate (base) at (-2,0);


	\coordinate (start) at ($(base)+(-1.5*\diaxoffset + 1*\diaxoffset,\diayoffbig)$);
	\path [fill=\bigcol!\bigfillop,rounded corners=2pt, thin] 
	($(base)-(\bignubradius,0)$) --
	($(start)-(\bignubradius,0)$) -- 
	($(start)-(\bignubradius,0)+(0,\bignubradius+\diabranchlen*\diayoffset+\diayoffset)$) --
	($(start)+(\bignubradius,0)+(0,\bignubradius+\diabranchlen*\diayoffset+\diayoffset)$) --
	($(start)+(\bignubradius,0)$) --
	($(base)+(\bignubradius,0)$) --
	($(base)-(0,\bignubradius)$) --
	($(base)-(\bignubradius,0)$);

	\path [fill=\smallcol!\smallfillop, thin] 
	($(start)-(\smallnubradius,0)$) -- 
	($(start)-(\smallnubradius,0)+(0,\diabranchlen*\diayoffset+\diayoffset)$) --
	($(start)+(\smallnubradius,0)+(0,\diabranchlen*\diayoffset+\diayoffset)$) --
	($(start)+(\smallnubradius,0)$) --
	($(start)-(0,\smallnubradius)$) --
	($(start)-(\smallnubradius,0)$);
	\path [fill=\smallcol!\smallfillop] (start) circle (\smallnubradius);
	\path [fill=\smallcol!\smallfillop] ($(start)+(0,\diabranchlen*\diayoffset+\diayoffset)$) circle (\smallnubradius);

\draw [fill] (base) circle (\diabulletsize);

\foreach \x in {1,...,\dianumbranches} {
	\coordinate (start) at ($(base)+(-1.5*\diaxoffset + \x*\diaxoffset,\diayoffbig)$);
	\foreach \n in {0,...,\diabranchlen} {
		\draw [fill] ($(start)+(0,\n*\diayoffset)$) circle (\diabulletsize);
	}
	\draw [very thin] (start) -- ($(start)+(0,\diabranchlen*\diayoffset)$);
	\draw [very thin] (base) -- (start);
	\node at ($(start)+(0,\diabranchlen*\diayoffset+0.65*\diayoffset)$) {$\vdots$};
	\draw [fill] ($(start)+(0,\diabranchlen*\diayoffset+\diayoffset)$) circle (\diabulletsize);
}
\node at ($(base)+(-1.8*\diaxoffset,\diayoffbig+\diabranchlen*\diayoffset+\diayoffset)$) {$\scriptstyle \infty$};
\foreach \n in {0,...,\diabranchlen} {
	\node at ($(base)+ (-1.5*\diaxoffset + \dianumbranches*\diaxoffset,0) +(0.4,\diayoffbig+\n*\diayoffset +0.5*\diayoffset)$) {$\cdots$};
}
\end{tikzpicture}
		}
		\caption*{(a)}
		\vspace{4ex}
	\end{subfigure}
	\begin{subfigure}[b]{0.49\textwidth}
		\resizebox{\linewidth}{!}{
\def\diayoffbig{1.0}
\def\diayoffset{0.6}
\def\diaxoffset{0.3}
\def\bignubradius{0.6*\diaxoffset}
\def\smallnubradius{0.4*\diaxoffset}
\def\diabulletsize{0.03cm}
\def\diatinybulletsize{0.005cm}
\def\dianumbranches{9}
\def\diabranchlen{4}
\def\diasmallratio{0.25}
\def\diabigratio{0.5}
\def\diasmallopac{nearly opaque}
\def\diabigopac{nearly transparent}
\def\bigfillop{50}
\def\smallfillop{25}
\def\smallcol{blue}
\def\bigcol{blue}
\begin{tikzpicture}[trim left, scale=.5, every node/.style={scale=0.5}]

	\coordinate (base) at (-2,0);

	\coordinate (start1) at ($(base)+(-1.5*\diaxoffset + 1*\diaxoffset,\diayoffbig)$);
	\coordinate (start2) at ($(base)+(-1.5*\diaxoffset + 2*\diaxoffset,\diayoffbig)$);
	\path [fill=\bigcol!\bigfillop,rounded corners=2pt, thin] 
	($(base)-(\smallnubradius,0)$) --
	($(start1)-(\bignubradius,0)$) -- 
	($(start1)-(\bignubradius,0)+(0,\bignubradius+0*\diayoffset+\diayoffset)$) --
	($(start2)-(\bignubradius,0)+(0,\bignubradius+0*\diayoffset+\diayoffset)$) --
	($(start2)-(\bignubradius,0)+(0,\bignubradius+\diabranchlen*\diayoffset+\diayoffset)$) --
	($(start2)+(\bignubradius,0)+(\dianumbranches*\diaxoffset,0)+(0,\bignubradius+\diabranchlen*\diayoffset+\diayoffset)$) --
	($(start2)-(0,\bignubradius)+(\bignubradius,0)+(\dianumbranches*\diaxoffset,0)$) --
	($(base)-(0,0.9*\smallnubradius)$) --
	($(base)-(\smallnubradius,0)$);
	\path [fill=\bigcol!\bigfillop] (base) circle (\smallnubradius);

	\coordinate (start) at ($(base)+(-1.5*\diaxoffset + 2*\diaxoffset,\diayoffbig)$);

	\path [fill=\smallcol!\smallfillop, thin] 
	($(start)-(\smallnubradius,0)$) -- 
	($(start)-(\smallnubradius,0)+(0,\diabranchlen*\diayoffset+\diayoffset)$) --
	($(start)+(0,\smallnubradius)+(0,\diabranchlen*\diayoffset+\diayoffset)$) --
	($(start)+(0,\smallnubradius)+(\dianumbranches*\diaxoffset,0)+(0,\diabranchlen*\diayoffset+\diayoffset)$) --
	($(start)+(\smallnubradius,0)+(\dianumbranches*\diaxoffset,0)+(0,\diabranchlen*\diayoffset+\diayoffset)$) --
	($(start)+(\smallnubradius,0)+(\dianumbranches*\diaxoffset,0)$) --
	($(start)-(0,\smallnubradius)+(\dianumbranches*\diaxoffset,0)$) --
($(start)-(0,\smallnubradius)$) --
	($(start)-(\smallnubradius,0)$);
	\path [fill=\smallcol!\smallfillop] (start) circle (\smallnubradius);
	\path [fill=\smallcol!\smallfillop] ($(start)+(\dianumbranches*\diaxoffset,0)$) circle (\smallnubradius);
	\path [fill=\smallcol!\smallfillop] ($(start)+(\dianumbranches*\diaxoffset,\diabranchlen*\diayoffset+\diayoffset)$) circle (\smallnubradius);
	\path [fill=\smallcol!\smallfillop] ($(start)+(0,\diabranchlen*\diayoffset+\diayoffset)$) circle (\smallnubradius);

\draw [fill] (base) circle (\diabulletsize);

\foreach \x in {1,...,\dianumbranches} {
	\coordinate (start) at ($(base)+(-1.5*\diaxoffset + \x*\diaxoffset,\diayoffbig)$);
	\foreach \n in {0,...,\diabranchlen} {
		\draw [fill] ($(start)+(0,\n*\diayoffset)$) circle (\diabulletsize);
	}
	\draw [very thin] (start) -- ($(start)+(0,\diabranchlen*\diayoffset)$);
	\draw [very thin] (base) -- (start);
	\node at ($(start)+(0,\diabranchlen*\diayoffset+0.65*\diayoffset)$) {$\vdots$};
	\draw [fill] ($(start)+(0,\diabranchlen*\diayoffset+\diayoffset)$) circle (\diabulletsize);
}
\node at ($(base)+(-1.5*\diaxoffset,\diayoffbig+\diabranchlen*\diayoffset+\diayoffset)$) {$\scriptstyle \infty$};
\foreach \n in {0,...,\diabranchlen} {
	\node at ($(base)+ (-1.5*\diaxoffset + \dianumbranches*\diaxoffset,0) +(0.4,\diayoffbig+\n*\diayoffset +0.5*\diayoffset)$) {$\cdots$};
}
\end{tikzpicture}
		}
		\caption*{(b)}
		\vspace{4ex}
	\end{subfigure}
	\begin{subfigure}[b]{0.49\textwidth}
		\resizebox{\linewidth}{!}{
\def\diayoffbig{1.0}
\def\diayoffset{0.6}
\def\diaxoffset{0.3}
\def\bignubradius{0.6*\diaxoffset}
\def\smallnubradius{0.4*\diaxoffset}
\def\diabulletsize{0.03cm}
\def\diatinybulletsize{0.005cm}
\def\dianumbranches{9}
\def\diabranchlen{4}
\def\diasmallratio{0.25}
\def\diabigratio{0.5}
\def\diasmallopac{nearly opaque}
\def\diabigopac{nearly transparent}
\def\bigfillop{50}
\def\smallfillop{25}
\def\smallcol{blue}
\def\bigcol{blue}
\begin{tikzpicture}[trim left, scale=.5, every node/.style={scale=0.5}]

	\coordinate (base) at (-2,0);


	\coordinate (start1) at ($(base)+(-1.5*\diaxoffset + 1*\diaxoffset,\diayoffbig)$);
	\coordinate (start2) at ($(base)+(-1.5*\diaxoffset + 2*\diaxoffset,\diayoffbig)$);
	\coordinate (start3) at ($(base)+(-1.5*\diaxoffset + 4*\diaxoffset,\diayoffbig)$);
	\coordinate (start4) at ($(base)+(-1.5*\diaxoffset + 5*\diaxoffset,\diayoffbig)$);
	\path [fill=\bigcol!\bigfillop] (base) circle (\smallnubradius);
	\path [fill=\bigcol!\bigfillop,rounded corners=2pt, thin] 
	($(base)-(\smallnubradius,0)$) --
	($(start1)-(\bignubradius,0)$) -- 
	($(start1)-(\bignubradius,0)+(0,\bignubradius+1*\diayoffset+\diayoffset)$) --
	($(start2)-(\bignubradius,0)+(0,\bignubradius+1*\diayoffset+\diayoffset)$) --
	($(start2)-(\bignubradius,0)+(0,\bignubradius+\diabranchlen*\diayoffset+\diayoffset)$) --
	($(start2)+(\bignubradius,0)+(0,\bignubradius+\diabranchlen*\diayoffset+\diayoffset)$) --
	($(start2)+(\bignubradius,0)+(0,\bignubradius+0*\diayoffset+\diayoffset)$) --
	($(start3)-(\bignubradius,0)+(0,\bignubradius+0*\diayoffset+\diayoffset)$) --
	($(start3)-(\bignubradius,0)+(0,\bignubradius+\diabranchlen*\diayoffset+\diayoffset)$) --
	($(start2)+(\bignubradius,0)+(\dianumbranches*\diaxoffset,0)+(0,\bignubradius+\diabranchlen*\diayoffset+\diayoffset)$) --
	($(start2)-(0,\bignubradius)+(\bignubradius,0)+(\dianumbranches*\diaxoffset,0)$) --
	($(base)-(0,0.9*\smallnubradius)$) --
	($(base)-(\smallnubradius,0)$);

	\path [fill=\smallcol!\smallfillop, thin] 
	($(start2)-(\smallnubradius,0)$) -- 
	($(start2)-(\smallnubradius,0)+(0,\diabranchlen*\diayoffset+\diayoffset)$) --
	($(start2)+(\smallnubradius,0)+(0,\diabranchlen*\diayoffset+\diayoffset)$) --
	($(start2)+(\smallnubradius,0)$) --
	($(start2)-(0,\smallnubradius)$) --
	($(start2)-(\smallnubradius,0)$);
	\path [fill=\smallcol!\smallfillop] (start2) circle (\smallnubradius);
	\path [fill=\smallcol!\smallfillop] ($(start2)+(0,\diabranchlen*\diayoffset+\diayoffset)$) circle (\smallnubradius);

	\path [fill=\smallcol!\smallfillop, thin] 
	($(start3)-(\smallnubradius,0)$) -- 
	($(start3)-(\smallnubradius,0)+(0,\diabranchlen*\diayoffset+\diayoffset)$) --
	($(start3)+(0,\smallnubradius)+(0,\diabranchlen*\diayoffset+\diayoffset)$) --
	($(start2)+(0,\smallnubradius)+(\dianumbranches*\diaxoffset,0)+(0,\diabranchlen*\diayoffset+\diayoffset)$) --
	($(start2)+(\smallnubradius,0)+(\dianumbranches*\diaxoffset,0)+(0,\diabranchlen*\diayoffset+\diayoffset)$) --
	($(start2)+(\smallnubradius,0)+(\dianumbranches*\diaxoffset,0)$) --
	($(start2)-(0,\smallnubradius)+(\dianumbranches*\diaxoffset,0)$) --
	($(start3)-(0,\smallnubradius)$) --
	($(start3)-(\smallnubradius,0)$);
	\path [fill=\smallcol!\smallfillop] (start3) circle (\smallnubradius);
	\path [fill=\smallcol!\smallfillop] ($(start2)+(\dianumbranches*\diaxoffset,0)$) circle (\smallnubradius);
	\path [fill=\smallcol!\smallfillop] ($(start2)+(\dianumbranches*\diaxoffset,\diabranchlen*\diayoffset+\diayoffset)$) circle (\smallnubradius);
	\path [fill=\smallcol!\smallfillop] ($(start3)+(0,\diabranchlen*\diayoffset+\diayoffset)$) circle (\smallnubradius);

\draw [fill] (base) circle (\diabulletsize);

\foreach \x in {1,...,\dianumbranches} {
	\coordinate (start) at ($(base)+(-1.5*\diaxoffset + \x*\diaxoffset,\diayoffbig)$);
	\foreach \n in {0,...,\diabranchlen} {
		\draw [fill] ($(start)+(0,\n*\diayoffset)$) circle (\diabulletsize);
	}
	\draw [very thin] (start) -- ($(start)+(0,\diabranchlen*\diayoffset)$);
	\draw [very thin] (base) -- (start);
	\node at ($(start)+(0,\diabranchlen*\diayoffset+0.65*\diayoffset)$) {$\vdots$};
	\draw [fill] ($(start)+(0,\diabranchlen*\diayoffset+\diayoffset)$) circle (\diabulletsize);
}
\node at ($(base)+(-1.5*\diaxoffset,\diayoffbig+\diabranchlen*\diayoffset+\diayoffset)$) {$\scriptstyle \infty$};
\foreach \n in {0,...,\diabranchlen} {
	\node at ($(base)+ (-1.5*\diaxoffset + \dianumbranches*\diaxoffset,0) +(0.4,\diayoffbig+\n*\diayoffset +0.5*\diayoffset)$) {$\cdots$};
}
\end{tikzpicture}
		}
		\caption*{(c)}
	\end{subfigure}
	\begin{subfigure}[b]{0.49\textwidth}
		\resizebox{\linewidth}{!}{
\def\diayoffbig{1.0}
\def\diayoffset{0.6}
\def\diaxoffset{0.3}
\def\bignubradius{0.6*\diaxoffset}
\def\smallnubradius{0.4*\diaxoffset}
\def\diabulletsize{0.03cm}
\def\diatinybulletsize{0.005cm}
\def\dianumbranches{9}
\def\diabranchlen{4}
\def\diasmallratio{0.25}
\def\diabigratio{0.5}
\def\diasmallopac{nearly opaque}
\def\diabigopac{nearly transparent}
\def\bigfillop{50}
\def\smallfillop{25}
\def\smallcol{blue}
\def\bigcol{blue}
\begin{tikzpicture}[trim left, scale=.5, every node/.style={scale=0.5}]

	\coordinate (base) at (-2,0);


	\coordinate (start) at ($(base)+(-1.5*\diaxoffset + 1*\diaxoffset,\diayoffbig)$);
	\path [fill=\bigcol!\bigfillop,rounded corners=2pt, thin] 
	($(base)-(\smallnubradius,0)$) --
	($(start)-(\bignubradius,0)$) -- 
	($(start)-(\bignubradius,0)+(0,\bignubradius+1*\diayoffset+\diayoffset)$) --
	($(start)+(\bignubradius,0)+(\dianumbranches*\diaxoffset+\diaxoffset,0)+(0,\bignubradius+1*\diayoffset+\diayoffset)$) --
	($(start)-(0,\bignubradius)+(\diaxoffset,0)+(\bignubradius,0)+(\dianumbranches*\diaxoffset,0)$) --
	($(base)-(0,0.9*\smallnubradius)$) --
	($(base)-(\smallnubradius,0)$);
	\path [fill=\bigcol!\bigfillop] (base) circle (\smallnubradius);
%

\draw [fill] (base) circle (\diabulletsize);

\foreach \x in {1,...,\dianumbranches} {
	\coordinate (start) at ($(base)+(-1.5*\diaxoffset + \x*\diaxoffset,\diayoffbig)$);
	\foreach \n in {0,...,\diabranchlen} {
		\draw [fill] ($(start)+(0,\n*\diayoffset)$) circle (\diabulletsize);
	}
	\draw [very thin] (start) -- ($(start)+(0,\diabranchlen*\diayoffset)$);
	\draw [very thin] (base) -- (start);
	\node at ($(start)+(0,\diabranchlen*\diayoffset+0.65*\diayoffset)$) {$\vdots$};
	\draw [fill] ($(start)+(0,\diabranchlen*\diayoffset+\diayoffset)$) circle (\diabulletsize);
}
\node at ($(base)+(-1.5*\diaxoffset,\diayoffbig+\diabranchlen*\diayoffset+\diayoffset)$) {$\scriptstyle \infty$};
\foreach \n in {0,...,\diabranchlen} {
	\node at ($(base)+ (-1.5*\diaxoffset + \dianumbranches*\diaxoffset,0) +(0.4,\diayoffbig+\n*\diayoffset +0.5*\diayoffset)$) {$\cdots$};
}
\end{tikzpicture}
		}
		\caption*{(d)}
	\end{subfigure}
\end{figure}
\begin{Exa}\label{exa:open-imm}
	Besides $\cat C=\SH$, we can also apply Proposition~\ref{prop:locus-of-finite-localization} to some other interesting examples, notably $\cat C=\Derqc(X)$ for a quasi-compact and quasi-separated scheme $X$.  Indeed, open immersions and the ``inclusion of a stalk'' are both examples of finite localizations on $\Derqc(X)$ so Proposition~\ref{prop:locus-of-finite-localization} applies.  For example, if $f:U \hookrightarrow X$ is the inclusion of a quasi-compact open subset then the compactness locus $\Z_f \subset U \cong \Spc(\Derqc(U)^c)$	of the derived pull-back $f^*:\Derqc(X)\to\Derqc(U)$ is precisely the union of all Thomason closed subsets of $X \cong \Spc(\Derqc(X)^c)$ which are contained in $U$.
\end{Exa}
\begin{Rem}
	If $X=\Spc(\cat C^c)$ is noetherian, then $X$ has finitely many connected
	components and the components are both open and closed.
	In this case, a subset $Y \subseteq X$ is open and closed iff it is a (possibly empty)
	union of connected components.
	Note also that when $X$ is noetherian we do not need to distinguish between
	closed subsets and Thomason closed subsets, so our understanding of Proposition~\ref{prop:locus-of-finite-localization} simplifies somewhat.
	In this case, $\Z_f$ is the union of all closed subsets of $X$ which are contained in $V$ ---
	\ie~which do not intersect $Y$ -- or, thinking more geometrically, a point
	$x \in X$ is contained in the \locus{} iff $\overline{\{x\}} \subset V$.
	In other words, $\Z_f$ and $V$ have the same closed points, but $\Z_f$
	really consists of those points $x$ which are not contained in $Y$,
	and for which, moreover, the entire irreducible closed set~$\overline{\{x\}}$ does not intersect $Y$.
	This kind of interplay has to do with the fact that $\Z_f$ is closed under specialization
	while $V$ is closed under generalization.
	In any case, the assumption that $X$ is noetherian is sometimes overkill for these
	considerations.
For example, although the space $X=\Spc(\SHc)$ in Example~\ref{exa:SH} is not noetherian,
the \locus{} of any finite localization coincides with the union of all (not necessarily Thomason) closed subsets contained in $V$.
	In any case, let us end this section by giving an explicit example where these considerations do matter.
\end{Rem}
\begin{Exa}
	Let $\cat C=\Der(\mathbb{F}_2^{\mathbb N})$ be the derived category of a countably infinite product of the field $\mathbb{F}_2$.
	The spectrum $X:=\Spc(\cat C^c)=\Spec(\mathbb{F}_2^{\mathbb N})$ is homeomorphic
	to the Stone-\v{C}ech compactification of $\mathbb N$ --- that is, the Cantor space.
	Since the Cantor space is compact Hausdorff, we have that a subset is (quasi-)compact iff
	it is closed.
	Thus, the Thomason closed subsets are precisely the clopen subsets.
	Moreover, since the clopen subsets form a basis for the topology on the Cantor space,
	we see that the Thomason subsets are precisely the open subsets.
	Then take any point $x \in X$. It is closed
	hence the complement $Y:=X\setminus\{x\}$ is open, and hence Thomason.
	Consider the associated finite localization of $\Der(\mathbb{F}_2^{\mathbb N})$
	associated to $Y$.
	According to Proposition~\ref{prop:locus-of-finite-localization}, $\Z_f$ is the union of all Thomason subsets of $X$
	which are contained in $V=\{x\}$.
	But the point $x$ is not open (every open set is uncountable)
	so $\Z_f = \emptyset$.
	On the other hand, if we took the union of all (not necessarily Thomason) closed subsets 
	of $X$ which are contained in $V$, then we would obtain $V$ itself.
\end{Exa}
\bigbreak
\section{The compactness locus of inflation}
\label{sec:adams-locus}
We have seen that by taking $\cat I$ to be the subcategory of finite $N$-free $G$-spectra, 
then Theorem~\ref{thm:main-thm}
applies
and provides the Adams isomorphism (Example~\ref{exa:adams}). But really one should apply the theorem
to the canonically determined subcategory associated to the \locus{} of the functor (Remark~\ref{rem:only-one}).
Does the \locus{} of inflation single out precisely the $N$-free $G$-spectra, or does it
single out a larger collection of $G$-spectra?
In this section, we will 
answer
this question, by figuring out precisely what the \locus{}
of inflation is.
This is possible, because we know what the spectrum of $\SHGc$
looks like, by \cite{BalmerSanders17},
up to an unresolved indeterminacy in the topology, related to the chromatic shifting behaviour of the
Tate construction.
Even though this indeterminacy in the topology is presently unresolved, 
we know enough about the topology to completely describe the \locus{} of inflation (Theorem~\ref{thm:locus-of-inflation} below).
All of this will be for finite groups $G$ as it is only in this case that we have a description of the spectrum $\Spc(\SHGc)$.

We require that the reader has some familiarity with \cite{BalmerSanders17}.
Recall that for a finite group $G$, the topological space $\Spc(\SHGc)$ consists of points
\[ \cat P(H,p,n) := \SET{ X \in \SH(G)^c}{\Phi^H(X) \in \cat C_{p,n} \text{ in } \SH^c} \]
for each 
conjugacy class of subgroups $H \le G$, prime number $p$, and ``chromatic'' 
integer $1 \le n \le \infty$, where $\Phi^H:\SH(G)^c \to \SH^c$ denotes the geometric $H$-fixed point functor.
These points are all distinct except when $n=1$ where we 
have $\cat P(H,p,1)=\cat P(H,q,1)$ for all primes $p,q$;
consequently, we will just write $\cat P(H,1)$ for this point.
We also sometimes write $\cat P_G(H,p,n)$ to indicate the ambient group~$G$.
Finally, recall that $\Op(G)$ denotes the smallest normal subgroup of $G$ with index a power of $p$.

\begin{Not}
	For a normal subgroup $N \lenormal G$, 
	we write 
	\[\ZGN \subset \Spc(\SHGc)\]
	for the \locus{} of the inflation functor $\Infl_{G/N}^G:\SH(G/N)\to\SH(G)$
	and write $\AGN \subset \SHGc$ for the corresponding thick tensor-ideal.
\end{Not}

The ultimate goal of this section is to prove the following theorem:

\begin{Thm}\label{thm:locus-of-inflation}
	Let $G$ be a finite group and let $N \lenormal G$ be a normal subgroup.
	The \locus{} $\ZGN \subset \Spc(\SHGc)$
	of $\Infl_{G/N}^G:\SH(G/N) \to \SH(G)$
	may be described as follows.
	For any $H \le G$, we have the following:
	\begin{enumerate}
		\item If $N\cap H \not\subseteq \Op(H)$, then $\cat P_G(H,p,n) \not\in \ZGN$ for all $2 \le n \le \infty$.
		\item If $N\cap H \subseteq \Op(H)$, then $\cat P_G(H,p,n) \in \ZGN$ for all $2 \le n \le \infty$.
		\item $\cat P_G(H,1) \in \ZGN$ if and only if $N\cap H \subseteq \Op(H)$ for all primes $p$.
	\end{enumerate}
\end{Thm}

This will be proved on page~
\pageref{prf:locus-of-inflation}, after we have developed the necessary lemmas. 
Some explicit examples will be drawn in 
Example~\ref{exa:locus-for-C_p} and Example~\ref{exa:locus-for-D_10}.

\begin{Rem}\label{rem:alternative-statement}
	The theorem can be reformulated as follows.
Define the $p$-local isotropy of a spectrum $X \in \SHGc$ to consist of those subgroups $H \le G$
such that $\Phi^H(X)_{(p)} \neq 0$.
Theorem~\ref{thm:locus-of-inflation} is equivalent to the statement that $X \in \AGN$
if and only if
	for all primes~$p$ and all $p$-local isotropy subgroups $H$ 
	of $X$
	we have $N \cap H \subseteq \Op(H)$.
	On the other hand,
	$X$ is $N$-free if and only if for every $p$-local isotropy subgroup~$H$ 
	of $X$
	we have $N \cap H=1$.
\end{Rem}

\begin{Exa}
	To gather some intuition in light of the above remark, take $N=G$ and consider the \mbox{mod-$p$} Moore spectrum 
	with trivial $G$-action: $\triv_G(M(p))$.
	For this $G$-spectrum, every subgroup of $G$ is $p$-local isotropy and no subgroup of $G$ is $q$-local isotropy for $q\neq p$.
	Evidently, it is not $G$-free (for $G \neq 1$).
	On the other hand, it is relatively compact iff every subgroup of $G$ is $p$-perfect iff $p$ does not divide the order of $G$.
	We can see this directly as follows.
	The fixed points are 
	\[\lambda^G(\triv_G(M(p))) \cong M(p) \smashh \lambda^G(\unit)
		\cong \bigvee_{(H)} M(p) \smashh \Sigma^\infty BWH_+\]
	by the tom Dieck splitting theorem
	and if $p$ does not divide the order of $G$ then this is just a finite wedge of finite spectra.
Indeed, in this case
\[\HZ_*(M(p)\smashh \Sigma^\infty BWH_+) \cong \HFp_*(\Sigma^\infty BWH_+) \cong H_*(WH;\Fp) = \Fp\]
and a connective spectrum is finite iff it has finitely generated total homology.
More generally, Theorem~\ref{thm:locus-of-inflation} says that $\lambda^N(X)$ is compact for any finite $G$-spectrum~$X$ which is $p$-torsion for some prime $p$ which does not divide $|N|$.
Indeed, a $p$-torsion spectrum only has $p$-isotropy for this particular prime $p$ and if $p$ does not divide the order of $N$ then $N \cap H \subseteq \Op(H)$ for any subgroup $H \le G$.

Although this example provides some intuition, keep in mind that it doesn't tell the whole story as the difference between the $N$-free $G$-spectra and the relatively compact $G$-spectra (with respect to inflation) manifests even $p$-locally for primes dividing $|N|$; see Remark~\ref{rem:adams-larger} for example.
\end{Exa}

\begin{Lem}\label{lem:supp-of-induction}
	Let $K \le H \le G$ be finite and let $x \in \SH(H)^c$.
	If $\cat P_H(K,p,n) \in \supp(x)$ then $\cat P_G(K,p,n) \in \supp(\Ind_H^G(x))$.
\end{Lem}
\begin{proof}
	The hypothesis means that $x \not\in \cat P_H(K,p,n)$, \ie~that $\Phi^K(\Res^H_K(x)) \not\in \cat C_{p,n}$.
	On the other hand, the conclusion means that
	$\Ind_H^G(x) \not\in \cat P_G(K,p,n)$, \ie~that
	$\Phi^K(\Res_K^G \Ind_H^G(x)) {\not\in \cat C_{p,n}}$.
	But 
	$\Phi^K(\Res_K^G \Ind_H^G x) \cong \Phi^K(\Res_K^H\Res_H^G \Ind_H^G x)$
	and $x$ is a direct summand of $\Res_H^G \Ind_H^G x$
	(see \eg~\cite[Lem.~3.3]{BalmerDellAmbrogioSanders15}).
	Thus the hypothesis implies the conclusion.
\end{proof}
\begin{Prop}\label{prop:reduce-to-subgroup}
	Let $G$ be a finite group and let $N \lenormal G$ and $H \le G$.
	Then $\cat P_G(H,p,n) \in \ZGN$ if and only if
	$\cat P_H(H,p,n) \in \Z_{H\cap N,H}$.
\end{Prop}
\begin{proof}
	First we prove the ($\Rightarrow$) direction.
	Let $x \in \AGN$ with $\cat P_G(H,p,n) \in \supp(x)$.
	Then $\cat P_H(H,p,n) \in \supp(\Res_H^G(x))$ by 
	\cite[Cor.~4.4]{BalmerSanders17}, so it
	suffices to prove that $\Res_H^G(x) \in \A_{H \cap N,H}$.
	For any $y \in \SH(H)^c$, we have that
	\[ \lambda^{H\cap N,H}(\Res^G_H(x) \smashh y) \in \SH(H/H\cap N) \cong \SH(HN/N) \]
	is a direct summand of 
	\begin{align*}
		\Res_{HN/N}^{G/N}\CoInd_{HN/N}^{G/N}\lambda^{H \cap N,H}(\Res_H^G(x) \smashh y)
		&\cong \Res_{HN/N}^{G/N} \lambda^{N,G} \CoInd_H^G(\Res_H^G(x) \smashh y) \\
		&\cong \Res_{HN/N}^{G/N} \lambda^{N,G} (x \smashh \CoInd_H^G(y))
	\end{align*}
	which is compact since restriction and coinduction both preserve compact objects
	and $x \in \AGN$ by assumption.
	Hence $\Res_H^G(x) \in \A_{H\cap N,H}$.
Now we prove the ($\Leftarrow$) direction.
	Let $x \in \A_{H\cap N,H}$ with $\cat P_H(H,p,n) \in \supp(x)$.
We claim that $\Ind_H^G(x)$ is contained in $\AGN$.
	The conclusion will then follow since Lemma~\ref{lem:supp-of-induction} implies that $\cat P_G(H,p,n) \in \supp(\Ind_H^G(x))$.
	Now, $x \in \A_{H\cap N,H}$ implies that
	\[ \lambda^{H\cap N,H}(x \smashh H/L_+) \]
	is compact in $\SH(H/H\cap N)$ for all $L \le H$.
	So, since $\Ind \cong \CoInd$ preserves compacts it follows (using $H/H\cap N \cong HN/N \hookrightarrow G/N$)
	that 
	\[ \lambda^{N,G}(\CoInd_H^G(x \smashh H/L_+)) \]
	is compact in $\SH(G/N)$ for all $L \le H$.
	On the other hand, for any $K \le G$, we have
	\begin{align*}
		\lambda^{N,G}(\Ind_H^G(x) \smashh G/K_+) &\cong \lambda^{N,G}(\Ind_H^G(x \smashh \Res_H^G(G/K_+))) \\
		& \cong \bigoplus_{[g]\in H \backslash G/K} \lambda^{N,G}(\Ind_H^G(x \smashh (H/H^g \cap K)_+))
	\end{align*}
	by the double-coset formula, which is compact.
	This proves $\Ind_H^G(x) \in \AGN$.
\end{proof}
\begin{Rem}\label{rem:N-concentrated-in-locus}
	Recall
	that a $G$-spectrum 
	is said to be $N$-concentrated
	if it is local 
	for 
	the finite localization associated 
	to the family $\cat F{[\notsupseteq N]} := \SET{{H \le G}}{{H \not\supseteq N}}$.
By definition, the geometric $N$-fixed point functor $\tilde{\Phi}^N:\SH(G) \to \SH(G/N)$
is given by $\tilde{\Phi}^N(x) := \lambda^N( \tildeE\cat F{[\notsupseteq N]} \smashh x)$.
Thus, if $x \in \SHGc$ is \mbox{$N$-concentrated} then $\lambda^N(x) \cong \tilde{\Phi}^N(x)$
is just the geometric $N$-fixed points of $x$, and is therefore compact in $\SH(G/N)$.
It follows that if $x \in \SHGc$ is $N$-concentrated then $x \in \AGN$.
\end{Rem}
\begin{Rem}\label{rem:N-concentrated-quasi-inverse}
	It will also be useful to recall (cf.~\cite[Sec.~2 (H)]{BalmerSanders17}) that 
	restricting $\lambda^N:\SH(G) \to \SH(G/N)$ to the subcategory of $N$-concentrated $G$-spectra
	produces
	an equivalence $\lambda^N:\tildeE\cat F[\notsupseteq N] \smashh \SH(G) \xra{\simeq} \SH(G/N)$,
	with quasi-inverse given by $x \mapsto \tildeE\cat F[\notsupseteq N]\smashh \Infl_{G/N}^G(x)$.
\end{Rem}
\begin{Lem}\label{lem:concentrated}
	Let $K,N \lenormal G$ be two normal subgroups of a finite group.
	Let ${x \in \AGN}$ and assume that $x$ is $K$-concentrated.
	Then $\lambda^K(x) \in \A_{KN/K,G/K}$.
\end{Lem}
\begin{proof}
	For any $y \in \SH(G/K)^c$, we have
	\begin{align*}
		\lambda^{KN/K}(\lambda^K(x) \smashh y) &\cong \lambda^{KN/K}\lambda^K(x \smashh \Infl_{G/K}^G(y)) \\
		&\cong \lambda^{KN}(x \smashh \Infl_{G/K}^G(y))
	\end{align*}
	and $x \smashh \Infl_{G/K}^G(y) \in \AGN$ and is also $K$-concentrated.
	So, it suffices to prove 
	just 
	that $\lambda^{KN}(x)$ is compact in $\SH(G/KN)$.
	Now compute in the opposite direction: $\lambda^{KN}(x) \cong \lambda^{KN/N}\lambda^N(x)$.
	Observe that if $H \supseteq K$ then $HN/N \supseteq KN/N$.
	So, if $H \le G$ is such that $HN/N \not\supseteq KN/N$ then $H \not\supseteq K$ and hence
	$G/H_+ \smashh x = 0$ since~$x$ is $K$-concentrated.
	In other words, $E\cat F_+ \smashh x= 0$ where $\cat F = \SET{{H \le G}}{HN/N \not\supseteq KN/N}$.
	But this family $\cat F$ is the family ``inflated'' from the family of subgroups of $G/N$ which do not contain $KN/N$
	(cf.~\cite[Thm~5.11]{BalmerSanders17}
	and \cite[Cor.~4.5]{BalmerSanders17}).
	That is, $\Infl_{G/N}^G(E\cat F[\notsupseteq KN/N]_+) \smashh x=0$,
	\ie~$x \cong \Infl_{G/N}^G(\tildeE\cat F[\notsupseteq KN/N])\smashh x$.
	So, $\lambda^N(x) \cong \lambda^N(\Infl_{G/N}^G(\tildeE\cat F[\notsupseteq KN/N])\smashh x)
	\cong \tildeE\cat F[\notsupseteq KN/N] \smashh \lambda^N(x)$.
	In other words, if $x$ is $K$-concentrated then $\lambda^N(x)$ is $KN/N$-concentrated.
	Hence $\lambda^{KN}(x) \cong \lambda^{KN/N}(\lambda^N(x)) \cong \Phi^{KN/N}(\lambda^N(x))$
	is compact \emph{provided that} $\lambda^N(x)$ is compact in $\SH(G/N)$.
	But this is the case since $x \in \AGN$ by assumption.
\end{proof}

\begin{Prop}\label{prop:inclusions-G}
	Let $G$ be a finite group and let $N \lenormal G$ be a normal subgroup.
	If $N \subseteq \Op(G)$ for some prime $p$, then $\cat P(G,p,n) \in \ZGN$ for all $2 \le n \le \infty$.
	If $N \subseteq \Op(G)$ for all primes $p$, then $\cat P(G,1) \in \ZGN$.
\end{Prop}
\begin{proof}
	Suppose $N \subseteq \Op(G)$ for some prime $p$.
	Consider the closed set $\overline{\{\cat P(G,p,2)\}} \subset \Spc(\SHGc)$.
	It is a \emph{Thomason} closed set (cf.~\cite[Prop.~10.1]{BalmerSanders17}), and hence is the support of some $X \in \SHGc$.
	We claim that $G/H_+ \smashh X = 0$ for all $H \le G$ such that $H \not\supseteq N$.
	Indeed, if $\cat P(K,q,m) \in \supp(X) = \overline{\{\cat P(G,p,2)\}}$
	then it follows from \cite[Cor.~6.4]{BalmerSanders17} 
	and \cite[Prop.~6.9]{BalmerSanders17} that $q=p$ and that $K$ is 
	a $p$-subnormal subgroup of $G$.
	On the other hand, if $\cat P(K,p,m) \in \supp(G/H_+)$ then $K \le_G H$ (by \cite[Cor.~4.13]{BalmerSanders17}).
	Hence, if $H \not\supseteq N$ then $K \not\supseteq N$, and so $K \not\supseteq \Op(G)$ by our hypothesis; that is, $K$ is not a $p$-subnormal subgroup of $G$ (by \cite[Lem.~3.3]{BalmerSanders17}).
	We thus conclude that
	$\supp(G/H_+) \cap \supp(X) = \emptyset$ (and hence $G/H_+ \smashh X = 0$) if $H \not\supseteq N$.
	Hence $E\cat F[\notsupseteq N]_+ \smashh X = 0$.
	In other words, $X$ is $N$-concentrated: $X \cong X \smashh \tildeE\cat F[\notsupseteq N]$.
	Hence $X \in \AGN$ (cf.~Rem.~\ref{rem:N-concentrated-in-locus})
	so $\cat P(G,p,2) \in \ZGN$, which proves the first claim (cf.~\cite[Prop.~6.2]{BalmerSanders17}).
	Next, suppose $N \subseteq \Op(G)$ for all primes $p$.
	Then we can apply the same argument to the Thomason closed subset $\overline{\{\cat P(G,1)\}}$.
	In this case, if a prime $\cat P(K,q,m) \in \overline{\{\cat P(G,1)\}}$
	then $K$ is 
	a $q$-subnormal subgroup of $G$, and we use $N \subseteq \Oq(G)$
	to ensure 
	that such a prime cannot be contained 
	in $\supp(G/H_+)$ when $H \not\supseteq N$.
	The argument is otherwise identical.
\end{proof}

\begin{Prop}\label{prop:p-group-no-inclusions}
	Let $G$ be a $p$-group and let $1 \neq N \lenormal G$ be a nontrivial normal subgroup.
	Then $\cat P(G,p,n) \not\in \ZGN$ for all $2 \le n \le \infty$.
\end{Prop}
\begin{proof}
	For any $2 \le m < \infty$, the subset $\overline{\{ \cat C_{p,m} \}} \subset \Spc(\SHc)$ 
	is a Thomason closed subset of the spectrum of the non-equivariant stable homotopy category (cf.~\cite[Cor.~10.5]{BalmerSanders17} or \cite[Cor.~9.5(d)]{Balmer10b}).
	Hence, there exists $X_{p,m} \in \SH^c$ such that $\supp(X_{p,m}) = \overline{\{\cat C_{p,m}\}}$.
	Then $\triv_G(X_{p,m}) \in \SHGc$ has support in $\Spc(\SHGc)$ equal to
	\[ \supp(\triv_G(X_{p,m})) = \SET{ \cat P(H,p,n)}{n \ge m, p \text{ fixed}, H \le G}.\]
	Suppose for a contradiction that $\cat P(G,p,n) \in \ZGN$ for some $2 \le n < \infty$.
	Then since $G$ is a $p$-group, every subgroup of $G$ is a $p$-subnormal subgroup, and thus there
	is some $m \ge n$ such that $\supp(\triv_G(X_{p,m})) \subseteq \ZGN$.
	Indeed, we can take $m = n + \log_p(|G|)$ for example (cf.~\cite[Cor.~8.3]{BalmerSanders17}).
	Hence, $\triv_G(X_{p,m}) \in \AGN$,
	so that $\lambda^N(\triv_G(X_{p,m}))$ is a compact object of $\SH(G/N)$.
	We claim that this cannot be the case.
	Indeed, by the projection formula we have
	\[\lambda^N(\triv_G(X_{p,m}))\cong \lambda^N(\Infl_{G/N}^G(\triv_{G/N}(X_{p,m})))
		\cong \lambda^N(\unit) \otimes \triv_{G/N}(X_{p,m})\]
	and the tom Dieck splitting theorem
	implies that $\lambda^N(\unit)$ has 
\[\Sigma^\infty_{G/N} E\cat F(N;G)/N_+\]
	as a direct summand.
	Since restriction $\SH(G/N)\to\SH$ preserves compact objects, 
 we would then have that $\Sigma^\infty BN_+ \smashh X_{p,m}$ is compact in $\SH$.
	But $\cat C_{p,\infty} \in \supp(X_{p,m})$
	means that $X_{p,m} \not\in \cat C_{p,\infty}$; \ie $H\Fp_*(X_{p,m}) \neq 0$.
	Since $X_{p,m}$ is compact (hence dualizable) we similarly have non-vanishing of cohomology: $H\Fpstar(X_{p,m}) \neq 0$.
	Now $H\Fpstar$ is a tensor-functor to the category of graded $\Fp$-modules.
	So if $\Sigma^\infty BN_+ \smashh X_{p,m}$ were compact then
	\begin{align*}
		H\Fpstar(\Sigma^\infty BN_+ \smashh X_{p,m}) &\cong H\Fpstar(\Sigma^\infty BN_+) \otimes_\Fp H\Fpstar(X_{p,m}) \\
		&\cong H^*(BN;\Fp) \otimes_{\Fp} H\Fpstar(X_{p,m})
	\end{align*}
	would be concentrated in finitely many degrees.
	Since $H\Fpstar(X_{p,m})\neq 0$, this contradicts Proposition~\ref{prop:group-coho}.
	This completes the proof that $\cat P(G,p,n) \not\in \ZGN$ for all $2 \le n < \infty$.
	Finally, let's prove that $\cat P(G,p,\infty) \not\in \ZGN$.
	By definition $\ZGN$ is a \emph{Thomason} closed subset.
	Thus, by \cite[Cor.~10.5]{BalmerSanders17} it is a union of irreducible closed sets $\overline{\{\cat P(H,q,n)\}}$
	for $H \le G$, $q$ a prime, and $1 \le n <\infty$ \emph{finite}.
	Thus, $\cat P(G,p,\infty) \in \ZGN$ necessarily implies that $\cat P(G,p,n) \in \ZGN$
	for some finite $n$, 
	which we have just shown is not possible.
\end{proof}

\begin{Prop}\label{prop:no-inclusions-G}
	Let $G$ be a finite group, and let $N \lenormal G$ be a normal subgroup.
	If $N \not\subseteq \Op(G)$ then 
	$\cat P(G,p,n) \not\in \ZGN$ for all $2 \le n \le \infty$.
\end{Prop}
\begin{proof}
	For any fixed prime $p$ and $2 \le m < \infty$, consider the closed subset
	\[ Z_{p,m} := \SET{\cat P(H,p,n)}{n \ge m, p \text{ fixed}, H \le G} = \bigcup_{H\le G} \overline{\{ \cat P(H,p,m)\}}\]
	of $\Spc(\SHGc)$ (cf.~\cite[Prop.~6.2]{BalmerSanders17}).
	By \cite[Cor.~10.5]{BalmerSanders17}, it is a \emph{Thomason} closed subset.
	Hence, 
	there exists
	$X_{p,m} \in \SHGc$ such that $\supp(X_{p,m}) = Z_{p,m}$.
	Now consider the family of subgroups $\cat F := \cat F[\notsupseteq\Op(G)]  = \SET{H \le G}{H \not\supseteq \Op(G)}$.
	The set $Z_{p,m}$ decomposes into two disjoint subsets:
	\[ \SET{\cat P(H,p,n)}{n \ge m, p \text{ fixed}, H \in \cat F} \amalg \SET{\cat P(H,p,n)}{n \ge m, p \text{ fixed}, H \not\in \cat F}.\]
	The first subset is closed; indeed, it is equal to
	\[ Z_{p,m} \cap \bigcup_{H \in \cat F} \supp(G/H_+) = \bigcup_{H \in \cat F} \overline{\{\cat P(H,p,m)\}}=:Z_1.\]
	On the other hand, the key reason for our choice of this particular family $\cat F$ is that 
	the right-hand subset is also closed. Indeed,
	\begin{equation}\label{eq:closed-blah}
		\SET{\cat P(H,p,n)}{n \ge m, p \text{ fixed}, H \not\in \cat F}
		= \bigcup_{H \supseteq \Op(G)} \overline{\{\cat P(H,p,m)\}} =:Z_2.
	\end{equation}
	The inclusion $\subseteq$ is evident.
	The point is that if $\cat P(K,q,l) \in \overline{\{\cat P(H,p,m)\}}$,
	then $q=p$ and $K$ is conjugate to a $p$-subnormal subgroup of $H$.
	But $H \supseteq \Op(G)$ means that $H$ is a $p$-subnormal subgroup of $G$.
	So $K$ is conjugate to a $p$-subnormal subgroup of $G$ which means that $K^g \supseteq \Op(G)$ for some $g \in G$
	so that $\cat P(K,q,l) = \cat P(K^g,p,l)$ is contained in the left-hand side of \eqref{eq:closed-blah}.

	Thus, $\supp(X_{p,m}) = Z_1 \amalg Z_2$ is a decomposition into two disjoint closed sets.
	It follows from \cite[Thm.~2.11]{Balmer07} that we have a decomposition $X_{p,m} = x_1 \oplus x_2$
	where $\supp(x_i) = Z_i$ for each $i=1,2$.
	In fact, one easily sees that $x_1 \cong E\cat F_+ \smashh X_{p,m}$ and $x_2 \cong \tildeE\cat F \smashh X_{p,m}$.
	The key point is that these objects are compact in $\SH(G)$.

	Now, suppose for a contradiction that $\cat P(G,p,n) \in \ZGN$ for some $2 \le n < \infty$.
	Then 
	\[\supp(\tildeE\cat F \smashh X_{p,m}) \subset \overline{\{\cat P(G,p,n)\}}\subset \ZGN\]
	for $m$ large enough, \eg, for $m \ge n + \log_p(|G/\hspace{-0.2ex}\Op(G)|)$.
	Then $\tildeE\cat F \smashh X_{p,m} \in \AGN$ for $m$ large enough.
	By Lemma~\ref{lem:concentrated}, 
	\[\lambda^{\Op(G)}(\tildeE\cat F \smashh X_{p,m}) \cong \Phi^{\Op(G)}(\tildeE\cat F \smashh X_{p,m})\]
	is contained in $\A_{\Op(G)N/\hspace{-0.2ex}\Op(G),G/\hspace{-0.2ex}\Op(G)}$.
	Moreover, 
	\[\tildeE\cat F \smashh X_{p,m} \cong \tildeE\cat F \smashh \Infl_{G/\hspace{-0.2ex}\Op(G)}^G(\Phi^{\Op(G)}(\tildeE\cat F \smashh X_{p,m}))\]	
	by Remark~\ref{rem:N-concentrated-quasi-inverse}.
	Thus $\cat P(G,p,m) \in \supp(X_{p,m})$
	implies that 
	\[\cat P(G,p,m) \in \supp(\Infl_{G/\hspace{-0.2ex}\Op(G)}^G(\Phi^{\Op(G)}(\tildeE\cat F\smashh X_{p,m})))\]
	and hence (by \cite[Cor.~4.5]{BalmerSanders17}) that
	\[\cat P(G/\hspace{-0.2ex}\Op(G),p,m) \in \supp(\Phi^{\Op(G)}(\tildeE\cat F \smashh X_{p,m})).\]
	In summary, if $\cat P(G,p,n) \in \ZGN$ for some $2 \le n < \infty$
	then $\cat P(G/\hspace{-0.2ex}\Op(G),p,m) \in \Z_{N\hspace{-0.2ex}\Op(G)/\hspace{-0.2ex}\Op(G),G/\hspace{-0.2ex}\Op(G)}$ for some $m \ge n \ge 2$.
	Our assumption that $N \not\subseteq \Op(G)$ implies that $N\hspace{-0.2ex}\Op(G)/\hspace{-0.2ex}\Op(G)$
	is a nontrivial subgroup of the $p$-group $G/\hspace{-0.2ex}\Op(G)$, and hence we conclude by 
	Proposition~\ref{prop:p-group-no-inclusions}.
	Finally, $\cat P_G(G,p,\infty)\in \ZGN$
	implies that $\cat P_G(G,p,n) \in \ZGN$ for some
	finite $n$ (by the same argument as in the proof of 
	Proposition~\ref{prop:p-group-no-inclusions})
	and so the full claim is proved.
\end{proof}

Finally, let us prove~Theorem~\ref{thm:locus-of-inflation}.

\begin{proof}[Proof of Theorem~\ref{thm:locus-of-inflation}]
	\label{prf:locus-of-inflation}
	By Proposition~\ref{prop:reduce-to-subgroup}, it suffices to prove (a) in the case ${H=G}$, which is Proposition~\ref{prop:inclusions-G}.
	On the other hand, Proposition~\ref{prop:reduce-to-subgroup}
	and Proposition~\ref{prop:no-inclusions-G} together prove (b) and the $\Leftarrow$ direction of (c).
	Finally, the $\Rightarrow$ direction of~(c) follows from (a), 
	since $\cat P_G(H,1) \in \ZGN$ implies $\cat P_G(H,p,2) \in \ZGN$ for all~$p$.
\end{proof}

\begin{Rem}
	We proved in Proposition~\ref{prop:infl-doesnt-satisfy-gn-duality} that $\Infl_{G/N}^G$ does not satisfy GN-duality, except when $N=1$.
	We can also deduce this (for $G$ finite) from Theorem~\ref{thm:locus-of-inflation}.
	Indeed, if $\Infl_{G/N}^G$ satisfied GN-duality then by the theorem, $H\cap N \subseteq \Op(H)$ for all $H \le G$
	which implies that $N=1$. Indeed, if $p$ divides $|N|$ then $N$ would contain a subgroup $H$ of order $p$, yielding the contradiction $H=H\cap N \subseteq \Op(H)=1$.
\end{Rem}

\begin{Exa}\label{exa:locus-for-C_p} 
	The \locus{} for $G=N=C_p$, the cyclic group of order~$p$, is displayed in Figure~\ref{fig:Cp} below.
Note that the $G$-free spectra correspond to the irreducible component $\overline{\{\cat P(1,1)\}} = \supp(G_+)$.
\begin{figure}[htp]
\begin{tikzpicture}
	\def\bradius{0.062};
	\def\nbranches{6}
	\def\njoints{5}
	\def\njointsplustwo{7}
	\def\yoffsmall{0.6};
	\def\yoffbig{2.9};
	\def\xoffmed{3*\xoffsmall};
	\def\xoffsmall{0.3};
	\def\xoffbase{3};
	\def\xbufferwidth{2};
	\def\yoffglabel{-0.35};
	\def\nubradius{0.25*\xoffmed};
	\def\diaopac{nearly transparent};
	\def\fillop{25};
	\def\braceextra{0.15*\xoffmed};

	\coordinate (base1) at (0*\xoffbase,0);
	\coordinate (baseCp) at (1*\xoffbase,0);
\path [fill=blue!\fillop] (base1) circle (\nubradius);
\foreach \flowernum in {0,1} {
	\coordinate (base) at (\flowernum*\xoffbase,0);
	\coordinate (start) at ($(base)+(-1.5*\xoffsmall + 1*\xoffsmall,\yoffbig)$);
	\path [fill=blue!\fillop,rounded corners=8pt] ($(start)+(-1*\nubradius,-1*\nubradius)$) rectangle ($(start)+(\nubradius,\nubradius)+(\nbranches*\xoffsmall+1*\xoffsmall,\njoints*\yoffsmall+2*\yoffsmall)$);
}
\coordinate (base) at (0*\xoffbase,0);
\coordinate (start) at ($(base)+(-1.5*\xoffsmall + 1*\xoffsmall,\yoffbig)$);
\path [fill=blue!\fillop]
($(base1)-(\nubradius,0)$) -- 
($(base1)-(\nubradius,0)-(0.5*\xoffsmall,0)+(0,\yoffbig)$)--
($(base1)-(\nubradius,0)-(0.5*\xoffsmall,0)+(0,\yoffbig)$)--
	($(start)+(\nubradius,0)+(\nbranches*\xoffsmall+1*\xoffsmall,0)$)--
($(base1)+(\nubradius,0)$);

	\coordinate (base) at (-\xbufferwidth-1*\xoffmed,0);
\path [fill=blue!\fillop]
($(base1)-(\nubradius,0)$) -- 
($(base)-(\nubradius,0)+(-1.5*\xoffsmall + 1*\xoffsmall,\yoffbig)$)--
($(base)+(\nubradius,0)+(-1.5*\xoffsmall + 1*\xoffsmall,\yoffbig)$)--
($(base1)+(\nubradius,0)$);
	\foreach \m in {0} {
	\coordinate (start) at ($(base)+(\m*\xoffmed,0)+(-1.5*\xoffsmall + \xoffsmall,\yoffbig)$);
	\path [fill=blue!\fillop, rounded corners=8pt] ($(start)+(-1*\nubradius,-1*\nubradius)$) rectangle ($(start)+(\nubradius,\njoints*\yoffsmall+2*\yoffsmall+\nubradius)$);
	}

\node at ($(base1)+(0,\yoffglabel)$) {$\scriptstyle \cat P(1,1)$};
\node at ($(baseCp)+(0,\yoffglabel)$) {$\scriptstyle \cat P(C_p,1)$};
\foreach \flowernum in {0,1} {
	\coordinate (base) at (\flowernum*\xoffbase,0);
	\draw [fill] ($(base)+(0,0)$) circle [radius=\bradius];
	\foreach \x in {1,2,...,\nbranches} {
		\coordinate (start) at ($(base)+(-1.5*\xoffsmall + \x*\xoffsmall,\yoffbig)$);
		\foreach \n in {0,...,\njoints} {
			\draw [fill]  ($(start)+(0,\n*\yoffsmall)$) circle (\bradius);
		}
		\draw [fill]  ($(start)+(0,\njoints*\yoffsmall+2*\yoffsmall)$) circle (\bradius);
		\draw (start) -- ($(start)+(0,\njoints*\yoffsmall+0.666*\yoffsmall)$);
		\foreach \n in {0,1,2} {
			\draw [fill] ($(start)+(0,\njoints*\yoffsmall+0.666*\yoffsmall+0.4166*\yoffsmall+0.25*\n*\yoffsmall)$) circle (0.25*\bradius);
		}
		\draw (base) -- (start);
	}
	\foreach \n in {1,2,...,\njoints} {
		\coordinate (start) at ($(base)+(\xoffsmall,0)+(-1.5*\xoffsmall + \nbranches*\xoffsmall,\yoffbig)$);
		\node at ($(start)+(0,\n*\yoffsmall-0.5*\yoffsmall)$) {$\hdots$};
	}
	\node at ($(base)+(\xoffsmall,0)+(-1.5*\xoffsmall + \nbranches*\xoffsmall,\yoffbig)+(0,\njoints*\yoffsmall+0.666*\yoffsmall+0.4166*\yoffsmall+0.25*1*\yoffsmall)$) {$\hdots$};
}	
	\coordinate (base) at (0*\xoffbase,0);
	\coordinate (start) at ($(base)+(-1.5*\xoffsmall + 4*\xoffsmall,\yoffbig)$);
	\node at ($(start)+(0,0.52)+(0,\njoints*\yoffsmall+2*\yoffsmall)$) {$\scriptstyle \substack{\cat P(1,q,n)\;\;\cdots\\(q\neq p,\,n \ge 2)}$};
	\coordinate (base) at (1*\xoffbase,0);
	\coordinate (start) at ($(base)+(-1.5*\xoffsmall + 4*\xoffsmall,\yoffbig)$);
	\node at ($(start)+(0,0.52)+(0,\njoints*\yoffsmall+2*\yoffsmall)$) {$\scriptstyle \substack{\cat P(C_p,q,n)\;\;\cdots\\\hspace{-1.5ex}(q\neq p,\,n \ge 2)}$};
	\coordinate (base) at (-\xbufferwidth-1*\xoffmed,0);
	\foreach \n in {2,3,...,\njointsplustwo} {
		\node at ($(base)-(0.6,0)+(0,\n*\yoffsmall-2*\yoffsmall)-(\nubradius,0)+(-1.5*\xoffsmall + 1*\xoffsmall,\yoffbig)$) {$\scriptstyle \cat P(1,p,\n)$};
		\node at ($(base)+(0.5,0)+(0,\n*\yoffsmall-2*\yoffsmall)+(\xoffmed,0)+(\nubradius,0)+(-1.5*\xoffsmall + 1*\xoffsmall,\yoffbig)$) {$\scriptstyle \cat P(C_p,p,\n)$};
	}
	\node at ($(base)-(0.6,0)-(\nubradius,0)+(-0.5*\xoffsmall,\yoffbig)+(0,\njoints*\yoffsmall+2*\yoffsmall)$) {$\scriptstyle \cat P(1,p,\infty) $};
	\node at ($(base)+(0.57,0)+(\nubradius,0)+(\xoffmed,0)+(-0.5*\xoffsmall,\yoffbig)+(0,\njoints*\yoffsmall+2*\yoffsmall)$) {$\scriptstyle \cat P(C_p,p,\infty) $};



	\coordinate (base) at (-\xbufferwidth-1*\xoffmed,0);
	\foreach \m in {0,1} {
		\coordinate (start) at ($(base)+(\m*\xoffmed,0)+(-1.5*\xoffsmall + \xoffsmall,\yoffbig)$);
		\foreach \n in {0,1,...,\njoints} {
				\draw [fill]  ($(start)+(0,\n*\yoffsmall)$) circle (\bradius);
			\draw [fill]  ($(start)+(0,\njoints*\yoffsmall+2*\yoffsmall)$) circle (\bradius);
			\draw (start) -- ($(start)+(0,\njoints*\yoffsmall+0.666*\yoffsmall)$);
			\foreach \n in {0,1,2} {
				\draw [fill] ($(start)+(0,\njoints*\yoffsmall+0.666*\yoffsmall+0.4166*\yoffsmall+0.25*\n*\yoffsmall)$) circle (0.25*\bradius);
			}
		}
	}
	\foreach \m in {0} {
		\coordinate (start) at ($(base)+(1*\xoffmed+\m*2*\xoffmed,0)+(-1.5*\xoffsmall + \xoffsmall,\yoffbig)$);
		\foreach \n in {1,...,\njoints} {
			\draw ($(start)+(0,\n*\yoffsmall-\yoffsmall)$) -- ($(start)+(0,\n*\yoffsmall)-(\xoffmed,0)$);
		}
		\begin{scope}
			\path [clip] ($(start)+(-1*\xoffmed,\njoints*\yoffsmall+0.666*\yoffsmall)$) rectangle ($(start)+(0,\njoints*\yoffsmall)$);
			\draw ($(start)+(0,\njoints*\yoffsmall)$) -- ($(start)+(0,\njoints*\yoffsmall+\yoffsmall)-(\xoffmed,0)$);
		\end{scope}
		\draw ($(start)+(0,\njoints*\yoffsmall+2*\yoffsmall)$) -- ($(start)+(-1*\xoffmed,\njoints*\yoffsmall+2*\yoffsmall)$);
	}
	\draw (base1) -- ($(base)+(0*\xoffmed,0)+(-1.5*\xoffsmall+\xoffsmall,\yoffbig)$);
	\draw (baseCp) -- ($(base)+(0*\xoffmed,0)+(-1.5*\xoffsmall+\xoffsmall,\yoffbig)$);
	\draw (baseCp) -- ($(base)+(1*\xoffmed,0)+(-1.5*\xoffsmall+\xoffsmall,\yoffbig)$);

\end{tikzpicture}
\caption{The \locus{} of $\Infl_{G/N}^G$ for $G=N=C_p$.}\label{fig:Cp}
\end{figure}
\end{Exa}
\begin{Exa}\label{exa:locus-for-D_10}
	The \locus{} for $G=D_{10}$, the dihedral group of order 10, is displayed in Figure~\ref{fig:D10}
	(for $N=G$)
	and in Figure~\ref{fig:D10C5} (for $N=C_5$) on pages~\pageref{fig:D10}--\pageref{fig:D10C5}.
	(We have a complete understanding of the topology 
	of the spectrum in this example since $D_{10}$ is a group of square-free order; cf.~\cite[Thm~8.12]{BalmerSanders17}.)
	The group $D_{10}$ has a unique (normal) subgroup of order $5$, and has five subgroups of order $2$ (forming a single conjugacy class).
	Observe the different behavior at the primes 2 and 5 due to the fact that $C_5$ is normal while the copies of $C_2$ are not.
\end{Exa}
\begin{figure}
\begin{sideways}
\begin{minipage}{\textheight}
\begin{tikzpicture}
	\def\bradius{0.062};
	\def\nbranches{6}
	\def\njoints{5}
	\def\yoffsmall{0.6};
	\def\yoffbig{4};
	\def\xoffmed{3*\xoffsmall};
	\def\xoffsmall{0.3};
	\def\xoffbase{3};
	\def\xbufferwidth{2};
	\def\yoffglabel{-0.5};
	\def\nubradius{0.25*\xoffmed};
	\def\diaopac{nearly transparent};
	\def\fillop{25};
	\def\braceextra{0.15*\xoffmed};

	\coordinate (base1) at (0*\xoffbase,0);
	\coordinate (baseC2) at (1*\xoffbase,0);
	\coordinate (baseC5) at (2*\xoffbase,0);
	\coordinate (baseD10) at (3*\xoffbase,0);
\foreach \flowernum in {1,2,3} {
	\coordinate (base) at (\flowernum*\xoffbase,0);
	\coordinate (start) at ($(base)+(-1.5*\xoffsmall + 1*\xoffsmall,\yoffbig)$);
	\path [fill=blue!\fillop,rounded corners=8pt] ($(start)+(-1*\nubradius,-1*\nubradius)$) rectangle ($(start)+(\nubradius,\nubradius)+(\nbranches*\xoffsmall+1*\xoffsmall,\njoints*\yoffsmall+2*\yoffsmall)$);
}
\begin{scope}
	\coordinate (base) at (0*\xoffbase,0);
	\coordinate (start) at ($(base)+(-1.5*\xoffsmall + 1*\xoffsmall,\yoffbig)$);
\path [fill=blue!\fillop, rounded corners=8pt] 
($(base1)-(\nubradius,0)$) -- ($(start)+(-1*\nubradius,-1*\nubradius)$) --
($(start)+(-1*\nubradius,\nubradius)+(0,\njoints*\yoffsmall+2*\yoffsmall)$) -- 
($(start)+(\nubradius,\nubradius)+(\nbranches*\xoffsmall+1*\xoffsmall,\njoints*\yoffsmall+2*\yoffsmall)$) -- 
($(start)+(\nubradius,0*\nubradius)+(\nbranches*\xoffsmall+1*\xoffsmall,0)$) -- 
($(base1)+(\nubradius,0)$);
\end{scope}
\path [fill=blue!\fillop] (base1) circle (\nubradius);
	\coordinate (base) at (\xbufferwidth+3*\xoffbase+\nbranches*\xoffsmall,0);
	\foreach \m in {0,1,3} {
	\coordinate (start) at ($(base)+(\m*\xoffmed,0)+(-1.5*\xoffsmall + \xoffsmall,\yoffbig)$);
	\path [fill=blue!\fillop] ($(start)+(-1*\nubradius,0)$) rectangle ($(start)+(\nubradius,\njoints*\yoffsmall+2*\yoffsmall)$);
	\path [fill=blue!\fillop] (start) circle (\nubradius);
	\path [fill=blue!\fillop] ($(start)+(0,\njoints*\yoffsmall+2*\yoffsmall)$) circle (\nubradius);
	}
	\coordinate (start) at ($(base)+(0*\xoffmed,0)+(-1.5*\xoffsmall + \xoffsmall,\yoffbig)$);
	\path [fill=blue!\fillop] ($(base1)+(0,\nubradius)$) -- ($(start)+(0,1*\nubradius)$) -- ($(start)+(0,-1*\nubradius)$) -- ($(base1)+(0,-1*\nubradius)$);
	\coordinate (base) at (-\xbufferwidth-3*\xoffmed,0);
	\coordinate (start) at ($(base)+(2*\xoffmed,0)+(-1.5*\xoffsmall + \xoffsmall,\yoffbig)$);
	\path [fill=blue!\fillop] ($(start)+(-1*\nubradius,0)$) rectangle ($(start)+(\nubradius,\njoints*\yoffsmall+2*\yoffsmall)$);
	\path [fill=blue!\fillop] (start) circle (\nubradius);
	\path [fill=blue!\fillop] ($(start)+(0,\njoints*\yoffsmall+2*\yoffsmall)$) circle (\nubradius);

	\coordinate (start) at ($(base)+(0*\xoffmed,0)+(-1.5*\xoffsmall + \xoffsmall,\yoffbig)$);
	\path [fill=blue!\fillop] ($(start)+(-1*\nubradius,0)$) -- ($(start)+(-1*\nubradius,\njoints*\yoffsmall+2*\yoffsmall)$) -- ($(start)+(1*\nubradius,\njoints*\yoffsmall+2*\yoffsmall)$) -- ($(start)+(1*\nubradius,0)$);
	\path [fill=blue!\fillop] ($(base1)+(0,\nubradius)$) -- ($(start)+(0,\nubradius)$) -- ($(start)+(0,-1*\nubradius)$) -- ($(base1)+(0,-1*\nubradius)$);
	\path [fill=blue!\fillop] (start) circle (\nubradius);
	\path [fill=blue!\fillop] ($(start)+(0,\njoints*\yoffsmall+2*\yoffsmall)$) circle (\nubradius);

	\node at ($(base1)+(0,\yoffglabel)$) {$1$};
	\node at ($(baseC2)+(0,\yoffglabel)$) {$C_2$};
	\node at ($(baseC5)+(0,\yoffglabel)$) {$C_5$};
	\node at ($(baseD10)+(0,\yoffglabel)$) {$D_{10}$};
\foreach \flowernum in {0,1,2,3} {
	\coordinate (base) at (\flowernum*\xoffbase,0);
	\draw [fill] ($(base)+(0,0)$) circle [radius=\bradius];
	\foreach \x in {1,2,...,\nbranches} {
		\coordinate (start) at ($(base)+(-1.5*\xoffsmall + \x*\xoffsmall,\yoffbig)$);
		\foreach \n in {0,...,\njoints} {
			\draw [fill]  ($(start)+(0,\n*\yoffsmall)$) circle (\bradius);
		}
		\draw [fill]  ($(start)+(0,\njoints*\yoffsmall+2*\yoffsmall)$) circle (\bradius);
		\draw (start) -- ($(start)+(0,\njoints*\yoffsmall+0.666*\yoffsmall)$);
		\foreach \n in {0,1,2} {
			\draw [fill] ($(start)+(0,\njoints*\yoffsmall+0.666*\yoffsmall+0.4166*\yoffsmall+0.25*\n*\yoffsmall)$) circle (0.25*\bradius);
		}
		\draw (base) -- (start);
	}
	\foreach \n in {1,2,...,\njoints} {
		\coordinate (start) at ($(base)+(\xoffsmall,0)+(-1.5*\xoffsmall + \nbranches*\xoffsmall,\yoffbig)$);
		\node at ($(start)+(0,\n*\yoffsmall-0.5*\yoffsmall)$) {$\hdots$};
	}
	\node at ($(base)+(\xoffsmall,0)+(-1.5*\xoffsmall + \nbranches*\xoffsmall,\yoffbig)+(0,\njoints*\yoffsmall+0.666*\yoffsmall+0.4166*\yoffsmall+0.25*1*\yoffsmall)$) {$\hdots$};

}


	\coordinate (base) at (-\xbufferwidth-3*\xoffmed,0);
	\foreach \m in {0,1,2,3} {
		\coordinate (start) at ($(base)+(\m*\xoffmed,0)+(-1.5*\xoffsmall + \xoffsmall,\yoffbig)$);
		\foreach \n in {0,1,...,\njoints} {
				\draw [fill]  ($(start)+(0,\n*\yoffsmall)$) circle (\bradius);
			\draw [fill]  ($(start)+(0,\njoints*\yoffsmall+2*\yoffsmall)$) circle (\bradius);
			\draw (start) -- ($(start)+(0,\njoints*\yoffsmall+0.666*\yoffsmall)$);
			\foreach \n in {0,1,2} {
				\draw [fill] ($(start)+(0,\njoints*\yoffsmall+0.666*\yoffsmall+0.4166*\yoffsmall+0.25*\n*\yoffsmall)$) circle (0.25*\bradius);
			}
		}
	}
	\foreach \m in {0,1} {
		\coordinate (start) at ($(base)+(1*\xoffmed+\m*2*\xoffmed,0)+(-1.5*\xoffsmall + \xoffsmall,\yoffbig)$);
		\foreach \n in {1,...,\njoints} {
			\draw ($(start)+(0,\n*\yoffsmall-\yoffsmall)$) -- ($(start)+(0,\n*\yoffsmall)-(\xoffmed,0)$);
		}
		\begin{scope}
			\path [clip] ($(start)+(-1*\xoffmed,\njoints*\yoffsmall+0.666*\yoffsmall)$) rectangle ($(start)+(0,\njoints*\yoffsmall)$);
			\draw ($(start)+(0,\njoints*\yoffsmall)$) -- ($(start)+(0,\njoints*\yoffsmall+\yoffsmall)-(\xoffmed,0)$);
		\end{scope}
		\draw ($(start)+(0,\njoints*\yoffsmall+2*\yoffsmall)$) -- ($(start)+(-1*\xoffmed,\njoints*\yoffsmall+2*\yoffsmall)$);
	}
	\draw (base1) -- ($(base)+(0*\xoffmed,0)+(-1.5*\xoffsmall+\xoffsmall,\yoffbig)$);
	\draw (baseC2) -- ($(base)+(0*\xoffmed,0)+(-1.5*\xoffsmall+\xoffsmall,\yoffbig)$);
	\draw (baseC2) -- ($(base)+(1*\xoffmed,0)+(-1.5*\xoffsmall+\xoffsmall,\yoffbig)$);
	\draw (baseC5) -- ($(base)+(2*\xoffmed,0)+(-1.5*\xoffsmall+\xoffsmall,\yoffbig)$);
	\draw (baseD10) -- ($(base)+(2*\xoffmed,0)+(-1.5*\xoffsmall+\xoffsmall,\yoffbig)$);
	\draw (baseD10) -- ($(base)+(3*\xoffmed,0)+(-1.5*\xoffsmall+\xoffsmall,\yoffbig)$);

	\draw[snake=brace, thick] 
	($(base)+(\braceextra,0)+(3*\xoffmed,0)+(-1.5*\xoffsmall+\xoffsmall,0.25*\yoffbig)$)
	-- ($(base)-(\braceextra,0)+(0*\xoffmed,0)+(-1.5*\xoffsmall+\xoffsmall,0.25*\yoffbig)$);
	 \node at ($(base)+(1.5*\xoffmed,0)+(-1.5*\xoffsmall+\xoffsmall,0.25*\yoffbig+\yoffglabel)$) {$p=2$};

	\coordinate (base) at (\xbufferwidth+3*\xoffbase+\nbranches*\xoffsmall,0);

	\foreach \m in {0,1,2,3} {
		\coordinate (start) at ($(base)+(\m*\xoffmed,0)+(-1.5*\xoffsmall + \xoffsmall,\yoffbig)$);
		\foreach \n in {0,1,...,\njoints} {
				\draw [fill]  ($(start)+(0,\n*\yoffsmall)$) circle (\bradius);
			\draw [fill]  ($(start)+(0,\njoints*\yoffsmall+2*\yoffsmall)$) circle (\bradius);
			\draw (start) -- ($(start)+(0,\njoints*\yoffsmall+0.666*\yoffsmall)$);
			\foreach \n in {0,1,2} {
				\draw [fill] ($(start)+(0,\njoints*\yoffsmall+0.666*\yoffsmall+0.4166*\yoffsmall+0.25*\n*\yoffsmall)$) circle (0.25*\bradius);
			}
		}
	}

	\coordinate (start) at ($(base)+(2*\xoffmed,0)+(-1.5*\xoffsmall + \xoffsmall,\yoffbig)$);
	\foreach \n in {1,...,\njoints} {
		\draw ($(start)+(0,\n*\yoffsmall-\yoffsmall)$) -- ($(start)+(0,\n*\yoffsmall)-(2*\xoffmed,0)$);
	}
	\begin{scope}
		\path [clip] ($(start)+(-2*\xoffmed,\njoints*\yoffsmall+0.666*\yoffsmall)$) rectangle ($(start)+(0,\njoints*\yoffsmall)$);
		\draw ($(start)+(0,\njoints*\yoffsmall)$) -- ($(start)+(0,\njoints*\yoffsmall+\yoffsmall)-(2*\xoffmed,0)$);
	\end{scope}
	\draw ($(start)+(0,\njoints*\yoffsmall+2*\yoffsmall)$) to [out=200,in=340] ($(start)+(-2*\xoffmed,\njoints*\yoffsmall+2*\yoffsmall)$);
	\draw (base1) -- ($(base)+(0*\xoffmed,0)+(-1.5*\xoffsmall+\xoffsmall,\yoffbig)$);
	\draw (baseC5) -- ($(base)+(0*\xoffmed,0)+(-1.5*\xoffsmall+\xoffsmall,\yoffbig)$);
	\draw (baseC2) -- ($(base)+(1*\xoffmed,0)+(-1.5*\xoffsmall+\xoffsmall,\yoffbig)$);
	\draw (baseC5) -- ($(base)+(2*\xoffmed,0)+(-1.5*\xoffsmall+\xoffsmall,\yoffbig)$);
	\draw (baseD10) -- ($(base)+(3*\xoffmed,0)+(-1.5*\xoffsmall+\xoffsmall,\yoffbig)$);
	\draw[snake=brace, thick] 
	($(base)+(\braceextra,0)+(3*\xoffmed,0)+(-1.5*\xoffsmall+\xoffsmall,0.25*\yoffbig)$)
	-- ($(base)-(\braceextra,0)+(0*\xoffmed,0)+(-1.5*\xoffsmall+\xoffsmall,0.25*\yoffbig)$);
	 \node at ($(base)+(1.5*\xoffmed,0)+(-1.5*\xoffsmall+\xoffsmall,0.25*\yoffbig+\yoffglabel)$) {$p=5$};

\end{tikzpicture}
\end{minipage}
\end{sideways}
\caption{The \locus{} of $\Infl_{G/N}^G$ for $G=N=D_{10}$.}\label{fig:D10}
\end{figure}
\begin{figure}
\begin{sideways}
\begin{minipage}{\textheight}
\begin{tikzpicture}
	\def\bradius{0.062};
	\def\nbranches{6}
	\def\njoints{5}
	\def\yoffsmall{0.6};
	\def\yoffbig{4};
	\def\xoffmed{3*\xoffsmall};
	\def\xoffsmall{0.3};
	\def\xoffbase{3};
	\def\xbufferwidth{2};
	\def\yoffglabel{-0.5};
	\def\nubradius{0.25*\xoffmed};
	\def\dnubradius{0.15*\xoffmed};
	\def\diaopac{nearly transparent};
	\def\fillop{25};
	\def\braceextra{0.15*\xoffmed};

	\coordinate (base1) at (0*\xoffbase,0);
	\coordinate (baseC2) at (1*\xoffbase,0);
	\coordinate (baseC5) at (2*\xoffbase,0);
	\coordinate (baseD10) at (3*\xoffbase,0);
\foreach \flowernum in {1,2,3} {
	\coordinate (base) at (\flowernum*\xoffbase,0);
	\coordinate (start) at ($(base)+(-1.5*\xoffsmall + 1*\xoffsmall,\yoffbig)$);
	\path [fill=blue!\fillop,rounded corners=8pt] ($(start)+(-1*\nubradius,-1*\nubradius)$) rectangle ($(start)+(\nubradius,\nubradius)+(\nbranches*\xoffsmall+1*\xoffsmall,\njoints*\yoffsmall+2*\yoffsmall)$);
}
\begin{scope}
	\coordinate (base) at (0*\xoffbase,0);
	\coordinate (start) at ($(base)+(-1.5*\xoffsmall + 1*\xoffsmall,\yoffbig)$);
\path [fill=blue!\fillop, rounded corners=8pt] 
($(base1)-(\nubradius,0)$) -- ($(start)+(-1*\nubradius,-1*\nubradius)$) --
($(start)+(-1*\nubradius,\nubradius)+(0,\njoints*\yoffsmall+2*\yoffsmall)$) -- 
($(start)+(\nubradius,\nubradius)+(\nbranches*\xoffsmall+1*\xoffsmall,\njoints*\yoffsmall+2*\yoffsmall)$) -- 
($(start)+(\nubradius,0*\nubradius)+(\nbranches*\xoffsmall+1*\xoffsmall,0)$) -- 
($(base1)+(\nubradius,0)$);
\end{scope}
\begin{scope}
	\coordinate (base) at (1*\xoffbase,0);
	\coordinate (start) at ($(base)+(-1.5*\xoffsmall + 1*\xoffsmall,\yoffbig)$);
\path [fill=blue!\fillop, rounded corners=8pt] 
($(baseC2)-(\nubradius,0)$) -- ($(start)+(-1*\nubradius,-1*\nubradius)$) --
($(start)+(-1*\nubradius,\nubradius)+(0,\njoints*\yoffsmall+2*\yoffsmall)$) -- 
($(start)+(\nubradius,\nubradius)+(\nbranches*\xoffsmall+1*\xoffsmall,\njoints*\yoffsmall+2*\yoffsmall)$) -- 
($(start)+(\nubradius,0*\nubradius)+(\nbranches*\xoffsmall+1*\xoffsmall,0)$) -- 
($(baseC2)+(\nubradius,0)$);
\end{scope}
\begin{scope}
	\coordinate (base) at (3*\xoffbase,0);
	\coordinate (start) at ($(base)+(-1.5*\xoffsmall + 1*\xoffsmall,\yoffbig)$);
\path [fill=blue!\fillop, rounded corners=8pt] 
($(baseD10)-(\nubradius,0)$) -- ($(start)+(-1*\nubradius,-1*\nubradius)$) --
($(start)+(-1*\nubradius,\nubradius)+(0,\njoints*\yoffsmall+2*\yoffsmall)$) -- 
($(start)+(\nubradius,\nubradius)+(\nbranches*\xoffsmall+1*\xoffsmall,\njoints*\yoffsmall+2*\yoffsmall)$) -- 
($(start)+(\nubradius,0*\nubradius)+(\nbranches*\xoffsmall+1*\xoffsmall,0)$) -- 
($(baseD10)+(\nubradius,0)$);
\end{scope}
\path [fill=blue!\fillop] (base1) circle (\nubradius);
\path [fill=blue!\fillop] (baseC2) circle (\nubradius);
\path [fill=blue!\fillop] (baseD10) circle (\nubradius);
	\coordinate (base) at (\xbufferwidth+3*\xoffbase+\nbranches*\xoffsmall,0);
	\foreach \m in {0,1,3} {
	\coordinate (start) at ($(base)+(\m*\xoffmed,0)+(-1.5*\xoffsmall + \xoffsmall,\yoffbig)$);
	\path [fill=blue!\fillop] ($(start)+(-1*\nubradius,0)$) rectangle ($(start)+(\nubradius,\njoints*\yoffsmall+2*\yoffsmall)$);
	\path [fill=blue!\fillop] (start) circle (\nubradius);
	\path [fill=blue!\fillop] ($(start)+(0,\njoints*\yoffsmall+2*\yoffsmall)$) circle (\nubradius);
	}
	\coordinate (start) at ($(base)+(0*\xoffmed,0)+(-1.5*\xoffsmall + \xoffsmall,\yoffbig)$);
	\path [fill=blue!\fillop] ($(base1)+(0,\nubradius)$) -- ($(start)+(0,1*\nubradius)$) -- ($(start)+(0,-1*\nubradius)$) -- ($(base1)+(0,-1*\nubradius)$);
	\coordinate (start) at ($(base)+(1*\xoffmed,0)+(-1.5*\xoffsmall + \xoffsmall,\yoffbig)$);
	\path [fill=blue!\fillop] ($(baseC2)+(0,\nubradius)$) -- ($(start)+(0,1*\nubradius)$) -- ($(start)+(0,-1*\nubradius)$) -- ($(baseC2)+(0,-1*\nubradius)$);
	\coordinate (start) at ($(base)+(3*\xoffmed,0)+(-1.5*\xoffsmall + \xoffsmall,\yoffbig)$);
	\path [fill=blue!\fillop] ($(baseD10)+(0,\nubradius)$) -- ($(start)+(0,1*\nubradius)$) -- ($(start)+(0,-1*\nubradius)$) -- ($(baseD10)+(0,-1*\nubradius)$);
	\coordinate (base) at (-\xbufferwidth-3*\xoffmed,0);
	\foreach \n in {0,1,2,3} {
		\coordinate (start) at ($(base)+(\n*\xoffmed,0)+(-1.5*\xoffsmall + \xoffsmall,\yoffbig)$);
		\path [fill=blue!\fillop] ($(start)+(-1*\nubradius,0)$) rectangle ($(start)+(\nubradius,\njoints*\yoffsmall+2*\yoffsmall)$);
		\path [fill=blue!\fillop] (start) circle (\nubradius);
		\path [fill=blue!\fillop] ($(start)+(0,\njoints*\yoffsmall+2*\yoffsmall)$) circle (\nubradius);
	}
	\coordinate (start) at ($(base)+(0*\xoffmed,0)+(-1.5*\xoffsmall + \xoffsmall,\yoffbig)$);
	\path [fill=blue!\fillop] ($(base1)+(0,\nubradius)$) -- ($(start)+(0,\nubradius)$) -- ($(start)+(0,-1*\nubradius)$) -- ($(base1)+(0,-1*\nubradius)$);
	\path [fill=blue!\fillop] ($(baseC2)+(0,\nubradius)$) -- ($(start)+(0,\nubradius)$) -- ($(start)+(0,-1*\nubradius)$) -- ($(baseC2)+(0,-1*\nubradius)$);
	\coordinate (start) at ($(base)+(1*\xoffmed,0)+(-1.5*\xoffsmall + \xoffsmall,\yoffbig)$);
	\path [fill=blue!\fillop] ($(baseC2)+(0,\nubradius)$) -- ($(start)+(0,\nubradius)$) -- ($(start)+(0,-1*\nubradius)$) -- ($(baseC2)+(0,-1*\nubradius)$);
	\coordinate (start) at ($(base)+(2*\xoffmed,0)+(-1.5*\xoffsmall + \xoffsmall,\yoffbig)$);
	\path [fill=blue!\fillop] ($(baseD10)+(0,\nubradius)$) -- ($(start)+(0,\nubradius)$) -- ($(start)+(0,-1*\nubradius)$) -- ($(baseD10)+(0,-1*\nubradius)$);
	\coordinate (start) at ($(base)+(3*\xoffmed,0)+(-1.5*\xoffsmall + \xoffsmall,\yoffbig)$);
	\path [fill=blue!\fillop] ($(baseD10)+(0,\nubradius)$) -- ($(start)+(0,\nubradius)$) -- ($(start)+(0,-1*\nubradius)$) -- ($(baseD10)+(0,-1*\nubradius)$);

	\foreach \m in {0,1} {
		\coordinate (start) at ($(base)+(1*\xoffmed+\m*2*\xoffmed,0)+(-1.5*\xoffsmall + \xoffsmall,\yoffbig)$);
		\foreach \n in {0,...,\njoints} {
			\path [fill=blue!\fillop] ($(start)+(0,-1*\dnubradius)+(0,\n*\yoffsmall)$) -- ($(start)+(0,-1*\dnubradius)+(0,\yoffsmall+\n*\yoffsmall)-(\xoffmed,0)$)
			-- ($(start)+(0,1*\dnubradius)+(0,\n*\yoffsmall+\yoffsmall)-(\xoffmed,0)$) -- ($(start)+(0,1*\dnubradius)+(0,\n*\yoffsmall)$);
		}
		\path [fill=blue!\fillop] ($(start)+(0,-1*\dnubradius)+(0,\njoints*\yoffsmall+2*\yoffsmall)$) --
		($(start)+(0,-1*\dnubradius)+(-1*\xoffmed,\njoints*\yoffsmall+2*\yoffsmall)$) --
		($(start)+(0,1*\dnubradius)+(-1*\xoffmed,\njoints*\yoffsmall+2*\yoffsmall)$) --
		($(start)+(0,1*\dnubradius)+(0,\njoints*\yoffsmall+2*\yoffsmall)$);
	}

	\node at ($(base1)+(0,\yoffglabel)$) {$1$};
	\node at ($(baseC2)+(0,\yoffglabel)$) {$C_2$};
	\node at ($(baseC5)+(0,\yoffglabel)$) {$C_5$};
	\node at ($(baseD10)+(0,\yoffglabel)$) {$D_{10}$};
\foreach \flowernum in {0,1,2,3} {
	\coordinate (base) at (\flowernum*\xoffbase,0);
	\draw [fill] ($(base)+(0,0)$) circle [radius=\bradius];
	\foreach \x in {1,2,...,\nbranches} {
		\coordinate (start) at ($(base)+(-1.5*\xoffsmall + \x*\xoffsmall,\yoffbig)$);
		\foreach \n in {0,...,\njoints} {
			\draw [fill]  ($(start)+(0,\n*\yoffsmall)$) circle (\bradius);
		}
		\draw [fill]  ($(start)+(0,\njoints*\yoffsmall+2*\yoffsmall)$) circle (\bradius);
		\draw (start) -- ($(start)+(0,\njoints*\yoffsmall+0.666*\yoffsmall)$);
		\foreach \n in {0,1,2} {
			\draw [fill] ($(start)+(0,\njoints*\yoffsmall+0.666*\yoffsmall+0.4166*\yoffsmall+0.25*\n*\yoffsmall)$) circle (0.25*\bradius);
		}
		\draw (base) -- (start);
	}
	\foreach \n in {1,2,...,\njoints} {
		\coordinate (start) at ($(base)+(\xoffsmall,0)+(-1.5*\xoffsmall + \nbranches*\xoffsmall,\yoffbig)$);
		\node at ($(start)+(0,\n*\yoffsmall-0.5*\yoffsmall)$) {$\hdots$};
	}
	\node at ($(base)+(\xoffsmall,0)+(-1.5*\xoffsmall + \nbranches*\xoffsmall,\yoffbig)+(0,\njoints*\yoffsmall+0.666*\yoffsmall+0.4166*\yoffsmall+0.25*1*\yoffsmall)$) {$\hdots$};

}


	\coordinate (base) at (-\xbufferwidth-3*\xoffmed,0);
	\foreach \m in {0,1,2,3} {
		\coordinate (start) at ($(base)+(\m*\xoffmed,0)+(-1.5*\xoffsmall + \xoffsmall,\yoffbig)$);
		\foreach \n in {0,1,...,\njoints} {
				\draw [fill]  ($(start)+(0,\n*\yoffsmall)$) circle (\bradius);
			\draw [fill]  ($(start)+(0,\njoints*\yoffsmall+2*\yoffsmall)$) circle (\bradius);
			\draw (start) -- ($(start)+(0,\njoints*\yoffsmall+0.666*\yoffsmall)$);
			\foreach \n in {0,1,2} {
				\draw [fill] ($(start)+(0,\njoints*\yoffsmall+0.666*\yoffsmall+0.4166*\yoffsmall+0.25*\n*\yoffsmall)$) circle (0.25*\bradius);
			}
		}
	}
	\foreach \m in {0,1} {
		\coordinate (start) at ($(base)+(1*\xoffmed+\m*2*\xoffmed,0)+(-1.5*\xoffsmall + \xoffsmall,\yoffbig)$);
		\foreach \n in {1,...,\njoints} {
			\draw ($(start)+(0,\n*\yoffsmall-\yoffsmall)$) -- ($(start)+(0,\n*\yoffsmall)-(\xoffmed,0)$);
		}
		\begin{scope}
			\path [clip] ($(start)+(-1*\xoffmed,\njoints*\yoffsmall+0.666*\yoffsmall)$) rectangle ($(start)+(0,\njoints*\yoffsmall)$);
			\draw ($(start)+(0,\njoints*\yoffsmall)$) -- ($(start)+(0,\njoints*\yoffsmall+\yoffsmall)-(\xoffmed,0)$);
		\end{scope}
		\draw ($(start)+(0,\njoints*\yoffsmall+2*\yoffsmall)$) -- ($(start)+(-1*\xoffmed,\njoints*\yoffsmall+2*\yoffsmall)$);
	}
	\draw (base1) -- ($(base)+(0*\xoffmed,0)+(-1.5*\xoffsmall+\xoffsmall,\yoffbig)$);
	\draw (baseC2) -- ($(base)+(0*\xoffmed,0)+(-1.5*\xoffsmall+\xoffsmall,\yoffbig)$);
	\draw (baseC2) -- ($(base)+(1*\xoffmed,0)+(-1.5*\xoffsmall+\xoffsmall,\yoffbig)$);
	\draw (baseC5) -- ($(base)+(2*\xoffmed,0)+(-1.5*\xoffsmall+\xoffsmall,\yoffbig)$);
	\draw (baseD10) -- ($(base)+(2*\xoffmed,0)+(-1.5*\xoffsmall+\xoffsmall,\yoffbig)$);
	\draw (baseD10) -- ($(base)+(3*\xoffmed,0)+(-1.5*\xoffsmall+\xoffsmall,\yoffbig)$);

	\draw[snake=brace, thick] 
	($(base)+(\braceextra,0)+(3*\xoffmed,0)+(-1.5*\xoffsmall+\xoffsmall,0.25*\yoffbig)$)
	-- ($(base)-(\braceextra,0)+(0*\xoffmed,0)+(-1.5*\xoffsmall+\xoffsmall,0.25*\yoffbig)$);
	 \node at ($(base)+(1.5*\xoffmed,0)+(-1.5*\xoffsmall+\xoffsmall,0.25*\yoffbig+\yoffglabel)$) {$p=2$};

	\coordinate (base) at (\xbufferwidth+3*\xoffbase+\nbranches*\xoffsmall,0);

	\foreach \m in {0,1,2,3} {
		\coordinate (start) at ($(base)+(\m*\xoffmed,0)+(-1.5*\xoffsmall + \xoffsmall,\yoffbig)$);
		\foreach \n in {0,1,...,\njoints} {
				\draw [fill]  ($(start)+(0,\n*\yoffsmall)$) circle (\bradius);
			\draw [fill]  ($(start)+(0,\njoints*\yoffsmall+2*\yoffsmall)$) circle (\bradius);
			\draw (start) -- ($(start)+(0,\njoints*\yoffsmall+0.666*\yoffsmall)$);
			\foreach \n in {0,1,2} {
				\draw [fill] ($(start)+(0,\njoints*\yoffsmall+0.666*\yoffsmall+0.4166*\yoffsmall+0.25*\n*\yoffsmall)$) circle (0.25*\bradius);
			}
		}
	}

	\coordinate (start) at ($(base)+(2*\xoffmed,0)+(-1.5*\xoffsmall + \xoffsmall,\yoffbig)$);
	\foreach \n in {1,...,\njoints} {
		\draw ($(start)+(0,\n*\yoffsmall-\yoffsmall)$) -- ($(start)+(0,\n*\yoffsmall)-(2*\xoffmed,0)$);
	}
	\begin{scope}
		\path [clip] ($(start)+(-2*\xoffmed,\njoints*\yoffsmall+0.666*\yoffsmall)$) rectangle ($(start)+(0,\njoints*\yoffsmall)$);
		\draw ($(start)+(0,\njoints*\yoffsmall)$) -- ($(start)+(0,\njoints*\yoffsmall+\yoffsmall)-(2*\xoffmed,0)$);
	\end{scope}
	\draw ($(start)+(0,\njoints*\yoffsmall+2*\yoffsmall)$) to [out=200,in=340] ($(start)+(-2*\xoffmed,\njoints*\yoffsmall+2*\yoffsmall)$);
	\draw (base1) -- ($(base)+(0*\xoffmed,0)+(-1.5*\xoffsmall+\xoffsmall,\yoffbig)$);
	\draw (baseC5) -- ($(base)+(0*\xoffmed,0)+(-1.5*\xoffsmall+\xoffsmall,\yoffbig)$);
	\draw (baseC2) -- ($(base)+(1*\xoffmed,0)+(-1.5*\xoffsmall+\xoffsmall,\yoffbig)$);
	\draw (baseC5) -- ($(base)+(2*\xoffmed,0)+(-1.5*\xoffsmall+\xoffsmall,\yoffbig)$);
	\draw (baseD10) -- ($(base)+(3*\xoffmed,0)+(-1.5*\xoffsmall+\xoffsmall,\yoffbig)$);
	\draw[snake=brace, thick] 
	($(base)+(\braceextra,0)+(3*\xoffmed,0)+(-1.5*\xoffsmall+\xoffsmall,0.25*\yoffbig)$)
	-- ($(base)-(\braceextra,0)+(0*\xoffmed,0)+(-1.5*\xoffsmall+\xoffsmall,0.25*\yoffbig)$);
	 \node at ($(base)+(1.5*\xoffmed,0)+(-1.5*\xoffsmall+\xoffsmall,0.25*\yoffbig+\yoffglabel)$) {$p=5$};

\end{tikzpicture}
\end{minipage}
\end{sideways}
\caption{The \locus{} of $\Infl_{G/N}^G$ for $G=D_{10}$, $N=C_5$.}\label{fig:D10C5}
\end{figure}
\begin{Rem}\label{rem:adams-larger}
	The Adams isomorphism is classically defined for $N$-free $G$-spectra (cf.~Ex.~\ref{exa:N-free}).
	Note that in both cases of Ex.~\ref{exa:locus-for-D_10}, even if we look $p$-locally, the \locus{} is larger than what is given by the $N$-free $G$-spectra.
	For $N=G$, the $G$-free spectra correspond to the irreducible closed subset $\supp(G_+) = \overline{\{\cat P(1,1)\}}$,
	while for $N=C_5$ the $N$-free $G$-spectra correspond to $\supp(G_+) \cup \supp({G/C_2}_+) = \overline{\{\cat P(1,1)\}} \cup \overline{\{\cat P(C_2,1)\}}$.
\end{Rem}

\section{The compactness locus in algebraic geometry}
\label{sec:alg-geom}

Our next goal is to provide a geometric description of the \locus{} for examples arising in algebraic geometry. Namely, if $f:X \to Y$ is a morphism of quasi-compact and quasi-separated schemes, then the \locus{} of the derived pull-back $f^*:\Derqc(Y) \to \Derqc(X)$ is a categorically defined subset of the domain scheme:
\[ \Z_f \subset X \cong \Spc(\Derqc(X)^c).\]
We would like to obtain a scheme-theoretic description of $\Z_f \subset X$ in terms of the morphism $f$.

\begin{Ter}
	Let $X$ be a quasi-compact and quasi-separated scheme.  Recall that the compact objects of $\Derqc(X)$ are the perfect complexes, \ie those complexes which are locally quasi-isomorphic to bounded complexes of finitely generated free modules. This is equivalent to being pseudo-coherent and of finite tor-dimension (cf.~\cite[Sec.~2]{ThomasonTrobaugh90}).  
\end{Ter}
\begin{Exa}\label{Exa:two-aspects-of-perfection}
	These two aspects of perfection (pseudo-coherence and finite tor-dimension)
	are well demonstrated in the affine case:
	An $R$-module~$M$ regarded as an object of $\Der(R)$
	is pseudo-coherent
	iff it has a (possibly infinite) resolution by finitely generated free modules,
	and it has finite tor-dimension iff 
	it has a finite resolution by (possibly infinitely generated) flat modules.
	Together, $M$ is perfect iff it has a finite resolution by finitely generated free modules.
\end{Exa}
\begin{Rem}
	We are interested in understanding when $\bbR f_* :\Derqc(X) \to \Derqc(Y)$ preserves perfect complexes and, as we shall see, it will be helpful to consider separately the question of when 
	it
	preserves pseudo-coherent complexes and when it preserves complexes of finite tor-dimension.
	An informal discussion of some basic examples will help us gain some intuition
	about how these two situations differ 
	and will set the stage for the results that follow.
\end{Rem}
\begin{Exa}\label{Exa:open-immersion-test}
	We have already obtained a description of the \locus{} of an open immersion in
	Proposition~\ref{prop:locus-of-finite-localization}
	(see Example~\ref{exa:open-imm}).
	In particular, an open immersion only satisfies GN-duality (\ie $\bbR f_*$ only preserves perfect complexes) if it is the inclusion of a set which is both open and Thomason closed (\ie a union of connected components in the noetherian case).
Since flat morphisms preserve complexes of finite tor-dimension (see Remark~\ref{rem:preserve-finite-tor-dim} below), this failure of open immersions to preserve perfect complexes is 
	really the failure of open immersions to preserve pseudo-coherent complexes.
	For example, 
	consider
	the inclusion of a principal open in an affine scheme: $D(s) \hookrightarrow \Spec(R)$.
	The (derived) pushforward is just restriction of scalars along $R \to R[1/s]$.
	If $R[1/s]$ is pseudo-coherent as an object of $\Der(R)$ 
	then in particular it is finitely generated as an $R$-module (cf.~Example~\ref{Exa:two-aspects-of-perfection}).
A straightforward consequence is that 
$s^n(1-rs)=0$ for some $r \in R$ and $n \ge 1$.
It follows that we have an equality of principal ideals 
	$(s^n)=(r^n s^n)$ and that $e:=r^n s^n$ is an idempotent.
	Hence the open set $D(s)=D(s^n)=D(e)$ is also a (Thomason) closed subset of $\Spec(R)$.
\end{Exa}
\begin{Exa}\label{Exa:closed-immersions-test}
	In contrast, closed immersions of noetherian schemes always preserve pseudo-coherent complexes. 
	In fact, more generally, proper morphisms of noetherian schemes preserve pseudo-coherent complexes
	(see \eg~\cite[Thm.~III.2.2, pp.~236--38]{SGA6} or \cite[\S 4.3]{Lipman09}).
	The key to understanding this fact is that over a noetherian scheme, a complex
	$\cat E$ is pseudo-coherent iff each cohomology sheaf $H^i(\cat E)$ is coherent and $H^i(\cat E) = 0$ for $i \gg 0$.
	Thus, the fact that proper morphisms preserve pseudo-coherent complexes
	is conceptually just a generalization of the facts (due to Serre and Grothendieck)
	that the cohomology groups $H^i(X,\cat F)$ of a coherent sheaf 
	on a proper algebraic variety are each finite-dimensional and vanish for $i > \dimm X$.

	For example, consider a map $A \to B$ of noetherian rings.
	Over a noetherian ring, every finitely generated module admits a resolution by finitely generated free modules and thus a module is pseudo-coherent iff it is finitely generated.
	(This is, of course, just a special case of the cohomological characterization of pseudo-coherent complexes mentioned above.)
	Thus, $B$ is pseudo-coherent as an $A$-module iff $B$ is finitely generated as an $A$-module.
	Moreover, in this case every pseudo-coherent complex of $B$-modules is pseudo-coherent when regarded as a complex of $A$-modules --- just use the cohomological characterization of pseudo-coherent complexes mentioned above.
	We conclude that a morphism of noetherian affine schemes preserves pseudo-coherent complexes iff it is a finite morphism (\ie a proper morphism).

Anticipating the importance of properness in what follows, let us introduce the following terminology.
\end{Exa}
\begin{Ter}
	Let $f :X \to Y$ be a morphism of finite type between quasi-compact, quasi-separated schemes.
We will say that a closed subset $Z \subseteq X$ is \emph{proper over $Y$}
	or that \emph{$f$ is proper on $Z$}
	if there exists a closed subscheme structure on~$Z$ such that
	$Z \hookrightarrow X \to Y$ is a proper morphism.
	This is equivalent to saying that $Z \hookrightarrow X \to Y$ is a
	proper morphism when~$Z$ is equipped with any closed subscheme structure
	(see \eg~\cite[\href{https://stacks.math.columbia.edu/tag/0CYL}{Lemma 0CYL}]{stacks-project}).
\end{Ter}
\begin{Exa}
Let $j:U \hookrightarrow X$ be an inclusion of a quasi-compact open subscheme.
The \locus{} of $j$ is, according to Example~\ref{exa:open-imm}, the set of points $x \in U$ whose closure in $X$ is contained in~$U$.
This coincides with the union of all closed subsets of $U$ 
on which $j$ is proper.
\end{Exa}
\begin{Exa}
	A closed immersion $Z \hookrightarrow X$ of noetherian schemes,
	being proper, preserves pseudo-coherent complexes (Example~\ref{Exa:closed-immersions-test}), so the question of whether it preserves perfect complexes
	boils down to the question of whether it preserves complexes of finite tor-dimension.
	This is usually not the case.
	For simplicity, consider the affine case $R \to R/I$.
	This will preserve complexes of finite tor-dimension (and consequently perfect complexes)
	iff $R/I$ has finite flat dimension as an $R$-module
	(which, since $R$ is noetherian, is equivalent to $R/I$ having finite projective dimension as an $R$-module).
	In the special case that $I$ is generated by a regular sequence,
	the associated Koszul complex provides a finite resolution of
	$R/I$ by finitely generated free $R$-modules, but in general there
	is no reason for $\pd_R(R/I)$ to be finite.
	For an explicit counterexample,
	take $R=k[X]/(X^2)$ and $I=(X)$ so that $R/I=k$.
	(As~$R$ is a local ring of Krull dimension $0$, the Auslander-Buchsbaum formula
	implies that a finitely generated $R$-module has finite projective dimension
	iff it is free.)
\end{Exa}
\begin{Rem}\label{rem:preserve-finite-tor-dim}
In contrast, if $f:X \to Y$ is a flat morphism then it
is straightforward to see that $\bbR f_*$ preserves complexes of finite tor-dimension
(although we have seen in Example~\ref{Exa:open-immersion-test} that it need not preserve pseudo-coherent complexes).
More generally, Illusie introduces the notion of a perfect morphism of schemes in \cite[Expos\'{e}~III]{SGA6}.
The precise definition is slightly technical but some technicalities 
are removed by restricting attention to finite type morphisms of noetherian schemes.\footnote{Our restriction to morphisms of finite type between noetherian schemes keeps some bugbears under the rug. Hidden there is the fact that a finite type morphism between noetherian schemes is automatically a pseudo-coherent morphism, which is a notion we do not wish to discuss.  Without noetherian assumptions, a morphism of finite type is perfect if $\cat O_X$ is ``perfect relative to $f$'' which means ``has finite tor-dimension relative to $f$'' and ``is pseudo-coherent relative to $f$''. The first condition is the definition above, while the second condition is the definition of $f$ being pseudo-coherent. One reason we have avoided discussing pseudo-coherence of morphisms is that the terminology could be confusing in the present context: the derived pushforward of a pseudo-coherent morphism does not preserve pseudo-coherent complexes; the crucial requirement, as we have discussed and as we shall see below, is that the morphism be \emph{proper}.}
In this case the definition becomes: $f:X\to Y$ is a perfect morphism if $\cat O_X$ 
	is perfect when regarded as a complex of $f^{-1}(\cat O_Y)$
	modules. As perfection can be checked on stalks this is equivalent to 
	saying that for each $x \in X$, the local ring $\cat O_{X,x}$ has finite tor-dimension when regarded as
	a $\cat O_{Y,f(x)}$-module.
Examples of perfect morphisms (of noetherian schemes) include 
regular immersions (see \cite[\S 16.9]{EGA4part4} and \cite[\href{https://stacks.math.columbia.edu/tag/068C}{Tag 068C}]{stacks-project})
and 
flat morphisms of finite type (\eg~open immersions and smooth morphisms).
\end{Rem}
	Let us now clarify that perfection is a local notion.
\begin{Lem}\label{lem:perf-locus}
	Let $f:X\to Y$ be a morphism of finite type between noetherian schemes.
	For a point $x \in X$, the following two conditions are equivalent:
	\begin{enumerate}
		\item There exists an open neighbourhood $x \in U$ such that $U \hookrightarrow X\to Y$ is a perfect morphism.
		\item The local ring $\cat O_{X,x}$ has finite tor-dimension when regarded as an $\cat O_{Y,f(x)}$-module.
	\end{enumerate}
\end{Lem}
\begin{proof}
	First we prove
	(a)$\Rightarrow$(b). 
	Let $j:U\hookrightarrow X$ be an open neighbourhood of $x$
	such that $f\circ j:U\to Y$ is perfect.
	By definition this means that $\cat O_U$ 
	has finite tor-dimension as a complex
	of $(f\circ j)^{-1}\cat O_Y$-modules.
	Taking the stalk at $x$ it follows that
	$\cat O_{X,x} \cong \cat O_{U,x}$ has finite tor-dimension as a complex
	of $((f\circ j)^{-1}\cat O_Y)_x\cong \cat O_{Y,f(x)}$-modules
	(using \eg~the characterization of tor-dimension in terms of the existence of bounded resolutions by flat modules).
	Conversely, let's prove
	(b)$\Rightarrow$(a). 
	Suppose $\cat O_{X,x}$ has finite tor-dimension as an $\cat O_{Y,f(x)}$-module, say
$h := \tordim_{\cat O_{Y,f(x)}} \cat O_{X,x}$.
	Since $f$ is of finite type, there exists an open neighbourhood $j:W\hookrightarrow X$ of $x$
	such that $f\circ j:W\to Y$ factors as
	a closed immersion $i:W\hookrightarrow X'$ followed by a smooth morphism $g:X'\to Y$.
	(Moreover, we may assume without loss of generality that $i:W \hookrightarrow X'$ is a closed immersion of affine schemes.)
	If $g$ is smooth of relative dimension $d$, the argument in the proof of 
	\cite[Prop.~III.3.6(ii)]{SGA6}
	implies that $\tordim_{\cat O_{X',x}} (i_*\cat O_W)_x \le h+d$.
	Now the set of points $x' \in X'$ such that $\tordim_{\cat O_{X',x'}} (i_*\cat O_W)_{x'} \le h+d$
	is an open subset of $X'$ (cf.~\cite[Cor.~9.4.7]{BrodmannSharp13}).
	So
	\[ U:=\SET{w\in W}{\tordim_{\cat O_{X',w}}(i_*\cat O_W)_w \le h+d} \]
	is an open neighbourhood of $x$ in $W$ (and hence in $X$).
	Moreover, for $w \in U$ we have
	\begin{align*}
		\tordim_{\cat O_{Y,f(w)}} \cat O_{U,w} &= \tordim_{\cat O_{Y,f(w)}} \cat O_{W,w} \\
	 &= \tordim_{\cat O_{Y,f(w)}} (i_*\cat O_W)_w \\
	 &\le \tordim_{\cat O_{X',w}}(i_* \cat O_W)_w \\
	 &\le h+d
	\end{align*}
	where the first inequality comes from the fact that smooth morphisms are perfect.
	It follows (cf.~\cite[Prop.~III.3.3]{SGA6}) that $U \hookrightarrow X \to Y$ has finite tor-dimension $\le h+d$.
	Since it is automatically pseudo-coherent, we conclude that it is a perfect morphism.
\end{proof}
\begin{Def}\label{def:perfect-locus}
	Let $f:X \to Y$ be a morphism of finite type between noetherian schemes.
The \emph{perfect locus} of $f$ is the set
	\[ \cat P_f := \SET{x \in X}{\cat O_{X,x}\text{ has finite tor-dimension as an }\cat O_{Y,f(x)}\text{-module}}. \]
	It is an open subset of $X$ and the morphism $f$ is perfect precisely when $\cat P_f = X$ (cf.~Lemma~\ref{lem:perf-locus}).
\end{Def}
\begin{Rem}
	The following theorem, the main result of this section, shows how these two notions --- 
	properness and perfection --- combine to give the compactness locus.
	The remainder of the section will be devoted to its proof (culminating on page~\pageref{proof:main-algebro-thm}).
\end{Rem}
\begin{Thm}\label{thm:main-algebro-thm}
	Let $f:X \to Y$ be a separated, finite type morphism of noetherian schemes.
	The \locus{} of
	$f^* :\Dqc(Y) \to \Dqc(X)$
	is the union of all closed subsets $Z \subseteq X$
	which are proper over Y and which are contained in the perfect locus of $f$.
\end{Thm}
\begin{Cor}
	If $f:X \to Y$ is a proper morphism of noetherian schemes
	then the \locus{} of 
	$f^* :\Dqc(Y) \to \Dqc(X)$
	is the largest specialization closed subset of $X$
	which is contained in the perfect locus of $f$.
	In other words, a point $x \in X$ is contained in the \locus{}
	iff the closure of $x$ is contained in the perfect locus.
\end{Cor}
\begin{Cor}
	If $f:X \to Y$ is a separated, perfect morphism
	of noetherian schemes
	then the \locus{} of
	$f^* :\Dqc(Y) \to \Dqc(X)$
	is the union of all closed subsets of $X$
	which are proper over $Y$.
\end{Cor}
\begin{Rem}
	Another corollary of the theorem is that a separated morphism of finite type
	between noetherian schemes satisfies GN-duality iff it is proper and perfect.
	This characterization was obtained 
	by Lipman and Neeman \cite{LipmanNeeman07} who made a detailed study of 
	those morphisms which satisfy GN-duality.
	In fact, we will utilize several of their techniques in the proof of Theorem~\ref{thm:main-algebro-thm}.
\end{Rem}
\begin{Ter}
	A functor $F:\Der(\cat A)\to \Der(\cat B)$ between two derived categories (or two subcategories thereof) is said to be \emph{bounded above} if there exists an integer $d$ such that for any object $\E$ and integer $m$,  $H^i(\E)=0$ for $i>m$ implies $H^i(F(\E))=0$ for $i>m+d$.  Similarly, the functor is \emph{bounded below} if there exists an integer $d$ such that $H^i(\E)=0$ for $i<m$ implies that $H^i(F(\E))=0$ for $i<m-d$.  We say the functor is \emph{bounded} if it is both bounded above and bounded below.
\end{Ter}

\begin{Prop}\label{prop:bounded}
	Let $f:X\to Y$ be a morphism of quasi-compact, quasi-separated schemes.
	If $\E \in \A_f$ then $\ihom(\E,f^!(-)):\Derqc(Y)\to\Derqc(X)$ is bounded.
\end{Prop}
\begin{proof}
	The functor $f^!$ is bounded below (cf.~\cite[Thm~4.1]{Lipman09}) and it is straightforward to see
	that $\ihom(\E,-)\cong \dual \E \otimes -$ is bounded for any perfect $\E$ (since the 
	perfect complex
	$\dual \E$ is globally quasi-isomorphic to a bounded complex of flat modules).
	Thus, the functor $\ihom(\E,f^!(-))$ is bounded below for any perfect $\E$; the key is to
	prove that it is bounded above when $\E \in \A_f$.
	To prove this, we will adapt the methods of \cite[Sect.~3--4]{LipmanNeeman07}.
	By \cite[Thm.~3.1.1]{BondalVAndenbergh03}, $\Derqc(X)$ is generated by a single perfect
	complex $\perfgen \in \Derqc(X)$; that is, there exists a perfect complex $\perfgen$ such that 
	if $\G \in \Derqc(X)$ then
	\[ \G \neq 0 \text{ in } \Derqc(X) \Longrightarrow \Hom(\perfgen[n],\G) \neq 0 \text{ for some } n \in \mathbb{Z}.\]
	More precisely, by \cite[Thm.~4.2]{LipmanNeeman07}, if $\perfgen$ is such a perfect generator, there
	exists an integer $A=A(\perfgen)$ such that for any $\G \in \Derqc(X)$ and $j \in \mathbb{Z}$,
	\begin{equation}\label{eq:bddedabove-ddag}
		H^j(\G) \neq 0 \Longrightarrow \Hom(\perfgen[n],\G) \neq 0 \text{ for some } n \le A-j.
	\end{equation}
	On the other hand, $\E \in \A_f$ implies that $f_*(\E\otimes \perfgen)$ is perfect, and hence
	is $a$-locally projective for some $a \in \mathbb{Z}$ (cf.~\cite[p.~218]{LipmanNeeman07}).
	Hence, by \cite[Lem.~3.2]{LipmanNeeman07}, there exists $s=s(Y)>0$ such that
	\[ \Hom(f_*(\E\otimes \perfgen),\F[-n]) = 0\]
	for any $\F \in \Derqc(Y)$ for which $H^j(\F) = 0$ for $j>a-s-n$.
	Then let's prove that $\ihom(\E,f^!(-))$ is bounded above.
	To this end,
	consider $\F\in \Derqc(Y)$ and suppose that $H^j(\F) = 0$ for $j>m$.
	If $j \ge m+A+s-a$ then for any $n \le A-j$ we have that $a-s-n\ge a-s-A+j \ge m$
	and hence
	\[ \Hom(\perfgen[n],\ihom(\E,f^!\F))\cong \Hom(f_*(\E\otimes \perfgen),\F[-n])= 0.\]
	Hence by~\eqref{eq:bddedabove-ddag}, we conclude that $H^j(\ihom(\E,f^!\F))=0$.
	This establishes that $\ihom(\E,f^!(-))$ is bounded above.
\end{proof}
\begin{Lem}\label{lem:restrict-to-open-in-base}
	Let $f:X \to Y$ be a morphism of quasi-compact, quasi-separated schemes.
	Let $V \subset Y$ be a quasi-compact open subset, and set
	$f_V := f|_{f^{-1}(V)} : f^{-1}(V) \to V$.
	If $\E \in \A_f$ then $\E|_{f^{-1}(V)} \in \A_{f_V}$.
\end{Lem}
\begin{proof}
	Let $\E \in \A_f$ and consider $\E|_{f^{-1}(V)} \in \Derqc(f^{-1}(V))^c$.
	We claim that 
	\[\RRb(f_V)_*(\E|_{f^{-1}(V)}\otimes \F)\]
	is compact in $\Derqc(V)$
	for any $\F \in \Derqc(f^{-1}(V))^c$.
	By the Thomason-Neeman localization theorem (cf.~\cite[Cor.~4.5.14, Rem.~4.5.15]{Neeman01}),
	$\F \oplus \Sigma \F \simeq
	\G|_{f^{-1}(V)}$ for some $\G \in \Derqc(X)^c$, and it suffices to prove that
	\[ \RRb(f_V)_*(\E|_{f^{-1}(V)}\otimes \G|_{f^{-1}(V)}) \simeq \RRb(f_V)_*((\E\otimes \F)_{f^{-1}(V)})\]
	is compact. 
	By flat base change
	we have
	\[ \RRb(f_V)_*\left((\E\otimes \G)|_{f^{-1}(V)}\right)
		\simeq \left( \RRb f_*(\E\otimes \G)\right)|_V\]
and the latter is compact since restriction to an open preserves compacts and $\E \in \A_f$ by hypothesis.
\end{proof}
\begin{Prop}\label{prop:finite-tor-dimension}
	Let $f:X \to Y$ be a morphism of quasi-compact, quasi-separated schemes.
	If $\E \in \A_f$ then $\E$ has finite tor-dimension as a complex of $f^{-1}\cat O_Y$-modules.
\end{Prop}
\begin{proof}
	Let $Y=\bigcup V_i$ be an open affine cover and set $f_i : f^{-1}(V_i)\to V_i$.
	By Lemma~\ref{lem:restrict-to-open-in-base}, 
	if $\E\in \A_f$ then
	$\E|_{f^{-1}(V_i)}\in \A_{f_i}$.
	On the other hand, if $\E|_{f^{-1}(V_i)}$ has finite tor-dimension as a complex of
	$f_i^{-1}\cat O_{V_i}\simeq f^{-1}\cat O_Y|_{f^{-1}(V_i)}$-modules
	for all~$i$, then~$\E$ has finite tor-dimension as a complex of $f^{-1}\cat O_Y$-modules.
	Thus, the problem is local in the base, and we can assume without loss of generality that $Y=\Spec(A)$ is affine.

	Now to show 
	that $\E$ has finite tor-dimension as complex of
	$f^{-1}(\cat O_Y)$-modules, it suffices to check that $\E|_U$ has finite tor-dimension
	as a complex of $(f^{-1}\cat O_Y)|_U \simeq (f\circ j)^{-1}(\cat O_Y)$-modules
	for any
	open affine $j:U\hookrightarrow X$, say $U=\Spec(B)$.
	Moreover, since $f\circ j$ is a morphism of affine schemes, we just need to prove that $(f\circ j)_*(\E|_U)$ has finite tor-dimension as a complex of $\cat O_Y$-modules (\ie~the complex of $B$-modules $\E|_U$ has finite tor-dimension when regarded as a complex of $A$-modules).

	Then 
	observe that for any $\G \in \Qcoh(Y)$ and $j \in \mathbb Z$ we have
	\begin{align*} \Ext_{\cat O_Y}^j( (f\circ j)_*(\E|_U),\G) &=
		\Hom_{\Der(Y)}( (f\circ j)_*(\E|_U),\G[j])\\
		&\cong \Hom_{\Der(Y)}(\RRb(f\circ j)_*(\E|_U),\G[j])\\
		&\cong \Hom_{\Der(Y)}(\RRb f_* \RRb j_*(\E|_U),\G[j])\\
		&\cong \Hom_{\Der(X)}(\RRb j_*(\E|_U),f^!\G[j])\\
		&\cong \Hom_{\Der(X)}(\E\otimes \RRb j_*\cat O_U,f^!\G[j])\\
		&\cong \Hom_{\Der(X)}(\RRb j_*\cat O_U,\ihom(\E,f^!\G)[j]).
	\end{align*}
	Applying \cite[Lem.~3.3]{LipmanNeeman07} to $j:U\hookrightarrow X$, there
	exists $t=t_U >0$ such that
	for any $a$-locally projective $\F \in \Derqc(U)$,
	$\RRb j_*(\F) \in \Derqc(X)$ is $(a-t)$-locally projective.
	Now, $\cat O_U$ is $0$-locally projective, so
	then $\RRb j_*(\cat O_U)$ is $-t$-locally projective.
	Then by \cite[Lem.~3.2]{LipmanNeeman07}, there exists $s=s(X)>0$
	such that 
	\[ \Hom_{\Der(X)}(\RRb j_* \cat O_U,\ihom(\E,f^!(\G))[j])=0\]
	for any $\G\in \Derqc(Y)$ with the property that $H^i(\ihom(\E,f^!\G)[j])=0$ for all $i> -t-s$.
	On the other hand, by Proposition~\ref{prop:bounded}, we 
	know
	that $\ihom(\E,f^!(-))$ is bounded since $\cat E \in \A_f$ by assumption.
	So there exists $m=m_\E \in \mathbb Z$ such that
	if $\G \in \Derqc(Y)$ has the property that $H^i(\G) = 0$ for $i>n$,
	then 
$H^i(\ihom(\E,f^!\G))=0$ for all $i > n +m$.
	Thus, setting $j_0 := m+t+s$ we conclude that
	\[ \Ext_{\cat O_Y}^j((f\circ j)_*(\E|_U),\G)\cong\Hom_{\Der(X)}(\RRb i_*\cat O_U,\ihom(\E,f^!\G)[j])=0 \]
	for all $j>j_0$ and $\G \in \Qcoh(Y)$.
	It follows that the bounded complex $(f\circ j)_*(\E|_U)$ 
	has finite projective dimension, hence has finite flat dimension (\ie is
	isomorphic to a bounded complex of flat modules) which completes the proof.
\end{proof}
\begin{Lem}
	Let $R$ be a commutative noetherian local ring and let $M$ and $N$ be complexes of $R$-modules.
	Suppose that $M$ is bounded below and homologically finite, and that $N$ is perfect.
	If $N \neq 0$ in $\Der(R)$ and $M \otimes^{\mathbb L}_R N$ is perfect then $M$ is perfect.
\end{Lem}
\begin{proof}
	Over a noetherian ring, a bounded below, homologically finite complex is perfect
	if and only if it has finite projective dimension if and only if it has finite flat dimension,
	in which case the projective and flat dimensions coincide (cf.~\cite[Cor.~2.10.F]{AvramovFoxby91}).
	Moreover, when $R$ is local, its flat dimension can be computed as
	\begin{equation}\label{eq:local-perfa}
		\fd_R(M) = \sup\SET{i \in \mathbb Z}{H_i(k\otimes_R^{\mathbb L} M) \neq 0}.
	\end{equation}
	This is mentioned, \eg, in \cite{DwyerGreenleesIyengar06}, and can be
	obtained by computing the derived tensor-product using a minimal free resolution of $M$.
	Now, if $M \otimes_R^{\mathbb L} N$ is perfect then evidently
	$\fd_R(M\otimes_R^{\mathbb L}N) < \infty$, which implies 
	that
	\begin{equation}\label{eq:local-perfb}
		H_i(k\otimes_R^{\mathbb L}(M\otimes_R^{\mathbb L} N)) = 0 \qquad\text{ for } i \gg0.
	\end{equation}
	Next note that since $R$ is local, a pefect complex $N$ is zero in $\Der(R)$
	if and only if its homological support $\supp_R(N) \neq \emptyset$
	if and only if $\supp_R(N)$ contains the closed point
	if and only if $\supp_k(k \otimes_R^{\mathbb L} N) \neq \emptyset$
	if and only if $k \otimes_R^{\mathbb L} N \neq 0$ in $\Der(k)$.
	So $H_*(k\otimes_R^{\mathbb L} N)\neq 0$.
	Hence, from
	\begin{align*}
		H_*(k\otimes_R^{\mathbb L}(M\otimes_R^{\mathbb L}N)) &\cong H_*((k\otimes_R^{\mathbb L}M)\otimes_k(k\otimes_R^{\mathbb L}N))\\
		&\cong H_*(k\otimes_R^{\mathbb L} M)\otimes_k H_*(k\otimes_R^{\mathbb L}N)
	\end{align*}
	we see that \eqref{eq:local-perfb} implies $H_*(k\otimes_R^{\mathbb L} M)$ must also be
	bounded above 
	and thus $M$ has finite projective dimension by \eqref{eq:local-perfa}.
\end{proof}
\begin{Cor}\label{cor:local-koszul}
	Let $R$ be a commutative noetherian local ring and $I \subset R$ an ideal.
	Let $K(J)$ denote the Koszul complex associated to a nonzero ideal $J \subset R$.
	If $R/I \otimes_R^{\mathbb L} K(J)$ is compact in $\Der(R)$ then
	$R/I$ is compact in $\Der(R)$.
\end{Cor}
\begin{Prop}\label{prop:closed-immersion}
	Let $i:Z \hookrightarrow X$ be a closed immersion of a noetherian scheme~$X$.
	Then the \locus{} $\Z_i$ is contained in the perfect locus $\cat P_i$.
\end{Prop}
\begin{proof}
	It follows from Lemma~\ref{lem:restrict-to-open-in-base} that the problem is
	local in the base, so it suffices to prove the claim for a closed immersion of affine schemes.
	So let $R$ be a commutative noetherian ring, $I\subset R$ an ideal
	and set $f:=\Spec(R/I)\hookrightarrow \Spec(R)$.
	Let~$J$ be any ideal of $R$ containing $I$
	and let $\overline{J}$ denote the corresponding ideal of~$R/I$.
	Now $f^*(K_R(J))\cong K_{R/I}(\overline{J})$, and hence
	\[ f_*(K_{R/I}(\overline{J}))\cong R/I \otimes_R^{\mathbb L} K_R(J)\]
	by the projection formula.
	Since $\supp(K_{R/J}(\overline{J}))=V(\overline{J})$,
	we have that $V(\overline{J}) \subset \Z_f$ if and only if
	$f_*(K_{R/I}(\overline{J}))$ is compact in $\Der(R)$ if and only if $R/I \otimes_R^{\mathbb L} K_R(J)$ is compact in $\Der(R)$.
	This is the case if and only if 
	\[ ( (R/I) \otimes_R^{\mathbb L} K_R(J))_{\mathfrak q}
		\simeq R_{\mathfrak q}/I_{\mathfrak q} \otimes_{R_{\mathfrak q}}^{\mathbb L} K_{R_{\mathfrak q}}(J_{\mathfrak q})\]
	is compact in $\Der(R_{\mathfrak q})$ for all $\mathfrak q \in V(J)$.
	Invoking Corollary~\ref{cor:local-koszul}, we conclude that if
	$V(J) \subset \Z_f \subset V(I)$ then 
	$R_{\mathfrak q}/I_{\mathfrak q}$ is compact in $\Der(R_{\mathfrak q})$ for all ${\mathfrak q}\in V(J)$,
	which means by definition that $V(J) \subset \cat P_f$.
	Since $\Z_f$ is closed under specialization, it follows that $\Z_f \subset \cat P_f$.
\end{proof}

\begin{Prop}\label{prop:finite-type}
	Let $f:X \to Y$ be a morphism of finite type between noetherian schemes.
	Then the \locus{} of $f^*:\Derqc(Y)\to\Derqc(X)$ is contained in the perfect
	locus of $f$.
\end{Prop}
\begin{proof}
	If $x \in \cat Z_f$ is a point in the \locus{}, then there exists a complex ${\E \in \A_f}$ with $\supp(\E)=\overline{\{x\}}$.
	By Proposition~\ref{prop:finite-tor-dimension}, $\E$ has finite tor-dimension 
	as a complex of $f^{-1}\cat O_Y$-modules.
	Since $f$ is of finite type, there is an open affine neighbourhood $j:U\hookrightarrow X$ of $x$
	such that $f\circ j$ factors as a closed immersion 
	${i:U\hookrightarrow X'}$ followed by a smooth morphism $g:X'\to Y$.
	Now, $j^*(\E) \in \Der(U)^c$ has finite tor-dimension relative to $f \circ j$,
	hence \cite[p.~246, Prop.~3.6]{SGA6} implies that $\RRb i_*(j^*\E)$ is perfect.
Moreover, since $\Derqc(U)$ is generated by $\unit$ (since $U$ is affine),
	$\mathbb{R}i_*(j^*\E)$ perfect implies $j^*\E \in \A_i$.
	Furthermore, since
	$\supp_U(j^*\E) = {U \cap \supp(\E) \ni x}$,
we conclude that $x \in \cat Z_i$.
	Hence by Proposition~\ref{prop:closed-immersion}, $x$ is contained in the perfect locus $\cat P_i$ of the closed immersion $i$.
	Since $g$ is smooth (hence perfect), it follows that $x \in \cat P_{g\circ i}$.
	Thus $x \in \cat P_{g \circ i}=\cat P_{f \circ j} = U \cap \cat P_f$.
\end{proof}
\begin{Prop}\label{prop:must-be-proper}
Let $f:X\to Y$ be a separated morphism of finite type between noetherian schemes
and let $Z \subseteq X$ be a closed subset.
If $Z$ is contained in the \locus{} of $f^*:\Derqc(Y)\to\Derqc(X)$ then $f$ is proper on $Z$.
\end{Prop}
\begin{proof} 
Let $i:Z\hookrightarrow X$ denote the closed immersion of $Z$ equipped with its reduced scheme structure.
Our assumption that $Z$ is contained in the \locus{} of $f^*$ ensures that 
$\RRb (f\circ i)_*: \Dqc(Z) \to \Dqc(Y)$ preserves perfect complexes.
On the other hand, 
it will be convenient to say that a morphism is ``quasi-proper'' if its derived pushforward preserves pseudo-coherent complexes.
Approximation of pseudo-coherent complexes by perfect complexes implies that a morphism is quasi-proper 
if and only if it sends perfect complexes to pseudo-coherent complexes 
(see \cite[\href{https://stacks.math.columbia.edu/tag/08EL}{Section 08EL}]{stacks-project} and \cite[Corollary~4.3.2]{LipmanNeeman07}).
Since perfect complexes are pseudo-coherent, it follows that $f\circ i$ is quasi-proper.
Our goal is to prove that $f\circ i$ is proper.
Since properness and quasi-properness are both local in
the target (for the latter just use flat base change as in
the proof of Lemma~\ref{lem:restrict-to-open-in-base})
we can assume that the target 
$Y$ is affine.
By Nagata compactification (\eg~\cite[Theorem~4.1]{Conrad07}) we can write $f\circ i$ as a composite
$Z \hookrightarrow W \to Y$
where $j:Z \hookrightarrow W$ 
is an open immersion and $g:W \to Y$ is proper.
We next prove that the open immersion $j$ is in fact quasi-proper.

To this end, let $\E \in \Dqc(Z)$ be pseudo-coherent and consider 
$j_*\E \in \Dqc(W)$.
An application 
of the derived Chow lemma 
(see \cite[\href{https://stacks.math.columbia.edu/tag/0CSL}{Lemma 0CSL}]{stacks-project}
and \cite[Corollaire III.1.12]{SGA6})
is 
that $j_*\E$ is pseudo-coherent if 
$\RRb g_*( j_*\E \otimes_{W}^{\mathbb{L}} \F)$ is pseudo-coherent for every 
pseudo-coherent $\F \in \Dqc(W)$.
This is indeed the case since
\[
\RRb g_*(j_*\E \otimes_{W}^{\mathbb{L}} \F)
\cong \RRb g_*(j_*(\E \otimes_Z^{\mathbb{L}} j^* \F)) \cong
\RRb (f\circ i)_* (\E \otimes_Z^{\mathbb{L}} j^* \F)\]
and $f\circ i=g\circ j$ is quasi-proper.
We conclude that the open immersion $j:Z \hookrightarrow W$ is quasi-proper.
But it is easy to see that a quasi-proper open immersion is in fact proper.
Indeed, we can cover $W$ by open affines and reduce to the case 
of a quasi-proper open immersion of affine schemes $\Spec(B) \to \Spec(A)$. 
This must be a finite (hence proper) morphism of affine schemes since quasi-properness implies that $B$ is finitely generated as an $A$-module
(cf.~Example~\ref{Exa:closed-immersions-test}).
We conclude that the open immersion $j:Z\hookrightarrow W$ is proper, and hence so is 
$g\circ j=f\circ i$.
\end{proof}
\begin{Prop}\label{prop:proper-preserves-enhanced}
Let $f:X\to Y$ be a finite type morphism of noetherian schemes.
Let $\E \in \Derqc(X)^c$ be a compact object 
and assume that $f$ is proper over $\supp(\E)$ and that $\E$ has finite tor-dimension as a complex of $f^{-1}\cat O_Y$-modules.
Then $\RRb f_*(\E)$ is a compact object in $\Derqc(Y)$.
\end{Prop}
\begin{proof}
	The argument is essentially contained in 
	\cite[\href{https://stacks.math.columbia.edu/tag/08EV}{Tag 08EV}]{stacks-project}.
We sketch the proof for convenience.
Each cohomology sheaf $H^q(\E)$ is a coherent $\cat O_X$-module whose support is proper over $Y$.
Hence by \cite[III.\,Cor.\,3.2.4]{EGA3}, the $\cat O_Y$-modules $R^pf_* H^q(\E)$ are also coherent.
It then follows that the cohomology sheaves $H^n(\RRb f_*\E)$ are also coherent by considering
the spectral sequence
$R^p f_* H^q(\E) \Rightarrow H^{p+q}(\RRb f_*\E)$.
Moreover, the $H^n(\RRb f_* \E)$ vanish for $n \gg 0$ (by \eg~\cite[\href{https://stacks.math.columbia.edu/tag/01XJ}{Tag 01XJ}]{stacks-project}).
This establishes that $\RRb f_* \E$ is pseudo-coherent  and
it remains to show that it has finite tor-dimension as a complex of $\cat O_Y$-modules.
Since $Y$ is quasi-separated, an object $\F \in \Derqc(Y)$ has tor-amplitude in $[a,b]$ if and only if
	$H^i(\F\otimes_{\cat O_Y}^{\mathbb L} \G)=0$ 
	for all $i \not\in [a,b]$ and $\G \in \Qcoh(Y)$.
	(The point being that we need only check this condition for quasi-coherent sheaves $\G$ rather than 
	for all $\cat O_Y$-modules; see \eg~\cite[\href{https://stacks.math.columbia.edu/tag/08EA}{Tag 08EA}]{stacks-project}.)	
	Then observe that for any $\G \in \Qcoh(Y)$, we have
	\begin{align*}
		\RRb f_*\E \otimes_{\cat O_Y}^{\mathbb L} \G &\simeq \RRb f_*(\E\otimes_{\cat O_X}^{\mathbb L} \mathbb{L}f^* \G)
		\simeq \RRb f_*(\E \otimes_{f^{-1}\cat O_Y}^{\mathbb L} f^{-1} \G).
	\end{align*}
	By assumption, $\E$ has finite tor-dimension as an object of $\Der(f^{-1}\cat O_Y)$.
	Hence the complex $\E \otimes_{f^{-1}\cat O_Y}^{\mathbb L} f^{-1} \G$ has cohomology sheaves
	in a given finite range, say $[a,b]$.
	Then its image under $\RRb f_*$ has cohomology in the range $[a,b+d]$ for some integer~$d$ depending on $Y$ but not on $\G$ (cf.~\cite[Prop.~3.9.2]{Lipman09}).
	So there is a universal bound for the cohomology
	\[ H^i(\RRb f_*\E \otimes^{\mathbb{L}}_{\cat O_Y} \G) \]
	for all quasi-coherent sheaves $\G\in\Qcoh(Y)$, independent of $\G$,
	and we conclude that $\RRb f_*(\E)$ has finite tor-dimension.
\end{proof}
\begin{proof}[Proof of Theorem~\ref{thm:main-algebro-thm}]
\label{proof:main-algebro-thm}
Since the \locus{} $\Z_f \subseteq X$ is specialization closed, it follows
from Proposition~\ref{prop:finite-type} and Proposition~\ref{prop:must-be-proper}
that
it is contained in the union of all closed subsets $V \subseteq X$ which are both proper over $Y$
and contained in the perfect locus $\cat P_f$.
For the converse, let $V \subseteq X$ be a closed subset over which $f$ is proper
and which is contained in the perfect locus.
To prove that $V \subseteq \Z_f$, consider
any perfect complex $\E \in \Derqc(X)^c$ with $\supp(\E) \subseteq V$.
	The assumption that $V \subseteq \cat P_f$ implies that $\E$ has finite tor-dimension as a complex of $f^{-1}\cat O_Y$-modules.
	Indeed, it suffices to check that the complex of $\cat O_{X,x}$-modules $\E_x$ has
	finite tor-dimension as an $(f^{-1}\cat O_Y)_x \cong \cat O_{Y,f(x)}$-module for each $x \in X$.
	For $x \in \supp(\E)$, this follows from the fact that $\cat O_{X,x}$ itself has finite tor-dimension as an $\cat O_{Y,f(x)}$-module,
	since $x \in \supp(\E)\subseteq V\subseteq\cat P_f$ by assumption, while $\E_x$ is zero if $x \not\in\supp(\E)$.
	Invoking Proposition~\ref{prop:proper-preserves-enhanced}, we conclude that $\RRb f_*(\E)$ is
	compact in $\Derqc(Y)$.
	This establishes that $V \subseteq \Z_f$ and the proof is complete.
\end{proof}
\bigbreak
\section{Further examples and discussion}
\label{sec:further-examples}
\label{sec:further-directions}

We conclude with 
some
additional examples
and a discussion of future research directions.

\begin{Exa}[Eilenberg-MacLane spectra] \label{exa:eilenberg-maclane}
	Consider the change-of-rings adjunction
	$f^* : \SH \adjto \Ho(H\mathbb Z\text{-Mod}):f_*$
	associated to the map of ring spectra $\mathbb S \to H \mathbb Z$.
	Under the equivalence $\Ho(H\mathbb Z\text{-Mod})\cong \Der(\mathbb Z)$
	the right adjoint $f_*:\Der(\mathbb Z) \to \SH$ sends an abelian group $A$ to its associated
	Eilenberg-MacLane spectrum $HA$.
	Since~$\mathbb Z$ is hereditary, every object $X \in \Der(\mathbb Z)^c$ splits as 
	a direct sum of shifts of its (finitely generated) homology groups.
	Since $HA$ is not compact for any nonzero finitely generated abelian group $A$ (by standard facts about stable cohomology operations),
	we conclude that if $f_*(X) \in \SH^c$ then 
	$X=0$ in $\Der(\mathbb Z)$.
	Thus, the \locus{} of the extension-of-scalars functor $f^*:\SH \to \Der(\mathbb Z)$ is empty.
\end{Exa}

\begin{Exa}[Geometric fixed points]
	Let $G$ be a finite group and let $N \lenormal G$ be a normal subgroup.
	We have discussed at length the 
	inflation functor in equivariant stable homotopy theory (see Section~\ref{sec:adams-locus}).
	However, there is another prominent tensor-triangulated functor in the theory: the geometric
	$N$-fixed point functor $\tilde{\Phi}^{N,G}:\SH(G) \to \SH(G/N)$.
	It is an example of a finite localization, namely finite localization with respect to the 
	family $\cat F[\notsupseteq N]:=\SET{H\le G}{H \not\supseteq N}$ (cf.~Example~\ref{exa:N-concentrated}).
	Thinking geometrically, this is the finite localization associated to the Thomason (closed) set
	\[ Y:=\bigcup_{H \not\supseteq N} \supp(G/H_+) = \SET{\cat P(H,p,n)}{1\le n \le \infty, \text{\,all } p, H \not\supseteq N} \subseteq \Spc(\SHGc)\]
	and we have an identification 
	\[ \Spc(\SH(G/N)^c \cong V 
		:= \SET{\cat P(H,p,n)}{1 \le n \le \infty, \text{\,all } p, H \supseteq N} \subseteq\Spc(\SHGc).\]
	Applying Proposition~\ref{prop:locus-of-finite-localization}, we can describe the \locus{} as follows:
	\[ \Z_{\tilde{\Phi}^{N,G}} = \big(\bigcup_{\substack{(H,p):\Op(H)\supseteq N}} \SET{\cat P(H,p,n)}{2 \le n \le \infty} \big) \cup \big(\bigcup_{H: \Op(H) \supseteq N\;\forall p} \big\{ \cat P(H,1) \big\}\big).\]
	To facilitate easy comparison with Theorem~\ref{thm:locus-of-inflation}, we can rewrite this as follows:
	For any $H \le G$, we have
	\begin{enumerate}
		\item If $\Op(H) \not\supseteq N$, then $\cat P(H,p,n) \not\in \Z_{\tilde{\Phi}^{N,G}}$ for all $2 \le n \le \infty$.
		\item If $\Op(H) \supseteq N$, then $\cat P(H,p,n) \in \Z_{\tilde{\Phi}^{N,G}}$ for all $2 \le n \le \infty$.
		\item $\cat P(H,1) \in \Z_{\tilde{\Phi}^{N,G}}$ if and only if $\Op(H) \supseteq N$ for all $p$.
	\end{enumerate}
	In the terminology of Remark~\ref{rem:alternative-statement}, this says that a compact object $X \in \SH(G/N)^c$ is relatively compact for $\tilde{\Phi}^{N,G}$ iff each $p$-local isotropy group $H/N$ satisfies ${\Op(H) \supseteq N}$.
	In fact, we can also describe the \locus{} of the absolute geometric \mbox{$H$-fixed} point functor
	$\Phi^{H,G}:\SH(G) \to \SH$ for any subgroup $H\le G$.
	Regarding $\Phi^{H,G} \cong \tilde{\Phi}^{H,H}\circ \Res_H^G$ as 
	the composite of $f^*:=\tilde{\Phi}^{H,H}$ and $g^*:=\Res_H^G$, 
	it is immediate that $\A_{g\circ f} \supseteq \A_f$ since $g^*$ satisfies GN-duality.
	Moreover, $g_* = \CoInd_H^G$ has the special property that it reflects compact objects.
	Indeed, $g^*$ preserves compactness and any $x \in \SH(H)$ is a direct summand of $g^*g_*(x)$
	(\eg~by \cite[Lem.~3.3]{BalmerDellAmbrogioSanders15}).
	Thus, $\A_f = \A_{g\circ f}$.
	Therefore, the \locus{} of $\Phi^{H,G} : \SHG \to \SH$ is the subset of $\Spc(\SHc)$ described as follows:
	\begin{enumerate}
		\item If $H$ is not $p$-perfect then $\cat C_{p,n}$ is not contained in $\Z_{\Phi^{H,G}}$ for any $2 \le n \le \infty$.
		\item If $H$ is $p$-perfect then $\cat C_{p,n}$ is contained in $\Z_{\Phi^{H,G}}$ for all $2 \le n\le \infty$.
		\item The generic point $\cat C_{p,1}$ ($=\cat C_{q,1}$ for all $q$) is contained in $\Z_{\Phi^{H,G}}$ if and only if~$H$ is $p$-perfect for all $p$.
	\end{enumerate}
	This can be rephrased as follows.
	If $H$ is perfect then all finite spectra are relatively compact for $\Phi^{H,G}$
	and so GN-duality holds.
	Otherwise, the relatively compact objects are the finite torsion spectra of exponent coprime
	to the order of the abelianization of $H$.
	(Note that the primes $p$ dividing the order of the abelianization of $H$ are precisely those primes $p$ for which $H$ is not $p$-perfect.)
\end{Exa}

\begin{Exa}\label{exa:incomplete}
	Let $G$ be a compact Lie group. For any $G$-universe $\U$ we can consider the 
	stable homotopy category $\SH_\U(G)$ of $G$-spectra indexed on $\U$.
	It is always compactly generated by the orbits $\Sigma_\U^\infty G/H_+$ but 
	these generators need not be rigid if the universe $\U$ is not complete.
	Indeed, by \cite[Prop.~7.1]{Lewis00}, $\Sigma_\U^\infty G/H_+$ is rigid in $\SH_\U(G)$
	if and only if $G/H$ embeds in $\U$ as a $G$-space.
	Then, 
	defining $\Iso(\U) := \SET{H \le G}{G/H \hookrightarrow \U}$, we can 
	consider
	the subcategory 
	\[ \Loc\langle \Sigma^\infty_\U G/H_+ \mid H \in \Iso(\U) \rangle \subset \SH_\U(G) \]
	generated by those orbits which are rigid.
	It is a rigidly-compactly generated tensor-triangulated subcategory of $\SH_\U(G)$ by \cite[p.~87]{HoveyPalmieriStrickland97}.
	Now, for any $G$-linear isometry $i:\V\hookrightarrow \U$, the
	change-of-universe functor
	$i_*:\SH_\V(G)\to\SH_\U(G)$
	induces a functor
	\begin{equation}\label{eq:induced-rigid}
		\Loc\langle \Sigma_\V^\infty G/H_+ \mid H \in \Iso(\V) \rangle \xra{i_*} \Loc\langle \Sigma^\infty_\U G/H_+ \mid H\in \Iso(\U) \rangle
	\end{equation}
	which is a geometric functor between rigidly-compactly generated categories.
	In fact, the inflation functor can be realized as a special case of this construction.
	Indeed, consider the inclusion $i:\U^N \hookrightarrow \U$ of the $N$-fixed points 
	in a
	complete $G$-universe $\U$.
	In this case, 
	\eqref{eq:induced-rigid} takes the form
	\begin{equation}\label{eq:induced-change}
		\Loc\langle \Sigma^\infty_{\U^N}G/H_+ \mid H \in \Iso(\U^N)\rangle \hookrightarrow \SH_{\U^N}(G) \xra{i_*} \SH(\U).
	\end{equation}
	But the change-of-groups functor $\eps^*:\SH_{\U^N}(G/N) \to \SH_{\U^N}(G)$ is fully faithful with essential image 
$\Loc\langle G/H_+ \mid H \supseteq N\rangle = \Loc\langle G/H_+ \mid H\in\Iso(\U^N)\rangle.$
Thus, under this equivalence 
$\SH_{\U^N}(G/N)\cong \Loc\langle G/H_+ \mid H\in \Iso(\U^N)\rangle$,
	the functor \eqref{eq:induced-change} is nothing but the 
	inflation functor $\Infl_{G/N}^G := i_*\circ \eps^*$.
The author has not pursued this
but one could attempt an analysis of the adjoints and \locus{} of the more general geometric functor \eqref{eq:induced-rigid} associated to any $G$-linear isometry $\V \hookrightarrow \U$.
\end{Exa}

\begin{Rem}\label{rem:right-adjoint-to-fixed-points}
	As discussed in Remark~\ref{rem:almost-dualizing}, we know that the relative dualizing object $\DO:=f^!(\unit)$ of the inflation functor $f^*:=\Infl_{G/N}^G$ coincides with $S^{-\Adhspace\Ad(N;G)}$ \emph{after} performing the colocalization $\Gamma^*:= E\cat F(N)_+ \smashh -$.
	However, it would be desirable to obtain an explicit description of $\DO\in\SH(G)$ in general, \ie~before performing the colocalization.
	(In particular, such a description may shed light upon the relationship between the Adams isomorphism and the tom Dieck splitting theorem.)
	More generally, the right adjoint $f^!$ of the categorical fixed point functor $f_*=\lambda^N$
	deserves further study.
	By Theorem~\ref{thm:main-thm}, we know that
	$\Gamma^* f^! \cong {\Gamma^*(\DO \otimes f^*)}$,
	but we also know (by Theorem~\ref{thm:original-thm} and Proposition~\ref{prop:infl-doesnt-satisfy-gn-duality})
	that $f^! \not\cong \DO \otimes f^*$ before colocalization,
	so the problem of understanding $f^!$ does not simply reduce to understanding~$\DO$.
\end{Rem}

\begin{Rem}
	In this work, we have focused on Grothendieck-Neeman duality and the construction of the Wirthm\"{u}ller isomorphism ---
	\ie~on the original Theorem~\ref{thm:original-thm} from \cite{BalmerDellAmbrogioSanders16}.
	However, in that work there is one more theorem which completes the picture.
	Indeed, \cite[Thm.~C]{BalmerDellAmbrogioSanders16}
	establishes that, for a functor $f^*$ satisfying GN-duality and hence having the five adjoints of Theorem~\ref{thm:original-thm},
	the existence of one more adjoint on the left is equivalent to the existence of one more adjoint on the right,
	and that this is equivalent to there being an infinite sequence of adjoints in both directions.
	This leads to the Trichotomy Theorem \cite[Cor.~1.13]{BalmerDellAmbrogioSanders16} which states that a 
	geometric functor between rigidly-compactly generated categories admits exactly 3, 5, or infinitely many adjoints.
	It would be interesting to see if one can force the next layer of adjoints, just as 
	our Theorem~\ref{thm:main-thm-intro} forces the first layer of adjoints after a colocalization.
	It would also be interesting to know if there is a dual story whereby one can force GN-duality by performing some
	kind of (co)localization on the source category $\cat D$ rather than 
	on the target category $\cat C$.
\end{Rem}
\begin{Rem}
	There is significant motivation for relaxing our assumption that the tensor-triangulated categories $\cat C$ and $\cat D$
	are rigidly-compactly generated, as this would potentially enable the theory developed here (and in \cite{BalmerDellAmbrogioSanders16})
	to embrace several important examples in chromatic and motivic homotopy theory.
	For example, although the stable $\bbA^1$-homotopy category $\SHA\hspace{-.7ex}(S)$ is compactly generated, it
	is unlikely to be generated by rigid objects if the base scheme $S$ is positive dimensional 
	(cf.~Example~\ref{exas:examples}(e) and Example~\ref{exa:incomplete}).
	On the other hand, the $K(n)$-local stable homotopy category 
 (see \cite{HoveyStrickland99})
	is a prominent example of a tensor-triangulated category which is
	\emph{almost} rigidly-compactly generated: it is generated by a set of compact-rigid objects and the compact objects are rigid, but rigid objects need not be compact (\eg~ the unit object is not compact).
	Generalizing our treatment of the Wirthm\"{u}ller isomorphism (and the duality results of \cite{BalmerDellAmbrogioSanders16})
	to cover such examples of non-unital algebraic stable homotopy categories 
	may
	lead to connections with 
	Gross-Hopkins duality 
	(\`{a} la~\cite{DwyerGreenleesIyengar11})
	and Hopkins and Lurie's work on ``ambidexterity''
	in $K(n)$-local stable homotopy theory (see \cite[Sec.~4--5]{HopkinsLurie13pp}). 
	This will be pursued in future work.
	Since the category of colocal objects appearing
	in Theorem~\ref{thm:main-thm} is precisely such an example of a non-unital algebraic stable homotopy category,
	the author has some optimism
	that there may be fruit down this road.
\end{Rem}

\bibliographystyle{alpha}%
\bibliography{TG-articles}

\begin{thebibliography}{BvdB03}

\bibitem[AF91]{AvramovFoxby91}
Luchezar~L. Avramov and Hans-Bj{\o}rn Foxby.
\newblock Homological dimensions of unbounded complexes.
\newblock {\em J. Pure Appl. Algebra}, 71(2-3):129--155, 1991.

\bibitem[Bal05]{Balmer05a}
Paul Balmer.
\newblock The spectrum of prime ideals in tensor triangulated categories.
\newblock {\em J. Reine Angew. Math.}, 588:149--168, 2005.

\bibitem[Bal07]{Balmer07}
P.~Balmer.
\newblock Supports and filtrations in algebraic geometry and modular
  representation theory.
\newblock {\em Amer. J. Math.}, 129(5):1227--1250, 2007.

\bibitem[Bal10a]{Balmer10b}
Paul Balmer.
\newblock Spectra, spectra, spectra -- tensor triangular spectra versus
  {Z}ariski spectra of endomorphism rings.
\newblock {\em Algebr. Geom. Topol.}, 10(3):1521--1563, 2010.

\bibitem[Bal10b]{BalmerICM}
Paul Balmer.
\newblock Tensor triangular geometry.
\newblock In {\em International {C}ongress of {M}athematicians, Hyderabad
  (2010), {V}ol. {II}}, pages 85--112. Hindustan Book Agency, 2010.

\bibitem[BDS15]{BalmerDellAmbrogioSanders15}
Paul Balmer, Ivo Dell'Ambrogio, and Beren Sanders.
\newblock Restriction to finite-index subgroups as \'etale extensions in
  topology, {KK}-theory and geometry.
\newblock {\em Algebr. Geom. Topol.}, 15(5):3025--3047, 2015.

\bibitem[BDS16]{BalmerDellAmbrogioSanders16}
Paul Balmer, Ivo Dell'Ambrogio, and Beren Sanders.
\newblock Grothendieck--{N}eeman duality and the {W}irthm\"uller isomorphism.
\newblock {\em Compos. Math.}, 152(8):1740--1776, 2016.

\bibitem[BF11]{BalmerFavi11}
Paul Balmer and Giordano Favi.
\newblock Generalized tensor idempotents and the telescope conjecture.
\newblock {\em Proc. Lond. Math. Soc. (3)}, 102(6):1161--1185, 2011.

\bibitem[BGI71]{SGA6}
Pierre Berthelot, Alexander Grothendieck, and Luc Illusie, editors.
\newblock {\em SGA 6\,: Th\'eorie des intersections et th\'eor\`eme de
  {R}iemann-{R}och}, volume 225 of {\em Lecture Notes in Mathematics}.
\newblock Springer-Verlag, 1971.

\bibitem[BS13]{BrodmannSharp13}
M.~P. Brodmann and R.~Y. Sharp.
\newblock {\em Local cohomology}, volume 136 of {\em Cambridge Studies in
  Advanced Mathematics}.
\newblock Cambridge University Press, Cambridge, second edition, 2013.

\bibitem[BS17]{BalmerSanders17}
Paul Balmer and Beren Sanders.
\newblock The spectrum of the equivariant stable homotopy category of a finite
  group.
\newblock {\em Invent. Math.}, 208(1):283--326, 2017.

\bibitem[BvdB03]{BondalVAndenbergh03}
A.~Bondal and M.~van~den Bergh.
\newblock Generators and representability of functors in commutative and
  noncommutative geometry.
\newblock {\em Mosc. Math. J.}, 3(1):1--36, 258, 2003.

\bibitem[Con07]{Conrad07}
Brian Conrad.
\newblock Deligne's notes on {N}agata compactifications.
\newblock {\em J. Ramanujan Math. Soc.}, 22(3):205--257, 2007.

\bibitem[DGI06]{DwyerGreenleesIyengar06}
W.~Dwyer, J.~P.~C. Greenlees, and S.~Iyengar.
\newblock Finiteness in derived categories of local rings.
\newblock {\em Comment. Math. Helv.}, 81(2):383--432, 2006.

\bibitem[DGI11]{DwyerGreenleesIyengar11}
W.~G. Dwyer, J.~P.~C. Greenlees, and S.~B. Iyengar.
\newblock Gross-{H}opkins duality and the {G}orenstein condition.
\newblock {\em J. K-Theory}, 8(1):107--133, 2011.

\bibitem[Gro61]{EGA3}
A.~Grothendieck.
\newblock \'{E}l\'{e}ments de g\'{e}om\'{e}trie alg\'{e}brique. {III}.
  \'{E}tude cohomologique des faisceaux coh\'{e}rents. {I}.
\newblock {\em Inst. Hautes \'{E}tudes Sci. Publ. Math.}, (11):167, 1961.

\bibitem[Gro67]{EGA4part4}
A.~Grothendieck.
\newblock \'{E}l\'ements de g\'eom\'etrie alg\'ebrique. {IV}. \'{E}tude locale
  des sch\'emas et des morphismes de sch\'emas {IV}.
\newblock {\em Inst. Hautes \'Etudes Sci. Publ. Math.}, (32):361, 1967.

\bibitem[HL13]{HopkinsLurie13pp}
Michael Hopkins and Jacob Lurie.
\newblock {A}mbidexterity in ${K}(n)$-{L}ocal {S}table {H}omotopy {T}heory.
\newblock Preprint, available online, 2013.

\bibitem[HPS97]{HoveyPalmieriStrickland97}
Mark Hovey, John~H. Palmieri, and Neil~P. Strickland.
\newblock Axiomatic stable homotopy theory.
\newblock {\em Mem. Amer. Math. Soc.}, 128(610), 1997.

\bibitem[HS99]{HoveyStrickland99}
Mark Hovey and Neil~P. Strickland.
\newblock Morava {$K$}-theories and localisation.
\newblock {\em Mem. Amer. Math. Soc.}, 139(666):viii+100, 1999.

\bibitem[HS06]{HunekeSwanson06}
Craig Huneke and Irena Swanson.
\newblock {\em Integral closure of ideals, rings, and modules}, volume 336 of
  {\em London Mathematical Society Lecture Note Series}.
\newblock Cambridge University Press, Cambridge, 2006.

\bibitem[Ill83]{Illman83}
S{\"o}ren Illman.
\newblock The equivariant triangulation theorem for actions of compact {L}ie
  groups.
\newblock {\em Math. Ann.}, 262(4):487--501, 1983.

\bibitem[Lew00]{Lewis00}
L.~Gaunce Lewis, Jr.
\newblock Splitting theorems for certain equivariant spectra.
\newblock {\em Mem. Amer. Math. Soc.}, 144(686):x+89, 2000.

\bibitem[Lip09]{Lipman09}
Joseph Lipman.
\newblock Notes on derived functors and {G}rothendieck duality.
\newblock In {\em Foundations of {G}rothendieck duality for diagrams of
  schemes}, volume 1960 of {\em Lecture Notes in Math.}, pages 1--259.
  Springer, Berlin, 2009.

\bibitem[LMS86]{LewisMaySteinbergerMcClure86}
L.~G. Lewis, Jr., J.~P. May, and M.~Steinberger.
\newblock {\em Equivariant stable homotopy theory}, volume 1213 of {\em Lecture
  Notes in Mathematics}.
\newblock Springer-Verlag, Berlin, 1986.
\newblock With contributions by J. E. McClure.

\bibitem[LN07]{LipmanNeeman07}
Joseph Lipman and Amnon Neeman.
\newblock Quasi-perfect scheme-maps and boundedness of the twisted inverse
  image functor.
\newblock {\em Illinois J. Math.}, 51(1):209--236, 2007.

\bibitem[May03]{May03}
J.~P. May.
\newblock The {W}irthm\"uller isomorphism revisited.
\newblock {\em Theory Appl. Categ.}, 11:No. 5, 132--142, 2003.

\bibitem[Nee92]{Neeman92b}
Amnon Neeman.
\newblock The connection between the {$K$}-theory localization theorem of
  {T}homason, {T}robaugh and {Y}ao and the smashing subcategories of
  {B}ousfield and {R}avenel.
\newblock {\em Ann. Sci. \'Ecole Norm. Sup. (4)}, 25(5):547--566, 1992.

\bibitem[Nee01]{Neeman01}
Amnon Neeman.
\newblock {\em Triangulated categories}, volume 148 of {\em Annals of
  Mathematics Studies}.
\newblock Princeton University Press, 2001.

\bibitem[{Sta}18]{stacks-project}
The {Stacks Project Authors}.
\newblock {\it {S}tacks {P}roject}.
\newblock \url{http://stacks.math.columbia.edu}, 2018.

\bibitem[Tot18]{Totaro18}
Burt Totaro.
\newblock Adjoint functors on the derived category of motives.
\newblock {\em J. Inst. Math. Jussieu}, 17(3):489--507, 2018.

\bibitem[TT90]{ThomasonTrobaugh90}
R.~W. Thomason and T.~Trobaugh.
\newblock Higher algebraic {$K$}-theory of schemes and of derived categories.
\newblock In {\em The Grothendieck Festschrift, Vol.\ III}, volume~88 of {\em
  Progr. Math.}, pages 247--435. Birkh\"auser, Boston, MA, 1990.

\end{thebibliography}

\end{document}